\documentclass[onecolumn,11pt]{IEEEtran}

\usepackage{amsthm}
\usepackage{times,amssymb,amsmath,amsfonts,float,nicefrac,color,bbm,mathrsfs,caption,float}
\usepackage{algorithm,enumerate,multirow,caption,tikz,graphicx}
\usepackage[mathscr]{eucal}
\usepackage{sidecap}
\usepackage{algpseudocode}
\usepackage{verbatim}
\usepackage{epstopdf}
\usepackage{textcomp}
\usetikzlibrary{shapes,arrows}
\usepackage{stmaryrd}
\usepackage{mathabx}
\usepackage[noadjust]{cite}
\usepackage{booktabs}
\usepackage[normalem]{ulem}
\usepackage[top=1in,bottom=1in,left=1in,right=1in]{geometry}
\allowdisplaybreaks

\RequirePackage[OT1]{fontenc}
\RequirePackage[numbers]{natbib}
\RequirePackage[colorlinks,citecolor=blue,urlcolor=blue]{hyperref}

\newcommand{\RN}[1]{%
  \textup{\expandafter{\romannumeral#1}}%
}

\newcommand\remove[1]{}

\allowdisplaybreaks
\usepackage{enumitem}
\setlist[enumerate]{leftmargin=*}

\newtheorem{theorem}{Theorem}
\newtheorem{definition}{Definition}
\newtheorem{proposition}{Proposition}

\newtheorem{lemma}{Lemma}
\newtheorem{corollary}{Corollary}

\newtheorem{remark}{Remark}

\newcommand{\cG}{\mathcal{G}}

\newcommand{\cI}{\mathcal{I}}
\newcommand{\cJ}{\mathcal{J}}

\newcommand{\cN}{\mathcal{N}}

\DeclareMathOperator{\Binom}{Binom}
\DeclareMathOperator{\good}{good}

\DeclareMathOperator{\con}{con}
\DeclareMathOperator{\Var}{Var}

\DeclareMathOperator{\Cov}{Cov}

\DeclareMathOperator{\bad}{bad}
\DeclareMathOperator{\dist}{dist}
\DeclareMathOperator{\sign}{sign}
\DeclareMathOperator{\SSBM}{SSBM}
\DeclareMathOperator{\SIBM}{SIBM}

\usepackage{changepage}

\begin{document}

\title{Exact recovery and sharp thresholds of Stochastic Ising Block Model}

\author{Min Ye}

\maketitle
{\renewcommand{\thefootnote}{}\footnotetext{

\vspace{-.2in}
 
\noindent\rule{1.5in}{.4pt}

Min Ye is with the Data Science and Information Technology Research Center, Tsinghua-Berkeley Shenzhen Institute, Tsinghua Shenzhen International Graduate School, Shenzhen 518055, China. Email: yeemmi@gmail.com
}}

\renewcommand{\thefootnote}{\arabic{footnote}}
\setcounter{footnote}{0}

\begin{abstract}
The stochastic block model (SBM) is a random graph model in which the edges are generated according to the underlying cluster structure on the vertices. The (ferromagnetic) Ising model, on the other hand, assigns $\pm 1$ labels to vertices according to an underlying graph structure in a way that if two vertices are connected in the graph then they are more likely to be assigned the same label. In SBM, one aims to recover the underlying clusters from the graph structure while in Ising model, an extensively-studied problem is to recover the underlying graph structure based on i.i.d. samples (labelings of the vertices).

In this paper, we propose a natural composition of SBM and the Ising model, which we call the Stochastic Ising Block Model (SIBM). In SIBM, we take SBM in its simplest form, where $n$ vertices are divided into two equal-sized clusters and the edges are connected independently with probability $p$ within clusters and $q$ across clusters. Then we use the graph $G$ generated by the SBM as the underlying graph of the Ising model and draw $m$ i.i.d. samples from it. The objective is to exactly recover the two clusters in SBM from the samples generated by the Ising model, without observing the graph $G$. As the main result of this paper, we establish a sharp threshold $m^\ast$ on the sample complexity of this exact recovery problem in a properly chosen regime, where $m^\ast$ can be calculated from the parameters of SIBM. We show that when $m\ge m^\ast$, one can recover the clusters from $m$ samples in $O(n)$ time as the number of vertices $n$ goes to infinity. When $m<m^\ast$, we further show that for almost all choices of parameters of SIBM, the success probability of any recovery algorithms approaches $0$ as $n\to\infty$.
\end{abstract}

\thispagestyle{empty}

\clearpage
\tableofcontents

\thispagestyle{empty}

\clearpage
\setcounter{page}{1}

\section{Introduction} \label{sect:intro}
The stochastic block model (SBM) is a generative model for random graphs, where the vertices are partitioned into several communities/clusters, and the edges are added independently in a way that depends on the community membership of its two endpoints. An extensively-studied problem in SBM is the exact recovery problem \cite{Bui87,Boppana87, Dyer89,Snijders97,Condon01,Mcsherry01,Bickel09,Choi12,Hajek16,Chen16,Vu18}, in which one aims to recover the exact underlying community structure from the graph topology, i.e., the edge connections in the graph.
Sharp thresholds for exact recovery were established in terms of the parameters of SBM, starting from the simplest case of two-community symmetric SBM \cite{ABH16,MNS16} to the most general case \cite{Abbe15}. See \cite{Abbe17} for a recent survey on this topic, where other interesting and important problems in SBM are also discussed.

Ising model, originally introduced in the context of statistical physics \cite{Ising25}, consists of binary random variables $\sigma_1,\dots,\sigma_n\in\{\pm 1\}$ whose pairwise dependency is modeled by an underlying dependency graph $G$.
Despite its simplicity, Ising
model has been used in a wide range of applications including finance, social networks, computer vision, biology, and signal processing.
A recent line of research on Ising model concerns estimating the dependency graph $G$ from independent samples of the random vector $\sigma=(\sigma_1,\dots,\sigma_n)\in\{\pm 1\}^n$ \cite{Bresler08,Ravikumar10,Santhanam12,Bresler15,Vuffray16,Hamilton17,Klivans17,Wu19}.
This problem is closely related to inferring social network structures, and a concrete example was presented in \cite{Banerjee08}.
The authors of \cite{Banerjee08} used Ising model to infer the political affinities among the U.S. senators from their voting records. In this example, $\sigma_i$ may represent the vote of U.S. senator $i$ on a random bill, and the dependency graph $G$ may depict the political affinities or network structure among the senators.

In this paper, we propose a natural composition of SBM and the Ising model, which we call the Stochastic Ising Block Model (SIBM). 
First we use SBM to generate a graph $G=([n],E(G))$ with vertex set $[n]$ and edge set $E(G)$ based on an (unknown) partition of the vertex set $[n]$.
Next we use $G$ as the underlying dependency graph of the Ising model and draw $m$ i.i.d. samples from it.
The objective is to exactly recover the partition of the vertex set in SBM from the samples generated by the Ising model, without observing the graph $G$.

In the above SIBM, we take SBM in its simplest form, where $n$ vertices are randomly divided into two equal-sized clusters and the edges are connected independently with probability $p$ within clusters and $q$ across clusters. 
We focus on the asymptotic regime of $p=a\log(n)/n$ and $q=b\log(n)/n$ for fixed $a>b>0$ and growing $n$.
In this regime, if the vertices $i$ and $j$ are connected in $G$, then the posterior probability of them belonging to the same cluster is $\frac{a}{a+b}>\frac{1}{2}$. On the other hand, if $i$ and $j$ are not connected, then the posterior probability of them belonging to the same cluster is 
\begin{equation}  \label{eq:rep}
\frac{1-p}{1-p+1-q}=\frac{1}{2}\big( 1-\frac{a-b}{2}\log(n)/n \big) ,
\end{equation}
implying a slight tendency towards being in different clusters. 
The Ising model that we use in the SIBM is a modification of a commonly used one based on this observation.
First recall that a commonly used Ising model on the graph $G$ is a probability distribution on the configurations $\sigma=(\sigma_1,\dots,\sigma_n)\in\{\pm 1\}^n$ such that\footnote{We use $\sigma$ to denote the random vector, and we usually use $\bar{\sigma}$ to denote the realization of $\sigma$. The subscript in $P_{\sigma|G}$ indicates that the distribution is determined by the graph $G$.}
\begin{equation} \label{eq:bv}
P_{\sigma|G}(\sigma=\bar{\sigma})=\frac{1}{Z_G(\beta)}
\exp\Big(\beta\sum_{\{i,j\}\in E(G)} \bar{\sigma}_i \bar{\sigma}_j \Big) ,
\end{equation}
where the parameter $\beta>0$ is called the inverse temperature and $Z_G(\beta)$ is the normalizing constant. The (random) vector $\sigma$ induces a (random) partition on the vertex set $[n]$ according to the sign of each coordinate. Since the objective in SIBM is to recover the original partition in SBM based on independent samples of $\sigma$, we would hope that the partition induced by $\sigma$ is close to the original partition in SBM. However, one can show that the distribution given in \eqref{eq:bv} is concentrated on the neighborhood of $\pm \mathbf{1}_n$, where $\mathbf{1}_n$ is the all-one vector of length $n$. Thus the samples of $\sigma$ provide little information about the partition in SBM. This happens because the distribution \eqref{eq:bv} does not reflect the small repulsive effect of non-edges in SBM, i.e., non-edge between vertices $i$ and $j$ implies that they have a slight tendency to be in the different clusters in SBM; see \eqref{eq:rep}. To overcome this issue, we use the following modification of the standard Ising model in SIBM:
\begin{equation}  \label{eq:im}
P_{\sigma|G}(\sigma=\bar{\sigma})=\frac{1}{Z_G(\alpha,\beta)}
\exp\Big(\beta\sum_{\{i,j\}\in E(G)} \bar{\sigma}_i \bar{\sigma}_j
-\frac{\alpha\log(n)}{n} \sum_{\{i,j\}\notin E(G)} \bar{\sigma}_i \bar{\sigma}_j \Big) ,
\end{equation}
where we add a new parameter $\alpha>0$, and $Z_G(\alpha,\beta)$ is again the normalizing constant. The small negative coefficient $-\alpha\log(n)/n$ in front of all non-edge pairs $\bar{\sigma}_i\bar{\sigma}_j$ makes $\sigma_i$ and $\sigma_j$ slightly more likely to take different signs, and so vertices $i$ and $j$ are slightly more likely to be in different clusters in the partition induced by $\sigma$. We choose the order of this coefficient to be $\Theta(\log(n)/n)$ in accordance with the calculation in \eqref{eq:rep} for SBM.

Before presenting the main results, let us first give a motivating example to validate the above definition of SIBM.
Think of the residents in a city. Half of them are Democrats and the other half are Republicans. Residents in the same political party are more likely to get to know each other and become friends, while this is less likely to happen for residents in different political parties. 
The two-community SBM defined above is a commonly used model of the friendship/social network between the residents.
Now let's say we want to find a partition of the residents based on their political party membership, but it is not possible to observe the whole social network among the residents due to privacy and complexity issues. One thing we can do is to take multiple political polls among the residents. It is natural to assume that if two residents are friends, i.e., if they are connected in the social network, then they tend to share similar opinions and give the same response to the polls; otherwise there is a slight tendency that they would give different answers. Therefore, a valid way to model the dependency of the residents' responses on their social network topology is the probability distribution in \eqref{eq:im}, where we view the random vector $\sigma$ as the residents' responses to a random poll. Finally, we hope to reveal the partition of the residents based on their responses to multiple independent polls, and this is analogous to the objective in SIBM defined above.

In this paper, our main focus is the optimal sample complexity for the exact recovery of the underlying partition in SBM. Below we use a vector $X=(X_1,\dots,X_n)\in\{\pm 1\}^n$ to represent this underlying partition, i.e., $X_i=X_j$ if vertices $i$
and $j$ are in the same cluster; otherwise $X_i=-X_j$. We also use $\sigma^{(1)},\dots,\sigma^{(m)} \in\{\pm 1\}^n$ to denote $m$ independent samples drawn from the distribution \eqref{eq:im} with respect to the {\em same} dependency graph $G$.
It is clear that the ground truth $X$, the graph $G$ and the samples form a Markov chain $X\to G\to \{\sigma^{(1)},\dots,\sigma^{(m)}\}$. Therefore, a necessary condition for $X$ to be recoverable from the samples (up to a global sign) is that $X$ must be recoverable from $G$. 
It is well known that the necessary and sufficient condition for the latter is  $\sqrt{a}-\sqrt{b}\ge\sqrt{2}$ \cite{MNS16,ABH16}, where $a$ and $b$ are parameters of the SBM defined above.
Under this condition, we prove the following results: If $\alpha<b\beta$, then all the samples are in the neighborhood of $\pm \mathbf{1}_n$, and one needs at least $\Omega(\log^{1/4}(n))$ samples to recover $X$. On the other hand, if $\alpha>b\beta$, then all the samples are in the neighborhood of $\pm X$, and $\Theta(1)$ samples suffice for the exact recovery.
As the main result of this paper, we establish a sharp threshold $m^\ast$ on the number of samples needed for exact recovery when $\alpha>b\beta$, where $m^\ast$ can be calculated from the parameters $a,b$ and $\beta$. We show that when $m\ge m^\ast$, one can recover $X$ (up to a global sign) from $m$ samples in $O(n)$ time; when $m<m^\ast$, we further show that for almost all choices of parameters $a,b$ and $\beta$, the success probability of all recovery algorithms approaches $0$ as $n\to\infty$.


{\bf Related works:} A static Ising block model was proposed in \cite{Berthet19}, and it is rather different from the SIBM proposed in this paper. More precisely, in SIBM we first use the ground truth $X$ to produce a graph $G$ and then use the graph $G$ to produce samples through an Ising model. In contrast, the samples in \cite{Berthet19} are produced directly from $X$ through a (different) Ising model. More specifically, given $X$, the samples in \cite{Berthet19} have the following distribution:
$$
P_{\sigma|X}(\sigma=\bar{\sigma})=\frac{1}{Z_{\alpha,\beta}}
\exp\Big(\frac{\beta}{2n} \sum_{(i,j):X_i=X_j} \bar{\sigma}_i \bar{\sigma}_j 
+ \frac{\alpha}{2n} \sum_{(i,j):X_i\neq X_j} \bar{\sigma}_i \bar{\sigma}_j \Big) ,
$$
where $Z_{\alpha,\beta}$ is the normalizing constant. The analysis of this model is rather different from ours, and the number of samples needed to recover $\pm X$ in this static model is $\Omega(\log(n))$, also quite different from our results.

As a final remark, we note that
the composition of SBM and Ising model has been considered in \cite{Peixoto19}. However, there are two major differences between \cite{Peixoto19} and this paper: First, the objective in \cite{Peixoto19} is to reconstruct the graph $G$ and the cluster structure $\pm X$ simultaneously while in this paper we only aim to recover $\pm X$. Second, the main focus in \cite{Peixoto19} is to propose some heuristic methods/algorithms and present their performance on synthetic and empirical datasets while in this paper we focus on rigorous proofs of performance guarantees and theoretical limits of the recovery algorithms.

{\bf Organization of the paper:}
In the next section, we formally define the new model SIBM and state the main results. In Section~\ref{sect:sketch}, we give a sketch of the proof which illustrates the main ideas. In Section~\ref{sect:aln}--\ref{sect:converse} we provide the complete proof. Finally, we conclude the paper in Section~\ref{sect:future} with some future directions.

\section{Problem setup and main results} \label{s:Preliminaries}
We first recall the definition of Symmetric Stochastic Block Model (SSBM) with two communities and the definition of Ising model.
\begin{definition}[SSBM with two communities] \label{def:SSBM}
Let $n$ be a positive even integer and let $p,q\in[0,1]$ be two real numbers. Let $X=(X_1,\dots,X_n)\in\{\pm 1\}^n$, and let $G=([n],E(G))$ be an undirected graph with vertex set $[n]$ and edge set $E(G)$. The pair $(X,G)$ is drawn under $\SSBM(n,p,q)$ if 

\noindent
(i) $X$ is drawn uniformly from the set of balanced partitions $\{(x_1,\dots,x_n)\in\{\pm 1\}^n:x_1+\dots+x_n=0\}$;

\noindent
(ii) the vertices $i$ and $j$ in $G$ are connected with probability $p$ if $X_i=X_j$ and with probability $q$ if $X_i=-X_j$, independently of other pairs of vertices.
\end{definition}
 
Note that for every $x=(x_1,\dots,x_n)\in\{\pm 1\}^n$ such that $x_1+\dots+x_n=0$,
$x$ and $-x$ correspond to the same balanced partition, and the conditional distribution $P(G|X=x)$ is the same as $P(G|X=-x)$ in the above definition. Therefore, we can only hope to recover $X$ from $G$ up to a global sign.
In this paper, we focus on the regime of $p=a\log(n)/n$ and $q=b\log(n)/n$, where $a>b> 0$ are constants. In this regime, it is well known that exact recovery of $X$ (up to a global sign) from $G$ is possible if and only if $\sqrt{a}-\sqrt{b}\ge \sqrt{2}$  \cite{ABH16,MNS16}.
 
Given a partition/labeling $X$ on $n$ vertices, the SBM specifies how to generate a random graph on these $n$ vertices according to the labeling. In some sense, Ising model works in the opposite direction, i.e., given a graph $G=([n],E(G))$, Ising model defines a probability distribution on all possible labelings $\{\pm 1\}^n$ of these $n$ vertices.


 \begin{definition}[Ising model]
Define an Ising model on a graph $G=([n],E(G))$ with parameters $\alpha,\beta>0$ as the probability distribution on the configurations $\sigma\in\{\pm 1\}^n$ such that\footnote{When there is only one sample, we usually denote it as $\bar{\sigma}$. When there are $m$ (independent) samples, we usually denote them as $\sigma^{(1)},\dots,\sigma^{(m)}$.}
\begin{equation} \label{eq:isingma}
P_{\sigma|G}(\sigma=\bar{\sigma})=\frac{1}{Z_G(\alpha,\beta)}
\exp\Big(\beta\sum_{\{i,j\}\in E(G)} \bar{\sigma}_i \bar{\sigma}_j
-\frac{\alpha\log(n)}{n} \sum_{\{i,j\}\notin E(G)} \bar{\sigma}_i \bar{\sigma}_j \Big) ,
\end{equation}
where the subscript in $P_{\sigma|G}$ indicates that the distribution depends on $G$, and 
\begin{equation}  \label{eq:zg}
Z_G(\alpha,\beta)=\sum_{\bar{\sigma} \in \{\pm 1\}^n} \exp\Big(\beta\sum_{\{i,j\}\in E(G)} \bar{\sigma}_i \bar{\sigma}_j
-\frac{\alpha\log(n)}{n} \sum_{\{i,j\}\notin E(G)} \bar{\sigma}_i \bar{\sigma}_j \Big) 
\end{equation}
is the normalizing constant.
\end{definition}

\begin{remark}
A more commonly used Ising model on a graph $G$ is obtained by setting $\alpha=0$ in \eqref{eq:isingma},
thereby only involving one parameter $\beta$. See the discussion in Section~\ref{sect:intro} on why we use the modified Ising model in SIBM.
\end{remark}

By definition we always have $P_{\sigma|G}(\sigma=\bar{\sigma})=P_{\sigma|G}(\sigma=-\bar{\sigma})$ in the Ising model. Next we present our new model, the Stochastic Ising Block Model (SIBM), which can be viewed as a natural composition of the SSBM and the Ising model. In SIBM, we first draw a pair $(X,G)$ under $\SSBM(n,p,q)$.  Then we draw $m$ independent samples $\{\sigma^{(1)},\dots,\sigma^{(m)}\}$ from the Ising model on the graph $G$, where $\sigma^{(u)}\in\{\pm 1\}^n$ for all $u\in[m]$.

\begin{definition}[Stochastic Ising Block Model]
Let $n$ and $m$ be positive integers such that $n$ is even. Let $p,q\in[0,1]$ be two real numbers and let $\alpha,\beta>0$. The triple $(X,G,\{\sigma^{(1)},\dots,\sigma^{(m)}\})$ is drawn under $\SIBM(n,p,q,\alpha,\beta,m)$ if

\noindent
(i) the pair $(X,G)$ is drawn under $\SSBM(n,p,q)$;

\noindent
(ii) for every $i\in[m]$, each sample $\sigma^{(i)}=(\sigma_1^{(i)},\dots,\sigma_n^{(i)}) \in\{\pm 1\}^n$ is drawn independently according to the distribution \eqref{eq:isingma}.
\end{definition}

Notice that we only draw the graph $G$ once in SIBM, and the samples $\{\sigma^{(1)},\dots,\sigma^{(m)}\}$ are drawn independently from the Ising model on the {\em same} graph $G$.
Our objective is to recover the underlying partition (or the ground truth) $\pm X$ from the samples $\{\sigma^{(1)},\dots,\sigma^{(m)}\}$, and we would like to use the smallest possible number of samples to guarantee the exact recovery of $X$ up to a global sign.
Below we use the notation $P_{\SIBM}(A):=E_G[P_{\sigma|G}(A)]$ for an event $A$, where the expectation $E_G$ is taken with respect to the distribution given by SSBM. In other words, $P_{\sigma|G}$ is the conditional distribution of  $\sigma$ given a fixed $G$ while $P_{\SIBM}$ is the joint distribution of both $\sigma$ and $G$.
By definition, we have $P_{\SIBM}(\sigma=\bar{\sigma})=P_{\SIBM}(\sigma=-\bar{\sigma})$ for all $\bar{\sigma}\in\{\pm 1\}^n$.

\begin{definition}[Exact recovery in SIBM]
Let $(X,G,\{\sigma^{(1)},\dots,\sigma^{(m)}\}) \sim \SIBM(n,p,q,\alpha,\beta,m)$.
We say that exact recovery is solvable for $\SIBM(n,p,q,\alpha,\beta,m)$ if there is an algorithm that takes $\{\sigma^{(1)},\dots,\sigma^{(m)}\}$ as inputs and outputs $\hat{X}=\hat{X}(\{\sigma^{(1)},\dots,\sigma^{(m)}\})$ such that
$$
P_{\SIBM}(\hat{X}=X \text{~or~} \hat{X}=-X) \to 1
\text{~~~as~} n\to\infty ,
$$
and we call $P_{\SIBM}(\hat{X}=X \text{~or~} \hat{X}=-X)$ the success probability of the recovery/decoding algorithm.
\end{definition}

As mentioned above, we consider the regime of $p=a\log(n)/n$ and $q=b\log(n)/n$, where $a>b> 0$ are constants. By definition, the ground truth $X$, the graph $G$ and the samples $\{\sigma^{(1)},\dots,\sigma^{(m)}\}$ form a Markov chain $X\to G\to \{\sigma^{(1)},\dots,\sigma^{(m)}\}$. Therefore, if we cannot recover $X$ from $G$, then there is no hope to recover $X$ from $\{\sigma^{(1)},\dots,\sigma^{(m)}\}$. Thus a necessary condition for the exact recovery in SIBM is $\sqrt{a}-\sqrt{b}\ge \sqrt{2}$, and we will limit ourselves to this case throughout the paper.

\vspace*{.1in}
\noindent{\bf Main Problem:} \emph{For any $a,b> 0$ such that $\sqrt{a}-\sqrt{b}> \sqrt{2}$ and any $\alpha,\beta>0$, what is the smallest sample size $m^\ast$ such that exact recovery is solvable for $\SIBM(n,a\log(n)/n, b\log(n)/n,\alpha,\beta,m^\ast)$?}
  
\vspace*{.1in}  It is this optimal sample size problem that we address---and resolve---in this paper. 
Our main results read as follows.

\begin{theorem} \label{thm:wt1}
For any $a,b> 0$ such that $\sqrt{a}-\sqrt{b}> \sqrt{2}$ and any $\alpha,\beta>0$, let
\begin{equation} \label{eq:defstar}
\beta^\ast := \frac{1}{2}
\log\frac{a+b-2-\sqrt{(a+b-2)^2-4ab}}{2 b} \text{~~and~~}
m^\ast := 2 \Big\lfloor \frac{\beta^\ast}{\beta} \Big\rfloor +1 .
\end{equation}
{\bf Case (i) when $\alpha>b\beta$}: If $m\ge m^\ast$, then exact recovery is solvable in $O(n)$ time for $\SIBM(n,a\log(n)/n, \linebreak[4] b\log(n)/n,\alpha,\beta,m)$, and the recovery algorithm does not require knowledge of the parameters of SIBM.
If $\beta^\ast/\beta$ is not an integer and $m<m^\ast$, then the success probability of all recovery algorithms approaches $0$ as $n\to\infty$. If $\beta^\ast/\beta$ is an integer and $m<m^\ast-2$, then the success probability of all recovery algorithms approaches $0$ as $n\to\infty$.
{\bf Case (ii) when $\alpha<b\beta$}: Exact recovery of $\SIBM(n,a\log(n)/n, b\log(n)/n,\alpha,\beta,m)$ is not solvable for any $m=O(\log^{1/4}(n))$, and in particular, it is not solvable for any constant $m$ that does not grow with $n$.
\end{theorem}
Note that the condition $\sqrt{a}-\sqrt{b} > \sqrt{2}$ guarantees that the term $\sqrt{(a+b-2)^2-4ab}$ in the definition of $\beta^\ast$ is a  real number.
When $\alpha>b\beta$ and $\beta^\ast/\beta$ is not an integer,
the above theorem establishes a sharp recovery threshold $m^\ast$ on the number of samples. It is worth mentioning that the threshold $m^\ast$ does {\em not} depend on the value of the parameter $\alpha$, as long as $\alpha$ satisfies $\alpha>b\beta$.
Below we present an equivalent characterization of the recovery threshold in terms of $\beta$.
\begin{theorem} \label{thm:wt2}
Let $a,b,\alpha,\beta> 0$ be constants satisfying that $\sqrt{a}-\sqrt{b} > \sqrt{2}$ and $\alpha>b\beta$. 
Let 
$$
(X,G,\{\sigma^{(1)},\dots,\sigma^{(m)}\}) \sim \SIBM(n,a\log(n)/n, b\log(n)/n,\alpha,\beta,m).
$$
If $\lfloor \frac{m+1}{2} \rfloor \beta>\beta^\ast$, then there is an algorithm that recovers $X$ from the samples in $O(n)$ time with success probability $1-o(1)$, and this recovery algorithm does not require knowledge of the parameters of SIBM. If $\lfloor \frac{m+1}{2} \rfloor \beta <\beta^\ast$, then the success probability of any recovery algorithm is $o(1)$. 
\end{theorem}
Note that $\lfloor \frac{m+1}{2} \rfloor \beta>\beta^\ast$ if and only if $m\ge 2 \Big\lfloor \frac{\beta^\ast}{\beta} \Big\rfloor +1$, so Theorem~\ref{thm:wt1} and Theorem~\ref{thm:wt2} give the same threshold\footnote{We give a proof of the equivalence between the two inequalities: $\lfloor \frac{m+1}{2} \rfloor \beta>\beta^\ast$ implies that $\frac{\beta^\ast}{\beta}<\lfloor \frac{m+1}{2} \rfloor$. The smallest integer that is larger than $\frac{\beta^\ast}{\beta}$ is $\lfloor \frac{\beta^\ast}{\beta}\rfloor +1$, so $\lfloor \frac{\beta^\ast}{\beta} \rfloor +1 \le \lfloor \frac{m+1}{2} \rfloor\le \frac{m+1}{2}$, and thus $m\ge 2 \Big\lfloor \frac{\beta^\ast}{\beta} \Big\rfloor +1$. Now assume $m\ge 2 \Big\lfloor \frac{\beta^\ast}{\beta} \Big\rfloor +1$, then $\frac{m-1}{2} \ge \lfloor \frac{\beta^\ast}{\beta} \rfloor$. Since the right hand side is an integer, we have $\lfloor \frac{m+1}{2} \rfloor = \lfloor \frac{m-1}{2} \rfloor +1  \ge \lfloor \frac{\beta^\ast}{\beta} \rfloor +1 >\frac{\beta^\ast}{\beta}$.}.
Apart from the results on the sharp recovery threshold, we also prove a structural result on the distance between the samples and the ground truth $X$.
For $\sigma,X\in\{\pm 1\}^n$, we define
$$
\dist(\sigma, X)
:=|\{i\in[n]:\sigma_i\neq X_i\}| 
\quad \text{and} \quad
\dist(\sigma,\pm X)
:=\min(\dist(\sigma, X),\dist(\sigma, -X)) .
$$
\begin{theorem}  \label{thm:wt3}
Let $a,b,\alpha,\beta> 0$ be constants satisfying that $\sqrt{a}-\sqrt{b} > \sqrt{2}$ and $\alpha>b\beta$. Let $m$ be a constant integer that does not grow with $n$.
Let 
$$
(X,G,\{\sigma^{(1)},\dots,\sigma^{(m)}\}) \sim \SIBM(n,a\log(n)/n, b\log(n)/n,\alpha,\beta,m).
$$
Define $g(\beta)  := \frac{b e^{2\beta}+a e^{-2\beta}}{2}-\frac{a+b}{2}+1$.
If $\beta>\beta^\ast$, then
$$
P_{\SIBM}(\sigma^{(i)}=\pm X \text{~for all~} i\in[m]) = 1-o(1).
$$
If $\beta\le \beta^\ast$, then
$$
P_{\SIBM}(\dist(\sigma^{(i)},\pm X)= \Theta(n^{g(\beta)}) \text{~for all~} i\in[m]) = 1-o(1) .
$$
\end{theorem}
One can show that (i) $g(\beta)$ is a strictly decreasing function in $[0,\beta^\ast]$, (ii) $g(0)=1$ and (iii) $g(\beta^\ast)=0$. Therefore, $0<g(\beta)<1$ when $0<\beta<\beta^\ast$. Thus Theorem~\ref{thm:wt3} implies that for all $\beta\le \beta^\ast$, $\dist(\sigma^{(i)},\pm X)=o(n)$ for all $i\in[m]$. In particular, for $\beta = \beta^\ast$, $\dist(\sigma^{(i)},\pm X)=\Theta(1)$ for all $i\in[m]$.

\section{Sketch of the proof}
\label{sect:sketch}

In this section, we illustrate the main ideas and explain the important steps in the proof of the main results. The complete proofs are given in Section~\ref{sect:aln}--\ref{sect:converse}.
As the first step, we prove that for $a>b>0$, if $\alpha>b\beta$, then all the samples are centered around $\pm X$ with probability $1-o(1)$; if $\alpha<b\beta$, then all the samples are centered around $\pm \mathbf{1}_n$.
We use the concentration results for adjacency matrices of random graphs with independent edges to prove this. Let $A=A(G)$ be the adjacency matrix of the graph $G$. Then \eqref{eq:isingma} can be written as
\begin{align*}
 P_{\sigma|G}(\sigma=\bar{\sigma})  
=  \frac{1}{Z_G(\alpha,\beta)}
\exp\Big( \frac{1}{2} \bar{\sigma} \Big( \big(\beta+\frac{\alpha\log(n)}{n} \big) A
-\frac{\alpha\log(n)}{n} (J_n-I_n) \Big) \bar{\sigma}^T 
\Big)  ,
\end{align*}
where $J_n$ is the all one matrix and $I_n$ is the identity matrix, both of size $n\times n$.
Define a matrix
$
M:= \big(\beta+\frac{\alpha\log(n)}{n} \big) E[A|X]
-\frac{\alpha\log(n)}{n} (J_n-I_n).
$
Then we can further write \eqref{eq:isingma}
as
$$
P_{\sigma|G}(\sigma=\bar{\sigma})
= \frac{1}{Z_G(\alpha,\beta)}
\exp\Big( \frac{1}{2} \bar{\sigma}  M \bar{\sigma}^T + \frac{1}{2} \big(\beta+\frac{\alpha\log(n)}{n} \big) \bar{\sigma}  (A-E[A|X])
  \bar{\sigma}^T 
\Big)  .
$$
One can show that if $\alpha>b\beta$, then $\bar{\sigma}  M \bar{\sigma}^T$ has two maximizers $\bar{\sigma}=\pm X$, and if $\alpha<b\beta$, then $\bar{\sigma}  M \bar{\sigma}^T$ has two maximizers $\bar{\sigma}=\pm \mathbf{1}_n$.
Moreover, \cite{Lei15} and \cite{Hajek16} proved the concentration of $A$ around its expectation $E[A|X]$ in the sense that
$\|A-E[A|X]\| = O(\sqrt{\log(n)})$ with probability $1-o(1)$, so the error term above is bounded by
$$
\left|\frac{1}{2} \big(\beta+\frac{\alpha\log(n)}{n} \big) \bar{\sigma}  (A-E[A|X])
  \bar{\sigma}^T \right| = O \big( n \sqrt{\log(n)} \big)
  \text{~~for all~} \bar{\sigma}\in\{\pm 1\}^n  .
$$
This allows us to prove that in both cases (no matter $\alpha>b\beta$ or $\alpha<b\beta$), the Hamming distance between the samples and the maximizers of $\bar{\sigma}  M \bar{\sigma}^T$ is upper bounded by $2n/\log^{1/3}(n)=o(n)$.
More precisely, if $\alpha>b\beta$, then $\dist(\sigma^{(i)},\pm X)< 2n/\log^{1/3}(n)$ for all $i\in[m]$ with probability $1-o(1)$; if $\alpha<b\beta$, then $\dist(\sigma^{(i)},\pm \mathbf{1}_n)< 2n/\log^{1/3}(n)$ for all $i\in[m]$ with probability $1-o(1)$; see Proposition~\ref{prop:1} for a rigorous proof.
In the latter case, each sample only takes $\sum_{j=0}^{2n/\log^{1/3}(n)}\binom{n}{j}$ values, so each sample contains at most $\log_2(\sum_{j=0}^{2n/\log^{1/3}(n)}\binom{n}{j})=O(\frac{\log\log(n)}{\log^{1/3}(n)} n)$ bits of information about $X$. On the other hand, $X$ itself is uniformly distributed over a set of $\binom{n}{n/2}$ vectors, so one needs at least $\log_2\binom{n}{n/2}=\Theta(n)$ bits of information to recover $X$. Thus if $\alpha<b\beta$, then exact recovery of $X$ requires at least $\Omega(\frac{\log^{1/3}(n)}{\log\log(n)})\ge \Omega(\log^{1/4}(n))$ samples; see Proposition~\ref{prop:ab} for a rigorous proof.

For the rest of this section, we will focus on the case $\alpha>b\beta$ and establish the sharp threshold on the sample complexity. We first analyze the typical behavior of one sample and explain how to prove Theorem~\ref{thm:wt3}. Then we use the method developed for the one sample case to analyze the distribution of multiple samples, which allows us to prove Theorem~\ref{thm:wt2}. Note that Theorem~\ref{thm:wt1} follows directly from Theorem~\ref{thm:wt2}.

\subsection{Why is $\beta^\ast$ the threshold?} \label{sect:why}

Let us analyze the one sample case, i.e., we take $m=1$.
Theorem~\ref{thm:wt3} implies that $\beta^\ast$ is a sharp threshold for the event $\{\sigma=\pm X\}$, i.e., $P_{\SIBM}(\sigma=\pm X)=1-o(1)$ if $\beta$ is above this threshold and $P_{\SIBM}(\sigma=\pm X)=o(1)$ if $\beta$ is below this threshold.
We already know that
$
P_{\SIBM} \big(\dist(\sigma,\pm X)< 2n/\log^{1/3}(n) \big) = 1- o(1) .
$
Therefore the following three statements are equivalent:
\begin{enumerate}[label=(\arabic*)]
\item $P_{\SIBM}(\sigma=\pm X)$ has a sharp transitions from $0$ to $1$ at $\beta^\ast$.
\item $P_{\SIBM} \big( 1\le \dist(\sigma,\pm X)< 2n/\log^{1/3}(n) \big)$ has a sharp transitions from $1$ to $0$ at $\beta^\ast$.
\item $\frac{P_{\SIBM} ( 1\le \dist(\sigma, X)< 2n/\log^{1/3}(n) )}{P_{\SIBM}(\sigma= X)}$ has a sharp transition from $\infty$ (or $\omega(1)$) to $0$ at $\beta^\ast$.
\end{enumerate}
Statements (2) and (3) are equivalent because $P_{\SIBM}(\sigma=\bar{\sigma})=P_{\SIBM}(\sigma=-\bar{\sigma})$ for all $\bar{\sigma}\in\{\pm 1\}^n$.
We will show that the above three statements are further equivalent to
\begin{enumerate}[label=(\arabic*)]
\setcounter{enumi}{3}
  \item  $\frac{P_{\SIBM} ( \dist(\sigma, X) = 1 )}{P_{\SIBM}(\sigma= X)}$ has a sharp transitions from $\infty$ (or $\omega(1)$) to $0$ at $\beta^\ast$.
\end{enumerate}
We first prove (4) and then show that it is equivalent to statement (3).
Instead of analyzing $\frac{P_{\SIBM} ( \dist(\sigma, X) = 1 )}{P_{\SIBM}(\sigma= X)}$, we analyze $\frac{P_{\sigma|G} ( \dist(\sigma, X) = 1 )}{P_{\sigma|G}(\sigma= X)}$ for a typical graph $G$.
To that end, we introduce some notation: 
For $\cI\subseteq[n]$, define $X^{(\sim\cI)}$ as the vector obtained by flipping the coordinates in $\cI$ while keeping all the other coordinates to be the same as $X$, i.e., $X_i^{(\sim\cI)}=-X_i$ for all $i\in\cI$ and $X_i^{(\sim\cI)}=X_i$ for all $i\notin\cI$. 
When $\cI$ only contains one element, e.g., $\cI=\{i\}$, we write $X^{(\sim i)}$ instead of $X^{(\sim\{i\})}$.
Then,
$$
\frac{P_{\sigma|G} ( \dist(\sigma, X) = 1 )}{P_{\sigma|G}(\sigma= X)}
=\sum_{i=1}^n \frac{P_{\sigma|G} ( \sigma= X^{(\sim i)} )} {P_{\sigma|G}(\sigma= X)} .
$$
Given the ground truth $X$, a graph $G$ and a vertex $i\in[n]$, define
\begin{align*}
& A_i=A_i(G):=|\{j\in[n]\setminus\{i\}:\{i,j\}\in E(G), X_j=X_i\}| , \\
& B_i=B_i(G):=|\{j\in[n]\setminus\{i\}:\{i,j\}\in E(G), X_j=-X_i\}| .
\end{align*}
Then by \eqref{eq:isingma}, we have
\begin{align*}
\frac{P_{\sigma|G}(\sigma=X^{(\sim i)} )}
{P_{\sigma|G}(\sigma=X)}
 = \exp\Big(2\big(\beta+\frac{\alpha\log(n)}{n} \big) (B_i-A_i)
-\frac{2\alpha\log(n)}{n} \Big) 
 = (1+o(1)) \exp (2 \beta (B_i-A_i))  ,
\end{align*}
where the second equality holds with high probability because $|B_i-A_i|=O(\log(n))$ with probability $1-o(1)$.
Note that $E[\exp (2 \beta (B_i-A_i))]$ is essentially the moment generating function of $B_i-A_i$. By definition,
$A_i\sim \Binom(\frac{n}{2}-1,\frac{a\log(n)}{n})$ and $B_i\sim \Binom(\frac{n}{2}, \frac{b\log(n)}{n})$, and they are independent. Therefore,
\begin{align*}
E_G[\exp (2 \beta (B_i-A_i))]
& =\Big(1-\frac{b\log(n)}{n}+\frac{b\log(n)}{n} e^{2\beta} \Big)^{n/2}
\Big(1-\frac{a\log(n)}{n}+\frac{a\log(n)}{n} e^{-2\beta} \Big)^{n/2-1}  \\
& = 
\exp\Big(\frac{\log(n)}{2} ( a e^{-2\beta}+b e^{2\beta} -a-b )
+o(1) \Big) 
 = (1+o(1)) n^{g(\beta)-1} ,
\end{align*}
where $E_G$ means that the expectation is taken over the randomness of $G$, and the function
$g(\beta)  := \frac{b e^{2\beta}+a e^{-2\beta}}{2}-\frac{a+b}{2}+1$ is defined in Theorem~\ref{thm:wt3}.
As a consequence,
\begin{equation} \label{eq:mew}
E_G \Big[ \frac{P_{\sigma|G} ( \dist(\sigma, X) = 1 )}{P_{\sigma|G}(\sigma= X)} \Big]
= (1+o(1)) \sum_{i=1}^n E_G[\exp (2 \beta (B_i-A_i))]
= (1+o(1)) n^{g(\beta)} .
\end{equation}
One can show that $g(\beta)$ is a convex function and takes minimum at $\beta=\frac{1}{4}\log\frac{a}{b}$, so $g(\beta)$ is strictly decreasing in the interval $(0,\frac{1}{4}\log\frac{a}{b})$. Furthermore, $\beta^\ast$ is a root of $g(\beta)=0$, and $0<\beta^\ast<\frac{1}{4}\log\frac{a}{b}$, so $g(\beta)>0$ for $\beta<\beta^\ast$ and $g(\beta)<0$ for $\beta^\ast<\beta<\frac{1}{4}\log\frac{a}{b}$.
Taking this into the above equation, we conclude that the expectation of $\frac{P_{\sigma|G} ( \dist(\sigma, X) = 1 )}{P_{\sigma|G}(\sigma= X)}$ has a sharp transition from $\omega(1)$ to $o(1)$ at $\beta^\ast$.
This at least intuitively explains why $\beta^\ast$ is the threshold. However, in order to formally establish statement (4) above, we need to prove that this sharp transition happens for a typical graph $G$, not just for the expectation.
Moreover, $g(\beta)$ is an increasing function in the interval $\beta\in(\frac{1}{4}\log\frac{a}{b}, +\infty)$, so the expectation first decreases in the interval $(0,\frac{1}{4}\log\frac{a}{b}]$ and then starts increasing. We will prove that there is a ``cut-off" effect when $\beta>\frac{1}{4}\log\frac{a}{b}$, i.e., although the expectation becomes much larger than $n^{g(\frac{1}{4}\log\frac{a}{b})}$, for a typical graph $G$, we always have
$$
\frac{P_{\sigma|G} ( \dist(\sigma, X) = 1 )}{P_{\sigma|G}(\sigma= X)} 
= O( n^{g(\frac{1}{4}\log\frac{a}{b})} ) =o(1)
$$
whenever $\beta>\frac{1}{4}\log\frac{a}{b}$.
Below we divide the proof into three cases: (i) $\beta\in(0,\beta^\ast]$, (ii) $\beta\in(\beta^\ast,\frac{1}{4}\log\frac{a}{b}]$, and (iii) $\beta\in(\frac{1}{4}\log\frac{a}{b},+\infty)$.
Case (ii) is the simplest case, and its proof is essentially an application of Markov inequality, so we start with this case.

\subsection{Proof for $\beta\in(\beta^\ast,\frac{1}{4}\log\frac{a}{b}]$: An application of Markov inequality}
\label{sect:simreg}

We know that $g(\beta)<0$ for $\beta\in(\beta^\ast,\frac{1}{4}\log\frac{a}{b}]$. By \eqref{eq:mew} and Markov inequality, for almost all\footnote{By almost all $G$, we mean there is a set $\cG$ such that $P(G\in\cG)=1-o(1)$ and for every $G\in\cG$ certain property holds. The probability $P(G\in\cG)$ is calculated according to SSBM defined in Definition~\ref{def:SSBM}.} $G$, we have $\frac{P_{\sigma|G} ( \dist(\sigma, X) = 1 )}{P_{\sigma|G}(\sigma= X)}=o(1)$. This proves that $\frac{P_{\SIBM} ( \dist(\sigma, X) = 1 )}{P_{\SIBM}(\sigma= X)}=o(1)$ in this interval. With a bit more extra effort, let us also prove that $\frac{P_{\SIBM} ( 1\le \dist(\sigma, X)< 2n/\log^{1/3}(n) )}{P_{\SIBM}(\sigma= X)}=o(1)$.
By definition,
$$
\frac{P_{\sigma|G} ( \dist(\sigma, X) = k )}{P_{\sigma|G}(\sigma= X)}
=\sum_{\cI\subseteq[n],|\cI|=k} \frac{P_{\sigma|G} ( \sigma= X^{(\sim \cI)} )} {P_{\sigma|G}(\sigma= X)} .
$$
Similarly to $A_i$ and $B_i$, for a set $\cI\subseteq[n]$, we define
$A_{\cI}=A_{\cI}(G):=|\{\{i,j\}\in E(G):  i\in\cI, j\in[n]\setminus\cI, X_i=X_j\}|$ and  
$B_{\cI}=B_{\cI}(G):=|\{\{i,j\}\in E(G): i\in\cI, j\in[n]\setminus\cI, X_i=-X_j\}|$.
Then by \eqref{eq:isingma} one can show that
$$
 \frac{P_{\sigma|G}(\sigma=X^{(\sim\cI)})}{P_{\sigma|G}(\sigma=X)} 
\le \exp\Big( 2\big(\beta+\frac{\alpha\log(n)}{n} \big) (B_{\cI}-A_{\cI})
\Big) = \exp ( 2 (\beta + o(1)) (B_{\cI}-A_{\cI})
 ) .
$$
Since we are only interested in the case $|\cI|<2n/\log^{1/3}(n)=o(n)$,
by definition we have $A_{\cI}\sim\Binom((\frac{n}{2}-o(n))|\cI|,\frac{a\log(n)}{n})$ and $B_{\cI}\sim\Binom((\frac{n}{2}-o(n))|\cI|,\frac{b\log(n)}{n})$, and they are independent. Therefore,
\begin{align*}
E_G[\exp ( 2 (\beta +o(1)) (B_{\cI}-A_{\cI}) ) ] = &
\exp\Big(\frac{|\cI|\log(n)}{2} ( a e^{-2\beta}+b e^{2\beta} -a-b +o(1) )
 \Big) \\
 = & n^{|\cI|(g(\beta)-1+o(1))} .
\end{align*}
As a consequence,
\begin{equation} \label{eq:nn}
E_G \Big[ \frac{P_{\sigma|G} ( \dist(\sigma, X) = k )}{P_{\sigma|G}(\sigma= X)} \Big]
\le \sum_{\cI\subseteq[n],|\cI|=k} 
n^{k (g(\beta)-1+o(1))}
= \binom{n}{k} n^{k (g(\beta)-1+o(1))}
< n^{k (g(\beta)+o(1))} .
\end{equation}
Then by Markov inequality, there is a set $\cG^{(k)}$ such that $P(G\in\cG^{(k)})=1-n^{k g(\beta)/4}$ and for every $G\in\cG^{(k)}$, $\frac{P_{\sigma|G} ( \dist(\sigma, X) = k )}{P_{\sigma|G}(\sigma= X)} \le n^{k g(\beta)/2}$.
Let $\cG=\cap_{k=1}^{2n/\log^{1/3}(n)} \cG^{(k)}$. By union bound, we have $P(G\in\cG)>1-\sum_{k=1}^\infty n^{k g(\beta)/4} = 1-o(1)$. Moreover, for every $G\in\cG$,
$\frac{P_{\sigma|G} ( 1\le \dist(\sigma, X)< 2n/\log^{1/3}(n) )}{P_{\sigma|G}(\sigma= X)} < \sum_{k=1}^\infty n^{k g(\beta)/2} = o(1)$.
This proves that $\frac{P_{\SIBM} ( 1\le \dist(\sigma, X)< 2n/\log^{1/3}(n) )}{P_{\SIBM}(\sigma= X)}=o(1)$, and so $P_{\SIBM}(\sigma=\pm X)=1-o(1)$ when $\beta\in(\beta^\ast,\frac{1}{4}\log\frac{a}{b}]$.

\subsection{Proof for $\beta\in(\frac{1}{4}\log\frac{a}{b},+\infty)$: The ``cut-off" effect}

The analysis in this interval is more delicate. Since $\frac{P_{\sigma|G} ( \dist(\sigma, X) = 1 )}{P_{\sigma|G}(\sigma= X)} 
= (1+o(1)) \sum_{i=1}^n \exp (2 \beta (B_i-A_i))$, we start with a more careful analysis of $\sum_{i=1}^n \exp (2 \beta (B_i-A_i))$.
Since $B_i-A_i$ takes integer value between $-n/2$ and $n/2$,
we can use indicator functions to write
\begin{align*}
 \exp (2 \beta (B_i-A_i))
=  \sum_{t\log(n)=-n/2}^{n/2}
\mathbbm{1}[B_i-A_i=t \log(n)] \exp (2\beta t \log(n) ) ,
\end{align*}
where the quantity $t\log(n)$ ranges over all integer values from $-n/2$ to $n/2$ in the summation.
Define $D(G,t):=|\{i\in[n]:B_i-A_i= t\log(n)\}|=\sum_{i=1}^n \mathbbm{1}[B_i-A_i=t \log(n)]$.
Therefore,
\begin{equation}  \label{eq:gour}
 \sum_{i=1}^n \exp (2 \beta (B_i-A_i))
=  \sum_{t\log(n)=-n/2}^{n/2}
D(G,t) \exp (2\beta t \log(n) ) .
\end{equation}
By Chernoff bound, we have $P(B_i- A_i\ge 0)\le \exp\big(\log(n)(-\frac{(\sqrt{a}-\sqrt{b})^2}{2} +o(1)) \big)$.
Define $\cG_1:=\{G:B_i- A_i < 0~\forall i\in[n]\}$. Then by union bound, $P(G\notin\cG_1)\le\exp\big(\log(n)(1-\frac{(\sqrt{a}-\sqrt{b})^2}{2} +o(1)) \big) = o(1)$, where the equality follows from the assumption that $\sqrt{a}-\sqrt{b}>\sqrt{2}$.
For every $G\in\cG_1$, $D(G,t)=0$ for all $t\ge 0$, and so
\begin{equation} \label{eq:duj}
 \sum_{i=1}^n \exp (2 \beta (B_i-A_i))
=  \sum_{t\log(n)=-n/2}^{-1}
D(G,t) \exp (2\beta t \log(n) ) .
\end{equation}
This indicates that there is a ``cut-off" effect at $t>0$, i.e., $D(G,t) \exp (2\beta t \log(n) )=0$ for all positive $t$ with probability $1-o(1)$, although its expectation can be very large, as we will show next.
Define a function
$$
f_{\beta}(t):=\sqrt{t^2+ab} -t\big(\log(\sqrt{t^2+ab}+t)-\log(b) \big) -\frac{a+b}{2} +1 +2\beta t.
$$
Using Chernoff bound, one can show that 
$$
E[D(G,t)
\exp\big(2\beta t \log(n) \big)]
\le \exp( f_{\beta}(t) \log(n) ) .
$$
(Using a more careful analysis, one can show that this bound is tight up to a $\frac{1}{\sqrt{\log(n)}}$ factor; see Appendix~\ref{ap:um}.)
The function $f_{\beta}(t)$ is a concave function and takes maximum value at $t^\ast=\frac{b e^{2\beta}-a e^{-2\beta}}{2}$, and its maximum value is $f_{\beta}(t^\ast)=\frac{b e^{2\beta}+a e^{-2\beta}}{2}-\frac{a+b}{2}+1
=g(\beta)$.
Therefore, if we take expectation on both sides of \eqref{eq:gour}, then the sum on the right-hand side is concentrated on a small neighborhood of $t^\ast$. When $\beta>\frac{1}{4}\log\frac{a}{b}$, we have $t^\ast>0$. Due to the ``cut-off" effect at $t>0$, we have $D(G,t)=0$ for all $t$ in the neighborhood of $t^\ast$ with probability $1-o(1)$, so the main contribution to the expectation comes from a rare event $G\notin\cG_1$. This explains why the behavior of a typical graph $G$ deviates from the behavior of the expectation.
Since $f_{\beta}(t)$ is a concave function, the sum $E[\sum_{t\log(n)=-n/2}^{-1}
D(G,t) \exp (2\beta t \log(n) )]$ is upper bounded by $O(\log(n))n^{f_{\beta}(0)}$  when $t^\ast>0$. Notice that $f_{\beta}(0)=g(\frac{1}{4}\log\frac{a}{b})=1-\frac{(\sqrt{a}-\sqrt{b})^2}{2}<0$. Now using \eqref{eq:duj} and Markov inequality, we conclude that $\sum_{i=1}^n \exp (2 \beta (B_i-A_i))=o(1)$ for almost all $G$, and so $\frac{P_{\sigma|G} ( \dist(\sigma, X) = 1 )}{P_{\sigma|G}(\sigma= X)} =o(1)$ for almost all $G$. Thus we have shown that $\frac{P_{\SIBM} ( \dist(\sigma, X) = 1 )}{P_{\SIBM}(\sigma= X)}=o(1)$ when $\beta>\frac{1}{4}\log\frac{a}{b}$. The analysis of $\frac{P_{\SIBM} ( \dist(\sigma, X) = k )}{P_{\SIBM}(\sigma= X)}$ for $1\le k<2n/\log^{1/3}(n)$ is similar to the analysis in Section~\ref{sect:simreg}, and we do not repeat it here. By now we have given a sketched proof of
$P_{\SIBM}(\sigma=\pm X)=1-o(1)$ when $\beta>\beta^\ast$; see Section~\ref{sect:equal} for a rigorous proof. Next we move to the case $\beta\le\beta^\ast$.

\subsection{Proof for $\beta\le\beta^\ast$: Structural results and tight concentration}

In the last inequality of \eqref{eq:nn}, we use a coarse bound $\binom{n}{k}<n^k$. Now let us use a tighter bound $\binom{n}{k}<n^k/(k!)$. 
For $k>n^{g(\beta)+\delta}$, we have
$
k!>(k/e)^k
=\exp(k\log(k)-k)
>\exp(k(g(\beta)+\delta)\log(n)-k)
=n^{k(g(\beta)+\delta-o(1))} .
$
Taking these into the last inequality of \eqref{eq:nn}, we obtain that for all $k>n^{g(\beta)+\delta}$,
$$
E_G \Big[ \frac{P_{\sigma|G} ( \dist(\sigma, X) = k )}{P_{\sigma|G}(\sigma= X)} \Big]
\le  \binom{n}{k} n^{k (g(\beta)-1+o(1))}
< n^{k (g(\beta)+o(1))} /(k!) 
< n^{-k(\delta-o(1))} .
$$
This immediately implies that  $P_{\SIBM} (\dist(\sigma, \pm X)<n^{g(\beta)+\delta} ) = 1- o(1)$ for any $\delta>0$. Since $g(\beta)<1$ for all $0<\beta\le\beta^\ast$, we have $P_{\SIBM} (\dist(\sigma, \pm X)<n^{\theta} ) = 1- o(1)$ for all $\theta\in (g(\beta), 1)$. This improves upon the upper bound 
$\dist(\sigma, \pm X)< 2n/\log^{1/3}(n)$ we obtained using spectral method at the beginning of this section.
More importantly, this allows us to prove a powerful structural result. (All the discussions below are conditioning on the event $\dist(\sigma,X)\le n/2$, i.e., $\sigma$ is closer to $X$ than to $-X$.) We say that $j$ is a ``bad" neighbor of vertex $i$ if the edge $\{i,j\}$ is connected in graph $G$ and $\sigma_j\neq X_j$. Then there is an integer $z>0$ such that with probability $1-o(1)$, every vertex has at most $z$ ``bad" neighbors.
Conditioning on the event every vertex has at most $z$ ``bad" neighbors, it is easy to show that $P_{\sigma|G}(\sigma_i=- X_i)$ differs from $\exp (2 \beta (B_i-A_i))$ by at most a constant factor $\exp(4\beta z)$.
Therefore, $E_{\sigma|G}[\dist(\sigma,X)]=\sum_{i=1}^n P_{\sigma|G}(\sigma_i=- X_i)$ differs from $\sum_{i=1}^n\exp (2 \beta (B_i-A_i))$ by at most a constant factor for almost all $G$.
We can further prove that the pairwise correlation of the events $\{\sigma_i=- X_i\}$ and $\{\sigma_j=- X_j\}$ is very small, so $\dist(\sigma,X)$ concentrates around its expectation. Thus we conclude that $\dist(\sigma,X)$ differs from $\sum_{i=1}^n\exp (2 \beta (B_i-A_i))$ by at most a constant factor for almost all $G$.
In Section~\ref{sect:why} (see \eqref{eq:mew}), we have shown that $E_G[\sum_{i=1}^n\exp (2 \beta (B_i-A_i))]=(1+o(1))n^{g(\beta)}$. Quite surprisingly, when $\beta\le\beta^\ast$, we can prove a very tight concentration around the expectation: For almost all graph $G$, we have $\sum_{i=1}^n\exp (2 \beta (B_i-A_i))=(1+o(1))n^{g(\beta)}$; see Proposition~\ref{prop:con} in Section~\ref{sect:struct} for a proof. Combining this with the above analysis, we conclude that $\dist(\sigma,X)=\Theta(n^{g(\beta)})$ with probability $1-o(1)$ when $\beta\le\beta^\ast$. This completes the sketched proof of Theorem~\ref{thm:wt3}.
See Sections~\ref{sect:theta}--\ref{sect:struct} for the rigorous proof of the above arguments.
As a final remark, we note that for the special case of $\beta=\beta^\ast$,  Theorem~\ref{thm:wt3} tells us that
$
P_{\SIBM}(\dist(\sigma,\pm X)= \Theta(1) ) = 1-o(1) ,
$
but this is not sufficient for us to draw any conclusion on $P_{\SIBM}(\sigma= \pm X)$.
In Proposition~\ref{prop:zj} (see Section~\ref{sect:struct}) we prove that when $\beta=\beta^\ast$, $P_{\SIBM}(\sigma= \pm X) \le \frac{1}{2}(1+o(1))$, where the $o(1)$ term goes to $0$ as $n\to\infty$.

\subsection{Multiple sample case: Proof of Theorem~\ref{thm:wt2}}
\label{sect:multi}

\begin{center}
    \begin{minipage}{.55\textwidth}
 \begin{algorithm}[H]
\caption{\texttt{LearnSIBM} in $O(n)$ time} \label{alg:ez}
Inputs: the samples $\sigma^{(1)},\sigma^{(2)}\dots,\sigma^{(m)}$ \\
Output: $\hat{X}$
\begin{algorithmic}[1]
\Statex \hspace*{-0.3in} 
{\bf Step 1: Align all the samples with $\sigma^{(1)}$ }
\For {$j=2,3,\dots,m$}
\If {$\sum_{i=1}^n \sigma_i^{(1)} \sigma_i^{(j)} <0$}
\State $\sigma^{(j)} \gets -\sigma^{(j)}$
\EndIf 
\EndFor
\Statex \hspace*{-0.3in}
{\bf Step 2: Majority vote at each coordinate}
\For {i=1,2,\dots,n}
\State $\hat{X}_i \gets \sign(\sum_{j=1}^m \sigma_i^{(j)})$
\State \Comment{If $\sum_{j=1}^m \sigma_i^{(j)}=0$, assign $\hat{X}_i$ a random sign}
\EndFor
\State Output $\hat{X}$
\end{algorithmic}
 \end{algorithm}
    \end{minipage}
\end{center}
For the multiple-sample case, we prove that the above simple algorithm can recover $X$ with probability $1-o(1)$ if and only if the Maximum Likelihood (ML) algorithm recovers $X$ with probability $1-o(1)$. 
Notice that Algorithm~\ref{alg:ez} does not require knowledge of the parameters of SIBM.
We already showed that each sample is either very close to $X$ or very close to $-X$, so after the alignment step in Algorithm~\ref{alg:ez}, all the samples are either simultaneously aligned with $X$ or simultaneously aligned with $-X$. We assume the former case. By the structural results discussed above, with probability $1-o(1)$, $P_{\sigma|G}(\sigma_i^{(j)}=- X_i)$ differs from $\exp (2 \beta (B_i-A_i))$ by at most a constant factor for all $j\in[m]$. Since the samples are independent, we further obtain that $P_{\sigma|G}(\sum_{j=1}^m\mathbbm{1}[\sigma_i^{(j)}=- X_i]\ge u)$ differs from $\exp (2 u\beta (B_i-A_i))$ by at most a constant factor. Here $u\beta$ plays the role of $\beta$ in the single-sample case. Therefore, if $u\beta>\beta^\ast$, then with probability $1-o(1)$ we have $\sum_{j=1}^m\mathbbm{1}[\sigma_i^{(j)}=- X_i] \le u-1$, or equivalently, $X_i(\sum_{j=1}^m \sigma_i^{(j)} ) \ge m-2u+2$ for all $i\in[n]$. In particular, if $\lfloor \frac{m+1}{2} \rfloor \beta>\beta^\ast$, then $X_i(\sum_{j=1}^m \sigma_i^{(j)} ) \ge m+2-2\lfloor \frac{m+1}{2} \rfloor\ge 1$ for all $i\in[n]$, which implies that $\hat{X}=X$ after the majority voting step in Algorithm~\ref{alg:ez}. See Section~\ref{sect:direct} for a rigorous proof of the above argument.

The proof of the converse results, i.e., even ML algorithm cannot recover $X$ with probability $1-o(1)$ when $\lfloor \frac{m+1}{2} \rfloor \beta < \beta^\ast$, also relies on the structural result and it is rather similar to the proof for $\beta\le\beta^\ast$ in the single-sample case. We refer the readers to Section~\ref{sect:converse} for details.

\section{Samples are concentrated around $\pm X$ or $\pm \mathbf{1}_n$} \label{sect:aln}

In this section, we show that for $a>b$, if $\alpha>b\beta$, then all the samples produced by $\SIBM(n, \linebreak[4]
a\log(n)/n, b\log(n)/n,\alpha,\beta,m)$ are very close to either $X$ or $-X$. More precisely, they differ from the ground truth $\pm X$ in at most $2n/\log^{1/3}(n)$ coordinates.
On the other hand, if $\alpha<b\beta$, then the samples differ from  $\pm \mathbf{1}_n$ in at most $2n/\log^{1/3}(n)$ coordinates, where $\mathbf{1}_n$ is the all-one vector of length $n$.
For the latter case, we prove that the number of samples needed for exact recovery of $X$ is at least $\Omega(\log^{1/4}(n))$.

Let $A=A(G)$ be the adjacency matrix of $G$. Then \eqref{eq:isingma} can be written as
\begin{align*}
& P_{\sigma|G}(\sigma=\bar{\sigma})=\frac{1}{Z_G(\alpha,\beta)}
\exp\Big(\frac{\beta}{2} \bar{\sigma} A \bar{\sigma}^T 
-\frac{\alpha\log(n)}{2n} \bar{\sigma}   (J_n-I_n-A) \bar{\sigma}^T
\Big)  \\
= & \frac{1}{Z_G(\alpha,\beta)}
\exp\Big( \frac{1}{2} \bar{\sigma} \Big( \big(\beta+\frac{\alpha\log(n)}{n} \big) A
-\frac{\alpha\log(n)}{n} (J_n-I_n) \Big) \bar{\sigma}^T 
\Big)  ,
\end{align*}
where $J_n$ is the all one matrix and $I_n$ is the identity matrix, both of size $n\times n$.
Conditioned on the ground truth $X$,
$A-E[A|X]$ is a symmetric matrix whose upper triangular part consists of independent entries. According to Theorem~5.2 in \cite{Lei15} and/or Theorem~5 in \cite{Hajek16}, both of which are built upon the classical results in \cite{Feige05}, the spectral norm of $A-E[A|X]$ is upper bounded by $O(\sqrt{\log(n)})$.
\begin{theorem}[Theorem~5.2 in \cite{Lei15}, Theorem~5 in \cite{Hajek16}] \label{thm:a2}
For any $r>0$, there exists $c>0$ such that the spectral norm of $A-E[A|X]$ satisfies
$$
P\big(\|A-E[A|X]\| \le c\sqrt{\log(n)} \big)\ge 1-n^{-r} .
$$
\end{theorem}
Define a matrix
$$
M:= \big(\beta+\frac{\alpha\log(n)}{n} \big) E[A|X]
-\frac{\alpha\log(n)}{n} (J_n-I_n).
$$
Then we can further write \eqref{eq:isingma}
as
\begin{equation} \label{eq:M}
P_{\sigma|G}(\sigma=\bar{\sigma})
= \frac{1}{Z_G(\alpha,\beta)}
\exp\Big( \frac{1}{2} \bar{\sigma}  M \bar{\sigma}^T + \frac{1}{2} \big(\beta+\frac{\alpha\log(n)}{n} \big) \bar{\sigma}  (A-E[A|X])
  \bar{\sigma}^T 
\Big)  .
\end{equation}

By definition, all the diagonal entries of $M$ are $0$.
For $i\neq j$, 
$$
M_{ij}=\left\{ 
\begin{array}{cc}
 \big(a\beta-\alpha+\frac{a\alpha\log(n)}{n} \big) \frac{\log(n)}{n}
   & \text{~if~} X_i=X_j \\
  \big(b\beta-\alpha+\frac{b\alpha\log(n)}{n} \big) \frac{\log(n)}{n}
   & \text{~if~} X_i\neq X_j
\end{array}
\right. .
$$
Given $\bar{\sigma}\in\{\pm 1\}^n$, let $u=u(\bar{\sigma}):=|\{i\in[n]:X_i=\bar{\sigma}_i=1\}|$ and $v=v(\bar{\sigma}):=|\{i\in[n]:X_i=\bar{\sigma}_i=-1\}|$. Then 
$|\{i\in[n]:X_i=1,\bar{\sigma}_i=-1\}|=n/2-u$ and $|\{i\in[n]:X_i=-1,\bar{\sigma}_i=1\}|=n/2-v$. Therefore,
\begin{equation} \label{eq:sMs}
\begin{aligned}
 \frac{1}{2} \bar{\sigma} M \bar{\sigma}^T  
= & \frac{1}{2}\Big(\big(\frac{n}{2}-2u\big)^2 +\big(\frac{n}{2}-2v\big)^2 \Big) \frac{(a\beta-\alpha)\log(n)}{n} \\
& -\big(\frac{n}{2}-2u\big) \big(\frac{n}{2}-2v\big)
\frac{(b\beta-\alpha)\log(n)}{n} 
+O(\log^2(n))  \\
 = & \frac{1}{2}\Big(\big(\frac{n}{2}-2u\big)^2 +\big(\frac{n}{2}-2v\big)^2 \Big) \frac{(a\beta-b\beta)\log(n)}{n} \\
& +\frac{1}{2} (2u-2v)^2
\frac{(b\beta-\alpha)\log(n)}{n}
+O(\log^2(n))  .
\end{aligned}
\end{equation}
Notice that $u$ and $v$ take values between $0$ and $n/2$.
According to \eqref{eq:M}, the configuration $\bar{\sigma}$ that maximizes $\frac{1}{2}\bar{\sigma} M \bar{\sigma}^T$ is (roughly) the most likely output of the Ising model.
Since we assume $a>b$, the term $\frac{1}{2}\Big(\big(\frac{n}{2}-2u\big)^2 +\big(\frac{n}{2}-2v\big)^2 \Big) \frac{(a\beta-b\beta)\log(n)}{n}$ takes maximum at four points $u,v\in\{0,n/2\}$. If $b\beta<\alpha$, then the second term $\frac{1}{2} (2u-2v)^2
\frac{(b\beta-\alpha)\log(n)}{n}$ takes maximum whenever $u=v$, and if $b\beta>\alpha$, then the second term takes maximum when $|u-v|=n/2$.
To summarize, if $b\beta<\alpha$, then $\frac{1}{2}\bar{\sigma} M \bar{\sigma}^T$ takes maximum at $u=v=0$ and $u=v=n/2$, i.e., the two maximizers are $\bar{\sigma}=\pm X$. If $b\beta>\alpha$, then $\frac{1}{2}\bar{\sigma} M \bar{\sigma}^T$ takes maximum at $(u=0,v=n/2)$ and $(u=n/2,v=0)$, i.e., the two maximizers are $\bar{\sigma}=\pm \mathbf{1}_n$.
Taking into account the effect of the error term $\frac{1}{2} \big(\beta+\frac{\alpha\log(n)}{n} \big) \bar{\sigma}  (A-E[A|X])
  \bar{\sigma}^T$ in \eqref{eq:M}, we have the following proposition:

\begin{proposition} \label{prop:1}
Let $a>b>0$ and $\alpha,\beta>0$ be constants. Let $m$ be a positive integer that is upper bounded by some polynomial of $n$.
Let 
$$
(X,G,\{\sigma^{(1)},\dots,\sigma^{(m)}\})\sim \SIBM(n,a\log(n)/n, b\log(n)/n,\alpha,\beta, m) .
$$
If $b\beta <\alpha$, then for any (arbitrarily large) $r>0$, there exists $n_0(r)$ such that for all even integers $n>n_0(\delta, r)$,
$$
P_{\SIBM} \Big(\dist(\sigma^{(i)},\pm X)< 2n/\log^{1/3}(n)
\text{~~for all~} i\in[m] \Big) \ge 1- n^{-r} .
$$
If $b\beta >\alpha$, then for any (arbitrarily large) $r>0$, there exists $n_0(r)$ such that for all even integers $n>n_0(\delta, r)$,
$$
P_{\SIBM} \Big(\dist(\sigma^{(i)},\pm \mathbf{1}_n)< 2n/\log^{1/3}(n)
\text{~~for all~} i\in[m] \Big) \ge 1- n^{-r} .
$$
\end{proposition}
\begin{proof}
We only prove the case of $b\beta <\alpha$ as the proof of the other case is virtually identical.
Since $A-E[A|X]$ is a symmetric matrix,
$$
|\bar{\sigma}  (A-E[A|X]) \bar{\sigma}^T|\le
\|A-E[A|X]\| \bar{\sigma} \bar{\sigma}^T=n \|A-E[A|X]\|
$$
for every $\bar{\sigma}\in\{\pm 1\}^n$.
Therefore, by Theorem~\ref{thm:a2}, for any $r>0$, there is $c>0$ such that
\begin{equation} \label{eq:cnA}
\left|\frac{1}{2} \big(\beta+\frac{\alpha\log(n)}{n} \big) \bar{\sigma}  (A-E[A|X])
  \bar{\sigma}^T \right| \le c n \sqrt{\log(n)}
  \text{~~for all~} \bar{\sigma}\in\{\pm 1\}^n
\end{equation}
with probability at least $1-n^{-2r}$.
Define the set 
\begin{align*}
\Gamma:= & \Big\{\bar{\sigma}\in\{\pm 1\}^n:u\le n/\log^{1/3}(n), v\le n/\log^{1/3}(n) \Big\} \\
& \bigcup \Big\{\bar{\sigma}\in\{\pm 1\}^n:u\ge n/2- n/\log^{1/3}(n), v \ge n/2- n/\log^{1/3}(n) \Big\} ,
\end{align*}
where the first set consists of $\bar{\sigma}$'s that differ from $-X$ in at most $2n/\log^{1/3}(n)$ coordinates, and the second set consists of $\bar{\sigma}$'s that differ from $X$ in at most $2n/\log^{1/3}(n)$ coordinates.
Given a graph $G$ whose adjacency matrix satisfies \eqref{eq:cnA}, we will show that $P_{\sigma|G}(\sigma\notin\Gamma)< e^{-n}$. Since $P_{\sigma|G}(\sigma=X)<1$, it suffices to prove that $\frac{P_{\sigma|G}(\sigma\notin\Gamma)}{P_{\sigma|G}(\sigma=X)}<e^{-n}$. 
Define another three sets
\begin{align*}
\Gamma_1:= & \Big\{ \bar{\sigma}\in\{\pm 1\}^n: n/\log^{1/3}(n)< u< n/2- n/\log^{1/3}(n)
\Big\} \\
\Gamma_2:= & 
\Big\{\bar{\sigma}\in\{\pm 1\}^n: n/\log^{1/3}(n)< v< n/2- n/\log^{1/3}(n)
\Big\} ,\\
\Gamma_3:= & \Big\{\bar{\sigma}\in\{\pm 1\}^n:u\le n/\log^{1/3}(n), v \ge n/2- n/\log^{1/3}(n)  \Big\} \\
& \bigcup \Big\{\bar{\sigma}\in\{\pm 1\}^n:u\ge n/2- n/\log^{1/3}(n), v\le n/\log^{1/3}(n) \Big\} .
\end{align*}
It is easy to verify that  $\Gamma^c=\Gamma_1\cup\Gamma_2\cup\Gamma_3$.
Next we prove that for every $G$ satisfying \eqref{eq:cnA} and every $\bar{\sigma}\in\Gamma_1\cup\Gamma_2\cup\Gamma_3$, 
\begin{equation} \label{eq:xmb}
\frac{P_{\sigma|G}(\sigma=\bar{\sigma})}{P_{\sigma|G}(\sigma=X)}<2^{-n}e^{-n} .
\end{equation}
Together with the trivial upper bound $|\Gamma_1\cup\Gamma_2\cup\Gamma_3|<2^n$, this implies that $\frac{P_{\sigma|G}(\sigma\notin\Gamma)}{P_{\sigma|G}(\sigma=X)}<e^{-n}$.

We first prove \eqref{eq:xmb} for $\bar{\sigma}\in\Gamma_1$.
Observe that $\sigma=X$ corresponds to $u=v=n/2$. By \eqref{eq:sMs}, we have
\begin{equation} \label{eq:X}
\begin{aligned}
P_{\sigma|G}(\sigma=X) & =\frac{1}{Z_G(\alpha,\beta)}
\exp\Big(\frac{1}{2} X M X^T + O(n\sqrt{\log(n)}) \Big)  \\
& =\frac{1}{Z_G(\alpha,\beta)}
\exp\Big(\frac{a\beta-b\beta}{4}n\log(n) + O(n\sqrt{\log(n)}) \Big) .
\end{aligned}
\end{equation}
On the other hand, if $\bar{\sigma}\in\Gamma_1$, then 
$$
(\frac{n}{2}-2u(\bar{\sigma}))^2\le (\frac{n}{2}-2n/\log^{1/3}(n))^2
=\frac{n^2}{4}-2n^2/\log^{1/3}(n)
+O(n^2/\log^{2/3}(n)).
$$
Since $b\beta-\alpha<0$ and $(\frac{n}{2}-2v)^2\le \frac{n^2}{4}$ for all $0\le v\le n/2$, by \eqref{eq:sMs} we have
$$
\frac{1}{2}\bar{\sigma} M \bar{\sigma}^T  
\le \frac{a\beta-b\beta}{4}n\log(n)
-(a\beta-b\beta)n\log^{2/3}(n)
+O(n\log^{1/3}(n)) ,
$$
and so
\begin{align*}
P_{\sigma|G}(\sigma=\bar{\sigma}) \le \frac{1}{Z_G(\alpha,\beta)}
\exp\Big(\frac{a\beta-b\beta}{4}n\log(n) 
-(a\beta-b\beta)n\log^{2/3}(n)
+ O(n\sqrt{\log(n)}) \Big) .
\end{align*}
Combining this with \eqref{eq:X}, we have
$$
\frac{P_{\sigma|G}(\sigma=\bar{\sigma})}{P_{\sigma|G}(\sigma=X)}
\le \exp\Big(-(a\beta-b\beta)n\log^{2/3}(n)
+ O(n\sqrt{\log(n)}) \Big)
<2^{-n} e^{-n}
$$
for all $\bar{\sigma}\in\Gamma_1$ and all $G$ satisfying \eqref{eq:cnA}. The case of $\bar{\sigma}\in\Gamma_2$ can be proved in the same way.

For $\bar{\sigma}\in\Gamma_3$, we have 
$$
|u(\bar{\sigma})-v(\bar{\sigma})| > n/4
$$
for large $n$. Since $b\beta-\alpha<0$, this implies that
$$
\frac{1}{2} (2u(\bar{\sigma})-2v(\bar{\sigma}))^2
\frac{(b\beta-\alpha)\log(n)}{n}
< - \frac{\alpha-b\beta}{8} n\log(n) ,
$$
so
$$
\frac{1}{2}\bar{\sigma} M \bar{\sigma}^T  
\le \frac{a\beta-b\beta}{4}n\log(n)
- \frac{\alpha-b\beta}{8} n\log(n) .
$$
Therefore,
\begin{align*}
P_{\sigma|G}(\sigma=\bar{\sigma}) \le \frac{1}{Z_G(\alpha,\beta)}
\exp\Big(\frac{a\beta-b\beta}{4}n\log(n) 
-\frac{\alpha-b\beta}{8} n\log(n)
+ O(n\sqrt{\log(n)}) \Big) .
\end{align*}
Combining this with \eqref{eq:X}, we have
$$
\frac{P_{\sigma|G}(\sigma=\bar{\sigma})}{P_{\sigma|G}(\sigma=X)}
\le \exp\Big(-\frac{\alpha-b\beta}{8} n\log(n)
+ O(n\sqrt{\log(n)}) \Big)
<2^{-n} e^{-n}
$$
for all $\bar{\sigma}\in\Gamma_3$ and all $G$ satisfying \eqref{eq:cnA}.

Now we have shown that for a single sample $\sigma$ produced by the SIBM, $P_{\sigma|G}(\sigma\in\Gamma)\ge 1- e^{-n}$ provided that $G$ satisfies \eqref{eq:cnA}. By the union bound, for $m$ independent samples $\sigma^{(1)},\dots,\sigma^{(m)}$ produced by the SIBM, $P_{\sigma|G}(\sigma^{(1)},\dots,\sigma^{(m)}\in\Gamma)\ge 1- m e^{-n}$ provided that $G$ satisfies \eqref{eq:cnA}.
We also know that $G$ satisfies \eqref{eq:cnA} with probability at least $1-n^{-2r}$, and by assumption $m$ is upper bounded by some polynomial of $n$. Therefore, the overall probability of $\sigma^{(1)},\dots,\sigma^{(m)}\in\Gamma$ is at least $1-n^{-r}$ when $n$ is large enough. This completes the proof of the proposition.
\end{proof}

\begin{proposition}  \label{prop:ab}
Let $a>b>0$ and $\alpha,\beta>0$ be constants. Let $m$ be a positive integer that is upper bounded by some polynomial of $n$.
Let 
$$
(X,G,\{\sigma^{(1)},\dots,\sigma^{(m)}\})\sim \SIBM(n,a\log(n)/n, b\log(n)/n,\alpha,\beta, m) .
$$
If $\alpha<b\beta$, then it is not possible to recover $X$ from the samples when $m=O(\log^{1/4}(n))$.
\end{proposition}

\begin{proof}
First observe that there are $\binom{n}{n/2}$ balanced partitions, so one needs at least $\log_2 \binom{n}{n/2}=\Theta(n)$ bits to recover $X$.
By Proposition~\ref{prop:1}, with probability $1-o(n^{-4})$, $\dist(\sigma^{(i)},\pm \mathbf{1}_n)< 2n/\log^{1/3}(n)$
for all $i\in[m]$. Therefore, each $\sigma^{(i)}$ takes at most
$$
T:=\sum_{j=0}^{2n/\log^{1/3}(n)} \binom{n}{j}
$$
values, so each $\sigma^{(i)}$ contains at most $\log_2 T$ bits of information. Next we prove that $\log_2 T=O(\frac{\log\log(n)}{\log^{1/3}(n)} n)$, so we need at least $\Omega(\frac{\log^{1/3}(n)}{\log\log(n)})$ samples to recover $X$, which proves the proposition.

In order to upper bound $T$, we define a binomial random variable $Y\sim\Binom(n,1/2)$. Then
$$
T=2^n P(Y\le 2n/\log^{1/3}(n))
= 2^n P(Y\ge n- 2n/\log^{1/3}(n)).
$$
The moment generating function of $Y$ is $(\frac{1}{2}+\frac{1}{2}e^s)^n$. By Chernoff bound, for any $s>0$,
$$
P(Y\ge n- 2n/\log^{1/3}(n)) \le
(\frac{1}{2}+\frac{1}{2}e^s)^n e^{-sn}
e^{2sn/\log^{1/3}(n)}
= 2^{-n} (1+e^{-s})^n e^{2sn/\log^{1/3}(n)} .
$$
As a consequence, for any $s>0$,
$$
\log_2 T\le n\Big(\log(1+e^{-s})
+\frac{2s}{\log^{1/3}(n)} \Big) .
$$
Taking $s=\log\log(n)$ into this bound, we obtain that $\log_2 T=O(\frac{\log\log(n)}{\log^{1/3}(n)} n)$.
\end{proof}

\section{$\sigma=\pm X$ with probability $1-o(1)$ when $\beta>\beta^\ast$} \label{sect:equal}

Recall the definition of $\beta^\ast$ in \eqref{eq:defstar}.
\begin{proposition} \label{prop:tt}
Let $a,b,\alpha,\beta> 0$ be constants satisfying that $\sqrt{a}-\sqrt{b} > \sqrt{2}$, $\beta>\beta^\ast$ and $\alpha>b\beta$. 
Let 
$
(X,G,\sigma) \sim \SIBM(n,a\log(n)/n, b\log(n)/n,\alpha,\beta, 1) .
$
Then
$$
P_{\SIBM}(\sigma=X \text{~or~} \sigma=-X)=1-o(1) .
$$
\end{proposition}

We have proved in Proposition~\ref{prop:1} that if $\alpha>b\beta$, then $\dist(\sigma,\pm X) \le 2n/\log^{1/3}(n)$
 with probability $1-o(1)$.
For $\cI\subseteq[n]$, define $X^{(\sim\cI)}$ as the vector obtained by flipping the coordinates in $\cI$ while keeping all the other coordinates to be the same as $X$, i.e., $X_i^{(\sim\cI)}=-X_i$ for all $i\in\cI$ and $X_i^{(\sim\cI)}=X_i$ for all $i\notin\cI$. 
Then Proposition~\ref{prop:1} tells us that
$$
\sum_{\cI\subseteq[n],~
2n/\log^{1/3}(n)<|\cI|<n-2n/\log^{1/3}(n)} P_{\SIBM}(\sigma=X^{(\sim\cI)})  = o(1) .
$$
By definition, $P_{\SIBM}(\sigma=\bar{\sigma})=P_{\SIBM}(\sigma=-\bar{\sigma})$ for all $\bar{\sigma}\in\{\pm 1\}^n$. Therefore, 
$$
\sum_{\cI\subseteq[n],1\le |\cI|\le 2n/\log^{1/3}(n)} P_{\SIBM}(\sigma=-X^{(\sim\cI)}) =\sum_{\cI\subseteq[n],1\le |\cI|\le 2n/\log^{1/3}(n)} P_{\SIBM}(\sigma=X^{(\sim\cI)}).
$$
As a consequence, to prove Proposition~\ref{prop:tt}, we only need to show that
$$
\sum_{\cI\subseteq[n],1\le |\cI|\le 2n/\log^{1/3}(n)} P_{\SIBM}(\sigma=X^{(\sim\cI)}) 
= o(1) .
$$
This is further equivalent to proving that there exists a set $\cG$ such that

\noindent (i)
$P(G\in\cG)=1-o(1)$, where the probability is calculated according to the $\SSBM(n, a\log(n)/n, \linebreak[4] b\log(n)/n)$.

\noindent (ii)
For every $G\in\cG$,
$$
\sum_{\cI\subseteq[n],1\le |\cI|\le 2n/\log^{1/3}(n)}
\frac{P_{\sigma|G}(\sigma=X^{(\sim\cI)})}{P_{\sigma|G}(\sigma=X)} =o(1) .
$$
In order to prove the existence of such a set $\cG$, we define two functions
\begin{equation}  \label{eq:gbt}
\begin{aligned}
g(\beta) & := \frac{b e^{2\beta}+a e^{-2\beta}}{2}-\frac{a+b}{2}+1 , \\
\tilde{g}(\beta) & :=\left\{
\begin{array}{cc}
  g(\beta)   & \text{~if~} \beta< \frac{1}{4}\log\frac{a}{b} \\
  g(\frac{1}{4}\log\frac{a}{b})= \sqrt{ab}-\frac{a+b}{2}+1  & \text{~if~} \beta\ge \frac{1}{4}\log\frac{a}{b}
\end{array}
\right. ,
\end{aligned}
\end{equation}
and we will prove in Lemma~\ref{lm:ele} below (see the end of Section~\ref{sect:k=1}) that $\tilde{g}(\beta)<0$ under the conditions of Proposition~\ref{prop:tt} (i.e., $\sqrt{a}-\sqrt{b} > \sqrt{2}$ and $\beta>\beta^\ast$).
The existence of $\cG$ is guaranteed by the following proposition:
\begin{proposition} \label{prop:big}
Let $a,b,\alpha,\beta> 0$ be constants satisfying that $\sqrt{a}-\sqrt{b} > \sqrt{2}$, $\beta>\beta^\ast$ and $\alpha>b\beta$. 
Let 
$
(X,G,\sigma) \sim \SIBM(n,a\log(n)/n, b\log(n)/n,\alpha,\beta, 1) .
$
There is an integer $n_0$ such that for every even integer $n>n_0$ and  every integer $1\le k\le 2n/\log^{1/3}(n)$, there is a set $\cG^{(k)}$ for which

\noindent (i)
$P(G\in\cG^{(k)}) \ge 1- 2 n^{k\tilde{g}(\beta)/8}$ ,

\noindent (ii) For every $G\in\cG^{(k)}$,
$$
\sum_{\cI\subseteq[n],|\cI|=k}
\frac{P_{\sigma|G}(\sigma=X^{(\sim \cI)} )}
{P_{\sigma|G}(\sigma=X)} <
n^{k \tilde{g}(\beta) /2} .
$$
\end{proposition}
With the $\cG^{(k)}$'s given by Proposition~\ref{prop:big}, we
define 
$$
\cG:=\bigcap_{k=1}^{2n/\log^{1/3}(n)} \cG^{(k)} .
$$
By the union bound,
$$
P(G\in\cG)\ge 1-2 \sum_{k=1}^{2n/\log^{1/3}(n)} n^{k\tilde{g}(\beta)/8}
> 1- \frac{2 n^{\tilde{g}(\beta)/8}}{1-n^{\tilde{g}(\beta)/8}}
=1-o(1),
$$
where the last equality follows from $\tilde{g}(\beta)<0$. Moreover, for every $G\in\cG$,
\begin{align*}
& \sum_{\cI\subseteq[n],1\le |\cI|\le 2n/\log^{1/3}(n)}
\frac{P_{\sigma|G}(\sigma=X^{(\sim\cI)})}{P_{\sigma|G}(\sigma=X)} =
\sum_{k=1}^{2n/\log^{1/3}(n)}
\hspace*{0.05in}
\sum_{\cI\subseteq[n],|\cI|=k}
\frac{P_{\sigma|G}(\sigma=X^{(\sim \cI)} )}
{P_{\sigma|G}(\sigma=X)}  \\
& < \sum_{k=1}^{2n/\log^{1/3}(n)}
n^{k \tilde{g}(\beta) /2}
< \frac{n^{\tilde{g}(\beta)/2}}{1-n^{\tilde{g}(\beta)/2}} =o(1) .
\end{align*}
Thus we have shown that Proposition~\ref{prop:tt} is implied by Proposition~\ref{prop:big}. In the rest of this section, we will prove the latter proposition. In Section~\ref{sect:k=1}, we will prove Proposition~\ref{prop:big} for the special case of $k=1$ to illustrate the basic idea of the proof. Then we prove Proposition~\ref{prop:big} for general $k$ in Section~\ref{sect:gen}.

\subsection{Proof of Proposition~\ref{prop:big} for $k=1$} \label{sect:k=1}

Given the ground truth $X$, a graph $G$ and a vertex $i\in[n]$, define
\begin{equation} \label{eq:defAB}
\begin{aligned}
& A_i=A_i(G):=|\{j\in[n]\setminus\{i\}:\{i,j\}\in E(G), X_j=X_i\}| , \\
& B_i=B_i(G):=|\{j\in[n]\setminus\{i\}:\{i,j\}\in E(G), X_j=-X_i\}| .
\end{aligned}
\end{equation}
Next we give an upper bound on $P(B_i-A_i\ge t\log(n))$ for $t\in [\frac{1}{2}(b-a), 0]$. We take the left boundary to be $\frac{1}{2}(b-a)$ because $\frac{E[B_i-A_i]}{\log(n)}\to\frac{1}{2}(b-a)$ as $n\to\infty$.
\begin{proposition}  \label{prop:cher}
For $t\in [\frac{1}{2}(b-a), 0]$,
\begin{equation} \label{eq:upba}
\begin{aligned}
& P(B_i-A_i\ge t\log(n))  \\
\le &  \exp\Big(\log(n)
\Big(\sqrt{t^2+ab} -t\big(\log(\sqrt{t^2+ab}+t)-\log(b) \big) -\frac{a+b}{2} + O\big(\frac{\log(n)}{n}\big) \Big)\Big) .
\end{aligned}
\end{equation}
\end{proposition}

\begin{proof}
By definition,
$A_i\sim \Binom(\frac{n}{2}-1,\frac{a\log(n)}{n})$ and $B_i\sim \Binom(\frac{n}{2}, \frac{b\log(n)}{n})$, and they are independent.
The moment generating function of $B_i-A_i$ is
\begin{align*}
E[e^{s(B_i-A_i)}]
& =\Big(1-\frac{b\log(n)}{n}+\frac{b\log(n)}{n} e^s \Big)^{n/2}
\Big(1-\frac{a\log(n)}{n}+\frac{a\log(n)}{n} e^{-s} \Big)^{n/2-1}  \\
& = 
\exp\Big(\frac{\log(n)}{2} \Big( b e^s-b +O\big(\frac{\log(n)}{n}\big) \Big)\Big)
\exp\Big(\frac{\log(n)}{2} \Big( a e^{-s}-a + O\big(\frac{\log(n)}{n}\big) \Big) \Big)
 \\
& = 
\exp\Big(\frac{\log(n)}{2} \Big( a e^{-s}+b e^s-a-b + O\big(\frac{\log(n)}{n}\big) \Big)\Big) ,
\end{align*}
where we use the Taylor expansion $\log(1+x)=x+O(x^2)$ to obtain the second equality.
By Chernoff bound, for any $s\ge 0$, we have
\begin{equation} \label{eq:mmd}
\begin{aligned}
& P(B_i-A_i\ge t\log(n))\le
\frac{E[e^{s(B_i-A_i)}]}{e^{st\log(n)}}  \\
 \le & \exp\Big(\frac{\log(n)}{2} \Big( a e^{-s}+b e^s -2st -a-b + O\big(\frac{\log(n)}{n}\big) \Big)\Big)  .
 \end{aligned}
\end{equation}
Let $f(s):=a e^{-s}+be^s-2st$. We want to find $\min_{s\ge 0}f(s)$ to plug into the above upper bound. Since 
$f'(s)=-ae^{-s}+be^s-2t$ and 
$f''(s)=ae^{-s}+be^s>0$, $f(s)$ is a convex function and takes global minimum at $s^\ast$ such that $f'(s^\ast)=0$. Next we show that $s^\ast\ge 0$ for all $t\ge \frac{1}{2}(b-a)$, so $\min_{s\ge 0}f(s)=f(s^\ast)$. Indeed, this follows directly from the facts that $f'(0)=b-a-2t\le 0=f'(s^\ast)$ and that $f'(s)$ is an increasing function. Taking $s^\ast=\log(\sqrt{t^2+ab}+t)-\log(b)$ into \eqref{eq:mmd}, we obtain \eqref{eq:upba} for all $t\in[\frac{1}{2}(b-a), 0]$ and large enough $n$.
\end{proof}

Note that $A_i$ and $B_i$ are functions of the underlying graph $G$.
Given a graph $G$, define 
$$
\tilde{D}(G):=|\{i\in[n]:B_i- A_i\ge 0\}|
\text{~~and~~}
\tilde{D}_i(G):=\mathbbm{1}[B_i-A_i\ge 0] ,
$$
where $\mathbbm{1}[\cdot]$ is the indicator function. Then
$\tilde{D}(G)=\sum_{i=1}^n \tilde{D}_i(G)$ and
$$
E[\tilde{D}(G)]=\sum_{i=1}^n E[\tilde{D}_i(G)]
=\sum_{i=1}^n P(B_i- A_i\ge 0).
$$
Taking $t=0$ into \eqref{eq:upba}, we have
$P(B_i- A_i\ge 0)\le \exp\big(\log(n)(-\frac{(\sqrt{a}-\sqrt{b})^2}{2} +o(1)) \big)$. Therefore,
$$
E[\tilde{D}(G)]\le n \exp\Big(\log(n)\big(-\frac{(\sqrt{a}-\sqrt{b})^2}{2} +o(1) \big) \Big)
= n^{1-\frac{(\sqrt{a}-\sqrt{b})^2}{2} +o(1)}.
$$
By Markov inequality,
\begin{equation} \label{eq:tD}
P\big(\tilde{D}(G)=0 \big) = 1-
P\big(\tilde{D}(G)\ge 1\big) \ge 1- E[\tilde{D}(G)]
\ge 1- n^{1-\frac{(\sqrt{a}-\sqrt{b})^2}{2} +o(1)} .
\end{equation}
Since $\sqrt{a}-\sqrt{b} > \sqrt{2}$, we have
$P\big(\tilde{D}(G)=0 \big)= 1-o(1)$.

Let $X^{(\sim i)}$ be the vector obtained by flipping the $i$th coordinate of $X$ while keeping all the other coordinates to be the same, i.e., $X_i^{(\sim i)}=-X_i$ and $X_j^{(\sim i)}=X_j$ for all $j\neq i$.
Next we calculate the ratio
$$
\frac{\sum_{i=1}^n P_{\sigma|G}(\sigma=X^{(\sim i)} )}
{P_{\sigma|G}(\sigma=X)} .
$$
By \eqref{eq:isingma}, we have
\begin{align*}
\frac{P_{\sigma|G}(\sigma=X^{(\sim i)} )}
{P_{\sigma|G}(\sigma=X)}
& = \exp\Big(2\big(\beta+\frac{\alpha\log(n)}{n} \big) (B_i-A_i)
-\frac{2\alpha\log(n)}{n} \Big) \\
& \le \exp\Big(2\big(\beta+\frac{\alpha\log(n)}{n} \big) (B_i-A_i) \Big)  .
\end{align*}
Since $B_i-A_i$ takes integer value between $-n/2$ and $n/2$,
we can use indicator functions to write
\begin{align*}
& \exp\Big(2\big(\beta+\frac{\alpha\log(n)}{n} \big) (B_i-A_i) \Big) \\
= & \sum_{t\log(n)=-n/2}^{n/2}
\mathbbm{1}[B_i-A_i=t \log(n)] \exp\Big(2\big(\beta+\frac{\alpha\log(n)}{n} \big) t \log(n) \Big) ,
\end{align*}
where the quantity $t\log(n)$ ranges over all integer values from $-n/2$ to $n/2$ in the summation on the second line.
Define $D(G,t):=|\{i\in[n]:B_i-A_i= t\log(n)\}|$ and notice that $D(G,t)=\sum_{i=1}^n \mathbbm{1}[B_i-A_i=t \log(n)]$.
Therefore,
\begin{equation}  \label{eq:fd}
\begin{aligned}
 & \frac{\sum_{i=1}^n P_{\sigma|G}(\sigma=X^{(\sim i)} )}
{P_{\sigma|G}(\sigma=X)}
\le \sum_{i=1}^n \exp\Big(2\big(\beta+\frac{\alpha\log(n)}{n} \big) (B_i-A_i) \Big) \\
= & \sum_{i=1}^n
\hspace*{0.05in}
\sum_{t\log(n)=-n/2}^{n/2}
\mathbbm{1}[B_i-A_i=t \log(n)] \exp\Big(2\big(\beta+\frac{\alpha\log(n)}{n} \big) t \log(n) \Big) \\
= & \sum_{t\log(n)=-n/2}^{n/2}
D(G,t) \exp\Big(2\big(\beta+\frac{\alpha\log(n)}{n} \big) t \log(n) \Big)
\end{aligned}
\end{equation}
Define a set 
$$
\cG_1:=\{G:\tilde{D}(G)=0\}.
$$
By \eqref{eq:tD}, $P(G\in\cG_1)= 1-o(1)$. By definition of $\tilde{D}(G)$, $G\in\cG_1$ implies that $D(G,t)=0$ for all $t\ge 0$. Therefore, for $G\in\cG_1$, we have
\begin{equation} \label{eq:lq}
\begin{aligned}
& \frac{\sum_{i=1}^n P_{\sigma|G}(\sigma=X^{(\sim i)} )}
{P_{\sigma|G}(\sigma=X)}
\le \sum_{t\log(n)=-n/2}^{-1}
D(G,t) \exp\Big(2\big(\beta+\frac{\alpha\log(n)}{n} \big) t \log(n) \Big)   \\
\overset{(a)}{\le} & \sum_{t\log(n)=-n/2}^{-1}
D(G,t) \exp\big(2\beta t \log(n) \big) \\
= & \sum_{t\log(n)=-n/2}^{\lfloor\frac{b-a}{2}\log(n) \rfloor} D(G,t)
\exp\big(2\beta t \log(n) \big)  + \sum_{t\log(n)=\lceil\frac{b-a}{2}\log(n) \rceil}^{-1} D(G,t)
\exp\big(2\beta t \log(n) \big) \\
\le & \sum_{t\log(n)=-n/2}^{\lfloor\frac{b-a}{2}\log(n) \rfloor} D(G,t)
\exp\big(\beta(b-a)\log(n) \big)  + \sum_{t\log(n)=\lceil\frac{b-a}{2}\log(n) \rceil}^{-1} D(G,t)
\exp\big(2\beta t \log(n) \big)  \\
\overset{(b)}{\le} & n
\exp\big(\beta(b-a)\log(n) \big)  + \sum_{t\log(n)=\lceil\frac{b-a}{2}\log(n) \rceil}^{-1} D(G,t)
\exp\big(2\beta t \log(n) \big) ,
\end{aligned}
\end{equation}
where inequality $(a)$ holds because $t\log(n)$ only takes negative values in the summation, and inequality $(b)$ follows from the trivial upper bound
$\sum_{t\log(n)=-n/2}^{\lfloor\frac{b-a}{2}\log(n) \rfloor} D(G,t)\le n$.
Define a function
\begin{equation} \label{eq:gt}
f_{\beta}(t):=\sqrt{t^2+ab} -t\big(\log(\sqrt{t^2+ab}+t)-\log(b) \big) -\frac{a+b}{2} +1 +2\beta t.
\end{equation}
Then for $t\in [\frac{1}{2}(b-a), 0]$, we have 
\begin{align*}
& E[D(G,t)
\exp\big(2\beta t \log(n) \big)] \\
= & \sum_{i=1}^n E[\mathbbm{1}[B_i-A_i=t \log(n)]] \exp\big(2\beta t \log(n) \big) \\
= & \sum_{i=1}^n P\big(B_i-A_i=t \log(n) \big) \exp\big(2\beta t \log(n) \big) \\
\le & \sum_{i=1}^n P\big(B_i-A_i \ge t \log(n) \big) \exp\big(2\beta t \log(n) \big) \\
\le &  \sum_{i=1}^n \exp\Big(\log(n)
\Big(\sqrt{t^2+ab} -t\big(\log(\sqrt{t^2+ab}+t)-\log(b) \big) -\frac{a+b}{2} +2\beta t +o(1)\Big)\Big) \\
= &  n^{f_{\beta}(t) +o(1)} ,
\end{align*}
where the second inequality follows from \eqref{eq:upba}.
For $\epsilon>0$, define a set
$$
\cG(\epsilon):=\left\{
\sum_{t\log(n)=\lceil\frac{b-a}{2}\log(n) \rceil}^{-1} D(G,t)
\exp\big(2\beta t \log(n) \big)
\le \sum_{t\log(n)=\lceil\frac{b-a}{2}\log(n) \rceil}^{-1}
 n^{f_{\beta}(t) + \epsilon}
\right\} .
$$
Then by Markov inequality,
\begin{equation} \label{eq:mk}
P(G\in \cG(\epsilon))\ge 1-n^{-(\epsilon-o(1))} >
1-n^{-\epsilon/2}
\end{equation}
for large $n$ and positive $\epsilon$.
Using \eqref{eq:lq}, we obtain that for $G\in\cG_1\cap\cG(\epsilon)$,
\begin{equation}  \label{eq:zz}
\begin{aligned}
\frac{\sum_{i=1}^n P_{\sigma|G}(\sigma=X^{(\sim i)} )}
{P_{\sigma|G}(\sigma=X)} 
& \le n
\exp\big(\beta(b-a)\log(n) \big)  + \sum_{t\log(n)=\lceil\frac{b-a}{2}\log(n) \rceil}^{-1}  n^{f_{\beta}(t) + \epsilon} \\
& = n^{f_{\beta}((b-a)/2)} + \sum_{t\log(n)=\lceil\frac{b-a}{2}\log(n) \rceil}^{-1}  n^{f_{\beta}(t) + \epsilon} ,
\end{aligned}
\end{equation}
where the equality follows from the fact that $f_{\beta}(\frac{1}{2}(b-a))=\beta(b-a)+1$.
Recall the definitions of the functions $g(\beta)$ and $\tilde{g}(\beta)$ in \eqref{eq:gbt}.
By Lemma~\ref{lm:tus} below, we have
$f_{\beta}(t)\le \tilde{g}(\beta)<0$ for all $t\le 0$.
Combining this with \eqref{eq:zz}, we obtain that for
$G\in\cG_1\cap\cG(\epsilon)$,
\begin{equation} \label{eq:lh}
\begin{aligned}
& \frac{\sum_{i=1}^n P_{\sigma|G}(\sigma=X^{(\sim i)} )}
{P_{\sigma|G}(\sigma=X)} 
\le n^{\tilde{g}(\beta)}
+ \sum_{t\log(n)=\lceil\frac{b-a}{2}\log(n) \rceil}^{-1}  n^{\tilde{g}(\beta) + \epsilon}  \\
& \le n^{\tilde{g}(\beta) + \epsilon}
\big(\frac{a-b}{2}\log(n)+1 \big)
< n^{\tilde{g}(\beta) + 2\epsilon} 
\end{aligned}
\end{equation}
for large $n$ and positive $\epsilon$.
Let $\epsilon=-\tilde{g}(\beta)/4>0$ and define
$$
\cG^{(1)}:=\cG_1\cap\cG(-\tilde{g}(\beta)/4) .
$$
By \eqref{eq:lh}, for $G\in\cG^{(1)}$ we have
$$
\frac{\sum_{i=1}^n P_{\sigma|G}(\sigma=X^{(\sim i)} )}
{P_{\sigma|G}(\sigma=X)} <
n^{\tilde{g}(\beta) /2} .
$$
By \eqref{eq:tD} and \eqref{eq:mk}, 
$$
P(G\in\cG^{(1)})\ge 1-n^{\tilde{g}(\beta)/8}- n^{1-\frac{(\sqrt{a}-\sqrt{b})^2}{2}+o(1)}> 1- 2n^{\tilde{g}(\beta)/8},
$$
where the last inequality follows from the fact that $1-\frac{(\sqrt{a}-\sqrt{b})^2}{2}\le
\tilde{g}(\beta)<\tilde{g}(\beta)/8< 0$.
This completes the proof of Proposition~\ref{prop:big} for the special case of $k=1$. 
Next we prove the two auxiliary lemmas used above.
\begin{lemma}[Elementary properties of $\beta^\ast$ defined in \eqref{eq:defstar}] \label{lm:ele}
Let $g(\beta)$ and $\tilde{g}(\beta)$ be the functions defined in \eqref{eq:gbt}. Assume that $\sqrt{a}-\sqrt{b}>\sqrt{2}$.
Then,

\begin{enumerate}[label=(\roman*)]
\item The equation $g(\beta) = 0$ has two roots, and the smaller one of them is $\beta^\ast$.

\item Denote the other root as $\beta'$. Then
$\beta^\ast< \frac{1}{4}\log\frac{a}{b} <\beta'$.

\item $g(\beta)<0$ for all $\beta^\ast< \beta \le \frac{1}{4}\log\frac{a}{b}$.

\item $\tilde{g}(\beta)<0$ for all $\beta>\beta^\ast$.

\item $\tilde{g}(\beta)$ is a decreasing function in $[0,+\infty)$. 

\item $\tilde{g}(\beta)<1$ for all $\beta>0$.
\end{enumerate}
\end{lemma}
\begin{proof}
{\bf Proof of (i)}:
We write $x=e^{2\beta}$. Then $g(\beta)=0$ can be written as $bx^2-(a+b-2)x+a=0$. This quadratic equation has two roots if and only if $(a+b-2)^2-4ab>0$, which is guaranteed by the assumption $\sqrt{a}-\sqrt{b}>\sqrt{2}$.
The two roots of $bx^2-(a+b-2)x+a=0$ are
$x^\ast=\frac{a+b-2 - \sqrt{(a+b-2)^2-4ab}}{2b}$ and $x'=\frac{a+b-2 + \sqrt{(a+b-2)^2-4ab}}{2b}$.
Therefore, $\beta^\ast=\frac{1}{2}\log(x^\ast)$ and $\beta'=\frac{1}{2}\log(x')$.
{\bf Proof of (ii)}:
Since $x^\ast x'=\frac{a}{b}=(\sqrt{\frac{a}{b}})^2$, we have $x^\ast<\sqrt{\frac{a}{b}}<x'$, so $\beta^\ast< \frac{1}{4}\log\frac{a}{b} <\beta'$.
{\bf Proof of (iii)}:
Since $b>0$, $bx^2-(a+b-2)x+a<0$ if and only if $x^\ast<x<x'$. Therefore, $g(\beta)<0$ if and only if $\beta^\ast< \beta <\beta'$. This implies (iii). 
{\bf Proof of (iv)}: (iv) follows directly from (iii).
{\bf Proof of (v)}: $g'(\beta)=b e^{2\beta} - a e^{-2\beta}$, so $g'(\beta)<0$ for $0\le \beta<\frac{1}{4}\log\frac{a}{b}$, and $g(\beta)$ takes minimum value at $\beta=\frac{1}{4}\log\frac{a}{b}$. This implies (v).
{\bf Proof of (vi)}: (vi) follows directly from (v) and the fact that $\tilde{g}(0)=1$.
\end{proof}

\begin{lemma} \label{lm:tus}
Let $f_{\beta}(t)$ be the function defined in \eqref{eq:gt}. If $a>b>0$, then $f_{\beta}(t)\le \tilde{g}(\beta)$ for all $t\le 0$.
If $\sqrt{a}-\sqrt{b}>\sqrt{2}$
and $\beta>\beta^\ast$, then we further have $f_{\beta}(t)\le \tilde{g}(\beta)<0$ for all $t\le 0$.
\end{lemma}
\begin{proof}
The first and second derivatives are
$f_{\beta}'(t)=
-\log(\sqrt{t^2+ab}+t)+\log(b) +2\beta$ and $f_{\beta}''(t)=-\frac{1}{\sqrt{t^2+ab}}<0$.
Therefore $f_{\beta}(t)$ is a concave function and takes global maximum at $t^\ast$ such that $f_{\beta}'(t^\ast)=0$.
Simple calculation shows that 
$$
t^\ast=\frac{b e^{2\beta}-a e^{-2\beta}}{2} \quad \text{~and~} \quad
f_{\beta}(t^\ast)=\frac{b e^{2\beta}+a e^{-2\beta}}{2}-\frac{a+b}{2}+1
=g(\beta) .
$$
We divide the proof into two cases. {\bf Case 1}: If $\beta\ge \frac{1}{4}\log\frac{a}{b}$, then $t^\ast\ge 0$. Since $f_{\beta}(t)$ is an increasing function for $t\le t^\ast$, we have $f_{\beta}(t)\le f_{\beta}(0)=\sqrt{ab}-\frac{a+b}{2}+1$ for all $t\le 0$.
{\bf Case 2}: If $\beta< \frac{1}{4}\log\frac{a}{b}$, then we simply use the global maximum $f_{\beta}(t^\ast)$ to upper bound $f_{\beta}(t)$, i.e., $f_{\beta}(t)\le f_{\beta}(t^\ast)=g(\beta)$ for all $t$.
Combining these two cases, we have $f_{\beta}(t)\le \tilde{g}(\beta)$ for all $t\le 0$ as long as $a>b>0$.
If $\sqrt{a}-\sqrt{b}>\sqrt{2}$ and $\beta>\beta^\ast$, then by property (iv) of Lemma~\ref{lm:ele} we further have $f_{\beta}(t)\le \tilde{g}(\beta)<0$ for all $t\le 0$.
\end{proof}

\subsection{Proof of Proposition~\ref{prop:big} for general $k$} \label{sect:gen}
Recall that for $\cI\subseteq[n]$, we define $X^{(\sim\cI)}$ as the vector obtained by flipping the coordinates in $\cI$ while keeping all the other coordinates to be the same as $X$, i.e., $X_i^{(\sim\cI)}=-X_i$ for all $i\in\cI$ and $X_i^{(\sim\cI)}=X_i$ for all $i\notin\cI$. 
We want to bound the ratio
$$
\sum_{\cI\subseteq[n]:|\cI|=k}
\frac{P_{\sigma|G}(\sigma=X^{(\sim\cI)})}{P_{\sigma|G}(\sigma=X)}
$$
for all $k\le 2n/\log^{1/3}(n)$.
To that end,
for $\cI\subseteq[n]$,
define the positive and negative parts of $\cI$ as
$$
\cI_+:=\{i\in\cI:X_i=+1\}
\quad \text{~and~} \quad
\cI_-:=\{i\in\cI:X_i=-1\} ,
$$
and define 
$$
\nabla \cI:=\{\{i,j\}:i\in\cI,j\in[n]\setminus\cI\} .
$$
We further define
$$
\nabla\cI_+:=\{\{i,j\}\in\nabla\cI:X_i=X_j\}
\quad \text{and} \quad
\nabla\cI_-:=\{\{i,j\}\in\nabla\cI:X_i=-X_j\}.
$$
Then 
\begin{equation} \label{eq:partial}
\begin{aligned}
& |\nabla\cI_+|=|\cI_+|(\frac{n}{2}-|\cI_+|) + |\cI_-|(\frac{n}{2}-|\cI_-|)
=\frac{n}{2}|\cI|-|\cI_+|^2-|\cI_-|^2  , \\
& |\nabla\cI_-|=|\cI_+|(\frac{n}{2}-|\cI_-|) + |\cI_-|(\frac{n}{2}-|\cI_+|) =\frac{n}{2}|\cI|
-2 |\cI_+| |\cI_-| .
\end{aligned}
\end{equation}
Given a graph $G$,
define
\begin{align*}
& A_{\cI}=A_{\cI}(G):=|\{\{i,j\}\in\nabla\cI\cap E(G):X_i=X_j\}| ,  \\
& B_{\cI}=B_{\cI}(G):=|\{\{i,j\}\in\nabla\cI\cap E(G):X_i=-X_j\}| .
\end{align*}

\begin{proposition}
For $t\in [\frac{1}{2}(b-a), 0]$
and $|\cI|\le 2n/\log^{1/3}(n)$,
\begin{equation} \label{eq:upmpt}
\begin{aligned}
& P(B_{\cI}-A_{\cI}\ge t|\cI|\log(n))  \\
\le & \exp\Big(|\cI|\log(n)
\Big(\sqrt{t^2+ab} -t\big(\log(\sqrt{t^2+ab}+t)-\log(b) \big) -\frac{a+b}{2} 
+ O(\log^{-1/3}(n)) \Big)\Big) .
\end{aligned}
\end{equation}
\end{proposition}
\begin{proof}
By definition, $A_{\cI}\sim\Binom(|\nabla\cI_+|,\frac{a\log(n)}{n})$ and $B_{\cI}\sim\Binom(|\nabla\cI_-|,\frac{b\log(n)}{n})$, and they are independent. For $s>0$, the moment generating function of $B_{\cI}-A_{\cI}$ for $|\cI|\le 2n/\log^{1/3}(n)$ can be bounded from above as follows:
\begin{align*}
& E[e^{s(B_{\cI}-A_{\cI})}] \\
& =\Big(1-\frac{b\log(n)}{n}+\frac{b\log(n)}{n} e^s \Big)^{n|\cI|/2
-2 |\cI_+| |\cI_-|}
\Big(1-\frac{a\log(n)}{n}+\frac{a\log(n)}{n} e^{-s} \Big)^{n |\cI|/2-|\cI_+|^2-|\cI_-|^2}  \\
& \le 
\Big(1-\frac{b\log(n)}{n}+\frac{b\log(n)}{n} e^s \Big)^{n|\cI|/2}
\Big(1-\frac{a\log(n)}{n}+\frac{a\log(n)}{n} e^{-s} \Big)^{n |\cI|/2-|\cI|^2}
 \\
& \le
\exp\Big(\frac{|\cI|\log(n)}{2}(a e^{-s}+b e^s-a-b +\frac{2a|\cI|}{n})
+|\cI|O(\frac{\log^2(n)}{n})\Big) \\
& =
\exp\Big(\frac{|\cI|\log(n)}{2}(a e^{-s}+b e^s-a-b +
O(\log^{-1/3}(n))) \Big),
\end{align*}
where the first inequality follows from $1-\frac{b\log(n)}{n}+\frac{b\log(n)}{n} e^s>1$ and $1-\frac{a\log(n)}{n}+\frac{a\log(n)}{n} e^{-s}<1$; in the second inequality we use the Taylor expansion $\log(1+x)=x+O(x^2)$; the last equality follows from the assumption that $|\cI|\le 2n/\log^{1/3}(n)$.
By Chernoff bound, for $s>0$, we have
\begin{align*} 
& P(B_{\cI}-A_{\cI}\ge t|{\cI}|\log(n))\le
\frac{E[e^{s(B_{\cI}-A_{\cI})}]}{e^{st|{\cI}|\log(n)}}  \\
\le & \exp\Big(\frac{{\cI}\log(n)}{2} \big(a e^{-s}+b e^s -2st -a-b
 + O(\log^{-1/3}(n)) \big)\Big)  .
\end{align*}
The rest of the proof is to find $s^\ast$ to minimize $a e^{-s}+b e^s -2st$ and take $s^\ast$ into the above bound. This is exactly the same as the proof of \eqref{eq:upba}, and we do not repeat it here.
\end{proof}

Given a graph $G$, define 
$$
\tilde{D}^{(k)}(G):=|\{\cI\subseteq[n],|\cI|=k:B_{\cI}- A_{\cI}\ge 0\}|
\text{~~and~~}
\tilde{D}_{\cI}(G):=\mathbbm{1}[B_{\cI}-A_{\cI}\ge 0] .
$$
 Then
$\tilde{D}^{(k)}(G)=\sum_{\cI\subseteq[n],|\cI|=k} \tilde{D}_{\cI}(G)$ and
$$
E[\tilde{D}^{(k)}(G)]=\sum_{\cI\subseteq[n],|\cI|=k} E[\tilde{D}_{\cI}(G)]
=\sum_{\cI\subseteq[n],|\cI|=k} P(B_{\cI}- A_{\cI}\ge 0).
$$
Taking $t=0$ into \eqref{eq:upmpt}, we have
$P(B_{\cI}- A_{\cI}\ge 0)\le \exp\big(|\cI|\log(n)(-\frac{(\sqrt{a}-\sqrt{b})^2}{2} + o(1) ) \big)$. Therefore,
\begin{align*}
& E[\tilde{D}^{(k)}(G)]\le \binom{n}{k} \exp\Big(k \log(n)\big(-\frac{(\sqrt{a}-\sqrt{b})^2}{2} + o(1) \big) \Big) \\
\le & n^k \exp\Big(k \log(n)\big(-\frac{(\sqrt{a}-\sqrt{b})^2}{2} + o(1) \big) \Big)
= n^{k( 1-\frac{(\sqrt{a}-\sqrt{b})^2}{2} + o(1) )}.
\end{align*}
By Markov inequality,
\begin{equation} \label{eq:Dk}
P\big(\tilde{D}^{(k)}(G)=0 \big) = 1-
P\big(\tilde{D}^{(k)}(G)\ge 1\big) \ge 1- E[\tilde{D}^{(k)}(G)]
\ge 1-  n^{k (1-\frac{(\sqrt{a}-\sqrt{b})^2}{2} + o(1) )} .
\end{equation}
Since $\sqrt{a}-\sqrt{b} > \sqrt{2}$, we have
$P\big(\tilde{D}^{(k)}(G)=0 \big)= 1-o(1)$.

By \eqref{eq:isingma}, we have
\begin{equation} \label{eq:ts}
\begin{aligned}
& \frac{P_{\sigma|G}(\sigma=X^{(\sim\cI)})}{P_{\sigma|G}(\sigma=X)} \\
= & \exp\Big(\beta\sum_{\{i,j\}\in E(G)} (X_i^{(\sim\cI)} X_j^{(\sim\cI)}
-X_i X_j)
-\frac{\alpha\log(n)}{n} \sum_{\{i,j\}\notin E(G)} (X_i^{(\sim\cI)} X_j^{(\sim\cI)}
-X_i X_j) \Big) \\
\overset{(a)}{=} & \exp\Big(-2 \beta\sum_{\{i,j\}\in E(G)} X_i X_j \mathbbm{1}[\{i,j\}\in\nabla\cI]
+\frac{2\alpha\log(n)}{n} \sum_{\{i,j\}\notin E(G)} X_i X_j \mathbbm{1}[\{i,j\}\in\nabla\cI] \Big) \\
= & \exp\Big(-2 \beta(A_{\cI} - B_{\cI})
+\frac{2\alpha\log(n)}{n} \big( (|\nabla\cI_+|-A_{\cI}) -(|\nabla\cI_-|-B_{\cI}) \big)\Big) \\
\overset{(b)}{=} & \exp\Big( 2\big(\beta+\frac{\alpha\log(n)}{n} \big) (B_{\cI}-A_{\cI})
- \frac{2\alpha\log(n)}{n} (|\cI_+|-|\cI_-|)^2
\Big) \\
\le & \exp\Big( 2\big(\beta+\frac{\alpha\log(n)}{n} \big) (B_{\cI}-A_{\cI})
\Big) ,
\end{aligned}
\end{equation}
where $(a)$ follows from the fact that $X_i^{(\sim\cI)} X_j^{(\sim\cI)}=-X_i X_j$ if $\{i,j\}\in\nabla\cI$ and $X_i^{(\sim\cI)} X_j^{(\sim\cI)}=X_i X_j$ if $\{i,j\}\notin\nabla\cI$;
and $(b)$ follows from \eqref{eq:partial}.
Since $B_{\cI}-A_{\cI}$ takes integer value between $-|\cI|n/2$ and $|\cI|n/2$,
we can use indicator functions to write
\begin{align*}
& \exp\Big(2\big(\beta+\frac{\alpha\log(n)}{n} \big) (B_{\cI}-A_{\cI}) \Big) \\
= & \sum_{t|\cI|\log(n)=-|\cI|n/2}^{|\cI|n/2}
\mathbbm{1}[B_{\cI}-A_{\cI}=t|\cI| \log(n)] \exp\Big(2\big(\beta+\frac{\alpha\log(n)}{n} \big) t|\cI| \log(n) \Big) ,
\end{align*}
where the quantity $t|\cI|\log(n)$ ranges over all integer values from $-|\cI|n/2$ to $|\cI|n/2$ in the summation on the second line.
Define $D^{(k)}(G,t):=|\{\cI\subseteq[n],|\cI|=k:B_{\cI}-A_{\cI}= tk\log(n)\}|$ and notice that $D^{(k)}(G,t)=\sum_{\cI\subseteq[n],|\cI|=k} \mathbbm{1}[B_{\cI}-A_{\cI}= tk\log(n)]$.
Therefore,
\begin{align*}
 & \sum_{\cI\subseteq[n],|\cI|=k}
 \frac{P_{\sigma|G}(\sigma=X^{(\sim \cI)} )}
{P_{\sigma|G}(\sigma=X)}
\le \sum_{\cI\subseteq[n],|\cI|=k} \exp\Big(2\big(\beta+\frac{\alpha\log(n)}{n} \big) (B_{\cI}-A_{\cI}) \Big) \\
= & \sum_{\cI\subseteq[n],|\cI|=k}
\hspace*{0.05in}
\sum_{tk\log(n)=-kn/2}^{kn/2}
\mathbbm{1}[B_{\cI}-A_{\cI}=tk \log(n)] \exp\Big(2\big(\beta+\frac{\alpha\log(n)}{n} \big) tk \log(n) \Big) \\
= & \sum_{tk\log(n)=-kn/2}^{kn/2}
D^{(k)}(G,t) \exp\Big(2\big(\beta+\frac{\alpha\log(n)}{n} \big) t k \log(n) \Big)
\end{align*}
Define a set 
\begin{equation} \label{eq:g1k}
\cG_1^{(k)}:=\{G:\tilde{D}^{(k)}(G)=0\}.
\end{equation}
By \eqref{eq:Dk}, $P(G\in\cG_1^{(k)})= 1-o(1)$. By definition of $\tilde{D}^{(k)}(G)$, $G\in\cG_1^{(k)}$ implies that $D^{(k)}(G,t)=0$ for all $t\ge 0$. Therefore, for $G\in\cG_1^{(k)}$, we have
\begin{equation} \label{eq:bk}
\begin{aligned}
& \sum_{\cI\subseteq[n],|\cI|=k}
\frac{P_{\sigma|G}(\sigma=X^{(\sim \cI)} )}
{P_{\sigma|G}(\sigma=X)} \\
\le & \sum_{tk\log(n)=-kn/2}^{-1}
D^{(k)}(G,t) \exp\Big(2\big(\beta+\frac{\alpha\log(n)}{n} \big) t k \log(n) \Big)   \\
\overset{(a)}{\le} & \sum_{tk\log(n)=-kn/2}^{-1}
D^{(k)}(G,t) \exp\big(2\beta t k \log(n) \big) \\
= & \sum_{tk\log(n)=-kn/2}^{\lfloor\frac{b-a}{2}k\log(n) \rfloor} D^{(k)}(G,t) \exp\big(2\beta t k \log(n) \big) \\
& \hspace*{1.5in} + \sum_{tk\log(n)=\lceil\frac{b-a}{2}k\log(n) \rceil}^{-1} D^{(k)}(G,t) \exp\big(2\beta t k \log(n) \big) \\
\le & \sum_{tk\log(n)=-kn/2}^{\lfloor\frac{b-a}{2}k\log(n) \rfloor} D^{(k)}(G,t)
\exp\big(\beta(b-a)k\log(n) \big) \\
& \hspace*{1.5in} + \sum_{tk\log(n)=\lceil\frac{b-a}{2}k\log(n) \rceil}^{-1} D^{(k)}(G,t) \exp\big(2\beta t k \log(n) \big) \\
\overset{(b)}{\le} & \binom{n}{k}
\exp\big(\beta(b-a) k \log(n) \big)  + \sum_{tk\log(n)=\lceil\frac{b-a}{2}k\log(n) \rceil}^{-1} D^{(k)}(G,t) \exp\big(2\beta t k \log(n) \big) ,
\end{aligned}
\end{equation}
where inequality $(a)$ holds because $tk\log(n)$ only takes negative values in the summation, and inequality $(b)$ follows from the trivial upper bound
$\sum_{tk\log(n)=-kn/2}^{\lfloor\frac{b-a}{2}k\log(n) \rfloor} D^{(k)}(G,t)\le \binom{n}{k}$.

Recall the function $f_{\beta}(t)$ defined in \eqref{eq:gt}.
Then for $t\in [\frac{1}{2}(b-a), 0]$, we have 
\begin{align*}
& E[D^{(k)}(G,t) \exp\big(2\beta t k \log(n) \big)] \\
= & \sum_{\cI\subseteq[n],|\cI|=k} E[\mathbbm{1}[B_{\cI}-A_{\cI}=t k \log(n)]] \exp\big(2\beta t k \log(n) \big) \\
= & \sum_{\cI\subseteq[n],|\cI|=k} P\big(B_{\cI}-A_{\cI}=t k \log(n) \big) \exp\big(2\beta t k \log(n) \big) \\
\le & \sum_{\cI\subseteq[n],|\cI|=k} P\big(B_{\cI}-A_{\cI} \ge t k \log(n) \big) \exp\big(2\beta t k \log(n) \big) \\
\le &  \sum_{\cI\subseteq[n],|\cI|=k} \exp\Big( k \log(n)
\Big(\sqrt{t^2+ab} -t\big(\log(\sqrt{t^2+ab}+t)-\log(b) \big) -\frac{a+b}{2} +2\beta t +o(1) \Big)\Big) \\
= & \binom{n}{k} n^{k(f_{\beta}(t)-1+o(1))} ,
\end{align*}
where the second inequality follows from \eqref{eq:upmpt}.
For $\epsilon>0$, define a set
\begin{equation} \label{eq:gep}
\begin{aligned}
\cG^{(k)}(\epsilon):=\left\{
\sum_{tk\log(n)=\lceil\frac{b-a}{2}k\log(n) \rceil}^{-1} D^{(k)}(G,t) \exp  \big(2\beta t k \log(n) \big) \hspace*{1.5in} \right. \\
\left.
 \le  \sum_{tk\log(n)=\lceil\frac{b-a}{2}k\log(n) \rceil}^{-1}
\binom{n}{k} n^{k(f_{\beta}(t)-1+\epsilon)}
\right\} .
\end{aligned}
\end{equation}
Then by Markov inequality,
\begin{equation} \label{eq:Gk}
P(G\in \cG^{(k)}(\epsilon))\ge 1-n^{-k(\epsilon-o(1))} >1-n^{-k\epsilon/2}
\end{equation}
for large $n$ and positive $\epsilon$.
Using \eqref{eq:bk}, we obtain that for $G\in\cG_1^{(k)}\cap\cG^{(k)}(\epsilon)$,
\begin{equation}  \label{eq:3l}
\begin{aligned}
& \sum_{\cI\subseteq[n],|\cI|=k}
\frac{P_{\sigma|G}(\sigma=X^{(\sim \cI)} )}
{P_{\sigma|G}(\sigma=X)}  \\
\le & \binom{n}{k}
\exp\big(\beta(b-a) k \log(n) \big)  + \sum_{tk\log(n)=\lceil\frac{b-a}{2}k\log(n) \rceil}^{-1}
\binom{n}{k} n^{k(f_{\beta}(t)-1+\epsilon)} \\
= & \binom{n}{k} n^{k(f_{\beta}((b-a)/2)-1)} + \sum_{tk\log(n)=\lceil\frac{b-a}{2}k\log(n) \rceil}^{-1}
\binom{n}{k} n^{k(f_{\beta}(t)-1+\epsilon)} ,
\end{aligned}
\end{equation}
where the equality follows from the fact that $f_{\beta}(\frac{1}{2}(b-a))=\beta(b-a)+1$.
Recall the function $\tilde{g}(\beta)$ defined in \eqref{eq:gbt}, and recall from  Lemma~\ref{lm:tus} that
$f_{\beta}(t)\le \tilde{g}(\beta)<0$ for all $t\le 0$.
Using this in \eqref{eq:3l} together with the fact $\binom{n}{k}<n^k$, we obtain that for
$G\in\cG_1^{(k)}\cap\cG^{(k)}(\epsilon)$,
\begin{equation} \label{eq:wuh}
\begin{aligned}
& \sum_{\cI\subseteq[n],|\cI|=k}
\frac{P_{\sigma|G}(\sigma=X^{(\sim \cI)} )}
{P_{\sigma|G}(\sigma=X)} 
\le n^{k\tilde{g}(\beta)}
+ \sum_{tk\log(n)=\lceil\frac{b-a}{2}k\log(n) \rceil}^{-1}  n^{k(\tilde{g}(\beta) + \epsilon)}  \\
& \le n^{k(\tilde{g}(\beta) + \epsilon)}
\Big(\frac{a-b}{2}\log(n^k)+1 \Big)
< n^{k(\tilde{g}(\beta) + 2\epsilon)} 
\end{aligned}
\end{equation}
for large $n$ and positive $\epsilon$.
Let $\epsilon=-\tilde{g}(\beta)/4>0$ and define
$$
\cG^{(k)}:=\cG_1^{(k)}\cap\cG^{(k)}(-\tilde{g}(\beta)/4) .
$$
By \eqref{eq:wuh}, for $G\in\cG^{(k)}$ we have
$$
\sum_{\cI\subseteq[n],|\cI|=k}
\frac{P_{\sigma|G}(\sigma=X^{(\sim \cI)} )}
{P_{\sigma|G}(\sigma=X)} <
n^{k \tilde{g}(\beta) /2} .
$$
By \eqref{eq:Dk}, \eqref{eq:g1k} and \eqref{eq:Gk}, 
$$
P(G\in\cG^{(k)})\ge 1-n^{k\tilde{g}(\beta)/8}- n^{k (1-\frac{(\sqrt{a}-\sqrt{b})^2}{2} + o(1) )}> 1- 2 n^{k\tilde{g}(\beta)/8},
$$
where the last inequality follows from the fact that $1-\frac{(\sqrt{a}-\sqrt{b})^2}{2}\le
\tilde{g}(\beta)<\tilde{g}(\beta)/8< 0$.

\section{Samples differ from $\pm X$ in $O(n^{\theta})$ coordinates for some $\theta<1$ when $\beta\le\beta^\ast$}
\label{sect:theta}

\begin{proposition}[Refinement of Proposition~\ref{prop:1}] \label{prop:43}
Let $a,b,\alpha,\beta> 0$ be constants satisfying that $\sqrt{a}-\sqrt{b} > \sqrt{2}$, $\alpha>b\beta$ and $\beta\le\beta^\ast$. Let $m$ be a constant integer that is independent of $n$.
Let 
$$
(X,G,\{\sigma^{(1)},\dots,\sigma^{(m)}\})\sim \SIBM(n,a\log(n)/n, b\log(n)/n,\alpha,\beta, m) .
$$
Then for any (arbitrarily small) $\delta>0$ and any (arbitrarily large) $r>0$, there exists $n_0(\delta, r)$ such that for all even integers $n>n_0(\delta, r)$,
$$
P_{\SIBM} \Big(\dist(\sigma^{(i)},\pm X)<n^{g(\beta)+\delta}
\text{~for all~} i\in[m] \Big) \ge 1- n^{-r} .
$$
\end{proposition}

By Lemma~\ref{lm:ele} (vi), we know that $g(\beta)<1$ for all $0<\beta\le\beta^\ast$, so we can always choose a $\delta>0$ such that $g(\beta)+\delta<1$.
Then Proposition~\ref{prop:43} implies that when $\beta\le\beta^\ast$, with probability $1-o(1)$ all samples differ from $\pm X$ in $O(n^\theta)$ coordinates for some $\theta<1$.

Since we assume that $m$ is a constant that is independent of $n$,
we only need to prove Proposition~\ref{prop:43} for the special case of $m=1$, and the case of general values of $m$ follows immediately from this special case.
Proposition~\ref{prop:1} tells us that if $\alpha>b\beta$, then $\dist(\sigma,\pm X) \le 2n/\log^{1/3}(n)$
 with probability $1-n^{-2r}$ for any given $r>0$ and large enough $n$.
Also note that
\begin{align*}
 \sum_{\cI\subseteq[n],~
n^{g(\beta)+\delta}
\le |\cI|\le 2n/\log^{1/3}(n)}
\frac{P_{\sigma|G}(\sigma=-X^{(\sim\cI)})}{P_{\sigma|G}(\sigma=-X)} 
=  \sum_{\cI\subseteq[n],~
n^{g(\beta)+\delta}
\le |\cI|\le 2n/\log^{1/3}(n)}
\frac{P_{\sigma|G}(\sigma=X^{(\sim\cI)})}{P_{\sigma|G}(\sigma=X)} .
\end{align*}
Therefore, to prove Proposition~\ref{prop:43}, we only need to show that given $r>0$, there exists a set $\cG_{\delta}$ such that the following two conditions hold for large enough $n$:  (i)
$P(G\in\cG_{\delta})\ge 1-n^{-2r}$, and (ii)
For every $G\in\cG_{\delta}$,
$$
\sum_{\cI\subseteq[n],~
n^{g(\beta)+\delta}
\le |\cI|\le 2n/\log^{1/3}(n)}
\frac{P_{\sigma|G}(\sigma=X^{(\sim\cI)})}{P_{\sigma|G}(\sigma=X)} \le n^{-2r} .
$$
The existence of $\cG_{\delta}$ is guaranteed by the following proposition:
\begin{proposition} \label{prop:xz}
Let $a,b,\alpha,\beta> 0$ be constants satisfying that $\sqrt{a}-\sqrt{b} > \sqrt{2}$ and $\alpha>b\beta$.
Let 
$$
(X,G,\sigma) \sim \SIBM(n,a\log(n)/n, b\log(n)/n,\alpha,\beta, 1) .
$$
For any $\delta>0$, there exists $n_0(\delta)$ such that for every even integer $n>n_0(\delta)$ and  every integer $n^{g(\beta)+\delta} \le k\le 2n/\log^{1/3}(n)$, there is a set $\cG_{\delta}^{(k)}$ for which

\noindent (i)
$P(G\in\cG_\delta^{(k)}) \ge 1- 2 n^{-k\delta'}$,
where $\delta':=\min(\frac{\delta}{8},\frac{(\sqrt{a}-\sqrt{b})^2}{4} - \frac{1}{2}) >0$.

\noindent (ii) For every $G\in\cG_\delta^{(k)}$,
$$
\sum_{\cI\subseteq[n],|\cI|=k}
\frac{P_{\sigma|G}(\sigma=X^{(\sim \cI)} )}
{P_{\sigma|G}(\sigma=X)} <
n^{-k \delta /4} .
$$
\end{proposition}
With the $\cG_\delta^{(k)}$'s given by Proposition~\ref{prop:xz}, we
define 
$$
\cG_\delta:=\bigcap_{k=n^{g(\beta)+\delta}}^{2n/\log^{1/3}(n)} \cG_\delta^{(k)} .
$$
By the union bound,
$$
P(G\in\cG_\delta)\ge 1-2 \sum_{k=n^{g(\beta)+\delta}}^{2n/\log^{1/3}(n)} n^{-k\delta'}
> 1-2 \sum_{k=\lceil 4r/\delta' \rceil}^{+\infty} n^{-k\delta'}
\ge 1- \frac{2 n^{-4r}}{1-n^{-\delta'}}
>1-n^{-2r}
$$
for large enough $n$.
 Moreover, for every $G\in\cG_\delta$ and large enough $n$,
\begin{align*}
& \sum_{\cI\subseteq[n],~~
n^{g(\beta)+\delta}\le |\cI|\le 2n/\log^{1/3}(n)}
\frac{P_{\sigma|G}(\sigma=X^{(\sim\cI)})}{P_{\sigma|G}(\sigma=X)} =
\sum_{k=n^{g(\beta)+\delta}}^{2n/\log^{1/3}(n)}
\hspace*{0.05in}
\sum_{\cI\subseteq[n],|\cI|=k}
\frac{P_{\sigma|G}(\sigma=X^{(\sim \cI)} )}
{P_{\sigma|G}(\sigma=X)}  \\
& < \sum_{k=n^{g(\beta)+\delta}}^{2n/\log^{1/3}(n)}
n^{-k \delta /4}
< \sum_{k=\lceil 16r/\delta \rceil}^{+\infty}
n^{-k \delta /4}
\le \frac{n^{-4r}}{1-n^{-\delta /4}} < n^{-2r} .
\end{align*}
Thus we have shown that Proposition~\ref{prop:43} is implied by Proposition~\ref{prop:xz}. Now we are left to prove the latter proposition.
It turns out that all we need for the proof of Proposition~\ref{prop:xz} is a tighter inequality than \eqref{eq:wuh}.
Recall that we obtain \eqref{eq:wuh} from \eqref{eq:3l} by using a coarse upper bound $\binom{n}{k}<n^k$. Here we use a tighter upper bound $\binom{n}{k}<n^k/(k!)$ in \eqref{eq:3l} and obtain that for
$G\in\cG_1^{(k)}\cap\cG^{(k)}(\epsilon)$, (see the definitions of $\cG_1^{(k)}$ and $\cG^{(k)}(\epsilon)$ in \eqref{eq:g1k} and \eqref{eq:gep})
\begin{equation} \label{eq:zm}
\begin{aligned}
& \sum_{\cI\subseteq[n],|\cI|=k}
\frac{P_{\sigma|G}(\sigma=X^{(\sim \cI)} )}
{P_{\sigma|G}(\sigma=X)} 
< (k!)^{-1} n^{kf_{\beta}((b-a)/2)} + \sum_{tk\log(n)=\lceil\frac{b-a}{2}k\log(n) \rceil}^{-1}
(k!)^{-1} n^{k(f_{\beta}(t)+\epsilon)} \\
& \le (k!)^{-1} n^{kg(\beta)}
+ \sum_{tk\log(n)=\lceil\frac{b-a}{2}k\log(n) \rceil}^{-1}  (k!)^{-1} n^{k(g(\beta) + \epsilon)}  \\
& \le (k!)^{-1} n^{k(g(\beta) + \epsilon)}
\Big(\frac{a-b}{2}\log(n^k)+1 \Big) \\
& < (k!)^{-1} n^{k(g(\beta) + 2\epsilon)}  ,
\end{aligned}
\end{equation}
where the second inequality holds because
$f_{\beta}(t)\le \tilde{g}(\beta) \le g(\beta)$ for all $t\le 0$
(see  Lemma~\ref{lm:tus}), and the last inequality holds for positive $\epsilon$ and large $n$.
It is well known that\footnote{We know from the Taylor expansion that $e^x>x^k/(k!)$ for any $x>0$. Taking $x=k$ gives us $k!>(k/e)^k$.} $k!>(k/e)^k$ for all positive integer $k$. Therefore, for $k>n^{g(\beta)+\delta}$, we have
$$
k!>(k/e)^k
=\exp(k\log(k)-k)
>\exp(k(g(\beta)+\delta)\log(n)-k)
=n^{k(g(\beta)+\delta-o(1))}
$$
Taking this into \eqref{eq:zm}, we obtain that $G\in\cG_1^{(k)}\cap\cG^{(k)}(\epsilon)$,
$$
\sum_{\cI\subseteq[n],|\cI|=k}
\frac{P_{\sigma|G}(\sigma=X^{(\sim \cI)} )}
{P_{\sigma|G}(\sigma=X)} 
< n^{k(2\epsilon-\delta+o(1))}
< n^{k(3\epsilon-\delta)}
$$
for positive $\epsilon$ and large $n$.
Let $\epsilon=\delta/4$ and define 
$$
\cG_{\delta}^{(k)}
:=\cG_1^{(k)}\cap\cG^{(k)}(\delta/4) .
$$
Then for $G\in\cG_{\delta}^{(k)}$ we have
$$
\sum_{\cI\subseteq[n],|\cI|=k}
\frac{P_{\sigma|G}(\sigma=X^{(\sim \cI)} )}
{P_{\sigma|G}(\sigma=X)} <
n^{-k \delta /4} .
$$
By \eqref{eq:Dk}, \eqref{eq:g1k} and \eqref{eq:Gk}, 
\begin{align*}
P(G\in\cG_{\delta}^{(k)})
& \ge 1-n^{-k\delta/8}- n^{k (1-\frac{(\sqrt{a}-\sqrt{b})^2}{2} + o(1) )} \\
& \ge 1-n^{-k\delta/8}- n^{k (\frac{1}{2}-\frac{(\sqrt{a}-\sqrt{b})^2}{4} )}
> 1- 2 n^{-k\delta'},
\end{align*}
where the second inequality follows from  $1-\frac{(\sqrt{a}-\sqrt{b})^2}{2}< 0$, and the last inequality follows from the definition $\delta'=\min(\frac{\delta}{8},\frac{(\sqrt{a}-\sqrt{b})^2}{4} - \frac{1}{2})$.

\section{Exact recovery in $O(n)$ time when $\lfloor \frac{m+1}{2} \rfloor \beta>\beta^\ast$}
\label{sect:direct}

In this section, we prove that Algorithm~\ref{alg:ez} in Section~\ref{sect:multi} is able to learn $\SIBM(n,a\log(n)/n, \linebreak[4] b\log(n)/n,\alpha,\beta, m)$ as long as $\sqrt{a}-\sqrt{b} > \sqrt{2}$,  $\alpha>b\beta$ and $\lfloor \frac{m+1}{2} \rfloor \beta>\beta^\ast$, where $\beta^\ast$ is defined in \eqref{eq:defstar}.

\begin{proposition} \label{prop:nr}
Let $a,b,\alpha,\beta> 0$ be constants satisfying that $\sqrt{a}-\sqrt{b} > \sqrt{2}$ and $\alpha>b\beta$. Let $m$ be an integer such that $\lfloor \frac{m+1}{2} \rfloor \beta>\beta^\ast$.
Let 
$$
(X,G,\{\sigma^{(1)},\dots,\sigma^{(m)}\})\sim \SIBM(n,a\log(n)/n, b\log(n)/n,\alpha,\beta, m) .
$$
Let $\hat{X}=\texttt{LearnSIBM}(\sigma^{(1)},\dots,\sigma^{(m)})$ be the output of Algorithm~\ref{alg:ez}. Then
$$
P(\hat{X}=X \text{~or~} \hat{X}=-X) = 1-o(1) .
$$
\end{proposition}

By Proposition~\ref{prop:tt},
if $\beta>\beta^\ast$, then $\sigma=\pm X$ with probability $1-o(1)$, so $m=1$ sample suffices for recovery. In the rest of this section, we will focus on the case $\beta\le \beta^\ast$.

In Proposition~\ref{prop:43} (or Proposition~\ref{prop:1}), we have shown that $\dist(\sigma^{(i)}, \pm X) =o(n)$ for all $i\in[m]$ with probability $1-O(n^{-r})$ for any constant $r>0$, but it is possible that $\sigma^{(1)}$ is close to $X$ while $\sigma^{(2)}$ is close to $-X$. Step 1 (the alignment step) in the above algorithm eliminates such possibility: After the alignment step, with probability $1-O(n^{-r})$ only the following two scenarios will happen: Either
$\dist(\sigma^{(i)}, X)=o(n)$ for all $i\in[m]$ or $\dist(\sigma^{(i)}, -X)=o(n)$ for all $i\in[m]$. Without loss of generality, we assume the former case, i.e., we assume that 
$\dist(\sigma^{(i)}, X) \le n/2$ for all $i\in[m]$,
and in the rest of this section we will prove that Algorithm~\ref{alg:ez} outputs $\hat{X}=X$ with probability $1-o(1)$.

Given the ground truth $X$, the graph $G$ and a vertex $i\in[n]$, define the neighbors of $i$ in $G$ as 
\begin{equation} \label{eq:ngbi}
\cN_i(G) := \{j\in[n]\setminus\{i\}:
\{i,j\}\in E(G) \} .
\end{equation}
Given a sample $\sigma\in\{\pm 1\}^n$ and a graph $G$,  we define the set of ``bad" neighbors of the vertex $i$ in $G$ as
$$
\Omega_i(\sigma,G):=\{j\in \cN_i(G): 
\sigma_j \neq X_j \} .
$$
Given a vertex $i\in[n]$, a graph $G$ and an integer $z>0$, we also define
$$
\Lambda_i(G, z):=\{ \sigma\in\{\pm 1\}^n: |\Omega_i(\sigma,G)| < z \} ,
$$
i.e., $\Lambda_i(G, z)$ consists of all the samples for which the number of ``bad" neighbors of the vertex $i$ in $G$ is less than $z$.

\begin{lemma} \label{lm:us}
Let $a,b,\alpha,\beta> 0$ be constants satisfying that $\sqrt{a}-\sqrt{b} > \sqrt{2}$, $\alpha>b\beta$ and $\beta\le \beta^\ast$.
Given any $r>0$, there exists an integer $z>0$ such that for large enough $n$ and every $i\in[n]$, 
\begin{equation} \label{eq:zr}
P_{\SIBM} (\sigma\in \Lambda_i(G, z)
| \dist(\sigma, X) \le n/2)
> 1 - n^{-r} .
\end{equation}
\end{lemma}
\begin{proof}
The proof of this lemma is similar to the proof of Proposition 4.6 in \cite{MNS16}.
The event $\{\sigma\notin \Lambda_i(G, z) \}$ can be written as
\begin{equation} \label{eq:lei}
\{\sigma\notin \Lambda_i(G, z) \}
=\{|\Omega_i(\sigma,G)| \ge z \}
=\bigcup_{\tilde{\cI}\subseteq[n]\setminus\{i\}, |\tilde{\cI}| = z} \{\tilde{\cI} \subseteq \Omega_i(\sigma,G)\} .
\end{equation}
Given a subset $\tilde{\cI}\subseteq[n]\setminus\{i\}$,
the event $\{\tilde{\cI} \subseteq \Omega_i(\sigma,G) \}$ is the intersection of two events $\{\sigma_j \neq X_j \text{~for all~} j\in \tilde{\cI}\}$ and $\{i\in\cN_j(G) \text{~for all~} j\in \tilde{\cI}\}$.

We start with the analysis of the first event.
In Proposition~\ref{prop:43}, we have shown that
for any $\delta>0$ and any $r'>0$, there exists $n_0(\delta, r')$ such that for all even integers $n>n_0(\delta, r')$,
$$ 
P_{\SIBM} \Big(\dist(\sigma,\pm X) \le n^{g(\beta)+\delta}
 \Big) \ge 1- n^{-r'} .
$$
By Lemma~\ref{lm:ele} (vi), we know that $g(\beta)<1$ for all $0<\beta\le\beta^\ast$, so we can always choose a $\delta>0$ such that $g(\beta)+\delta<1$.
Let $\theta:=g(\beta)+\delta<1$. Then for large enough $n$,
$$
P_{\SIBM} \Big(\dist(\sigma,\pm X) \le n^{\theta} \Big) \ge 1- n^{-r'} .
$$
This in particular implies that
$$
P_{\SIBM} \Big(\dist(\sigma, X) \le n^{\theta} | \dist(\sigma, X) \le n/2
\Big) \ge 1- n^{-r'} .
$$
Moreover, Lemma~\ref{lm:bq} in Appendix~\ref{ap:6} (see inequality \eqref{eq:l2}) tells us that
$$
 P_{\SIBM}(\sigma_j=-X_j \text{~for all~}  j\in\tilde{\cI}
~\big| \dist(\sigma,X) \le n^\theta) \le \Big(\frac{n^\theta}{n/2-|\tilde{\cI}|}
\Big)^{|\tilde{\cI}|}
\quad \text{for all~} \tilde{\cI}\subseteq [n] .
$$
Therefore, for any subset $\tilde{\cI}\subseteq[n]\setminus\{i\}$ with size $|\tilde{\cI}|=z$, we have
$$
P_{\SIBM}(\sigma_j \neq X_j \text{~for all~} j\in \tilde{\cI}
| \dist(\sigma, X) \le n/2 )
\le O(n^{-(1-\theta) z}) + O(n^{-r'}) . 
$$
Taking $r'>(1-\theta) z$ gives us
\begin{equation} \label{eq:gr}
P_{\SIBM}(\sigma_j \neq X_j \text{~for all~} j\in \tilde{\cI}
| \dist(\sigma, X) \le n/2 )
\le O(n^{-(1-\theta) z}) . 
\end{equation}

Given $r>0$, we will prove \eqref{eq:zr} for any integer $z\ge \frac{r+1}{1-\theta}$. 
Now condition on the event $\{\sigma_j \neq X_j \text{~for all~} j\in \tilde{\cI}\}$.
By Lemma~\ref{lm:bmd} in Appendix~\ref{ap:6}, with probability $1-o(1)$ each vertex $j\in \tilde{\cI}$ has at most $O(\log(n))$ neighbors in $[n]_+:=\{v\in[n]:X_v=+1\}$ and at most $O(\log(n))$ neighbors in $[n]_-:=\{v\in[n]:X_v= -1\}$. Conditioned on the number of neighbors in $[n]_+$ (respectively, $[n]_-$), the neighbors of each vertex $j\in \tilde{\cI}$ are  uniformly distributed in $[n]_+$ (respectively, $[n]_-$).
Therefore, 
$$
P(i\in\cN_j(G))=O\Big( \frac{\log(n)}{n} \Big)
\quad \text{and} \quad
P(i\in\cN_j(G) \text{~for all~} j\in \tilde{\cI}) = O\Big( \frac{\log^z(n)}{n^z} \Big) .
$$
Combining this with \eqref{eq:gr}, we have
$$
P_{\SIBM} (\tilde{\cI} \subseteq \Omega_i(\sigma,G))
= O (n^{-(2-\theta)z}\log^z(n)) .
$$
Finally, combining this with \eqref{eq:lei} and the union bound, we have
\begin{align*}
P_{\SIBM} \big(\sigma\notin \Lambda_i(G, z)
| \dist(\sigma, X) \le n/2 \big)
& \le n^z O (n^{-(2-\theta)z}\log^z(n))
= O (n^{-(1-\theta)z}\log^z(n))  \\
& \le O (n^{-r-1}\log^z(n)) <n^{-r}
\end{align*}
for large enough $n$. This completes the proof of the lemma.
\end{proof}

We define two subsets of $\cN_i(G)$ with the same labeling and the opposite labeling as
$$
\cN_{i,+}(G):=\{j\in \cN_i(G): X_j=X_i \}
\quad \text{and} \quad
\cN_{i,-}(G):=\{j\in \cN_i(G): X_j=-X_i \} ,
$$
respectively.
Recall that in \eqref{eq:defAB}, we defined $A_i=A_i(G):=|\cN_{i,+}(G)|$ and $B_i=B_i(G):=|\cN_{i,-}(G)|$.
Also recall that at the beginning of Section~\ref{sect:gen}, we defined $X^{(\sim\cI)}$ as the vector obtained by flipping the coordinates in $\cI\subseteq[n]$ while keeping all the other coordinates the same as $X$.
By definition, $X^{(\sim\cI)}\in \Lambda_i(G,z)$ if and only if $|\cI\cap \cN_i(G)| < z$.

\begin{lemma} \label{lm:et}
Let $0<\theta<1$ be some constant. Then for large enough $n$ we have
\begin{align*}
\exp\Big(2\big(\beta+\frac{\alpha\log(n)}{n} \big) (B_i-A_i - 2z) \Big) & \le 
\frac{P_{\sigma|G}(\sigma_i= -X_i, \sigma\in \Lambda_i(G,z) ,\dist(\sigma, X) \le n^{\theta} ) } { P_{\sigma|G}(\sigma_i=X_i, \sigma\in \Lambda_i(G,z) ,\dist(\sigma, X) \le n^{\theta} ) } \\
& \le \exp\Big(2\big(\beta+\frac{\alpha\log(n)}{n} \big) (B_i-A_i + 2z) \Big)
\end{align*}
\end{lemma}

\begin{proof}
We only need to show that for every $\cI\in [n]\setminus\{i\}$ with size $|\cI|\le n^\theta$ such that $X^{(\sim\cI)}\in \Lambda_i(G,z)$,
\begin{equation} \label{eq:qk}
\begin{aligned}
\exp\Big(2\big(\beta+\frac{\alpha\log(n)}{n} \big) (B_i-A_i - 2z) \Big) & \le 
\frac{P_{\sigma|G}(\sigma= X^{(\sim(\cI\cup\{i\}))} ) } { P_{\sigma|G}(\sigma= X^{(\sim\cI)} ) } \\
& \le \exp\Big(2\big(\beta+\frac{\alpha\log(n)}{n} \big) (B_i-A_i + 2z) \Big) .
\end{aligned}
\end{equation}
Define
$$
A_i':=A_i-|\cI\cap \cN_{i,+}(G)|+|\cI\cap \cN_{i,-}(G)| \text{~~and~~}
B_i':=B_i-|\cI\cap \cN_{i,-}(G)| + |\cI\cap \cN_{i,+}(G)| .
$$
Since $X^{(\sim\cI)}\in \Lambda_i(G,z)$, we have $|\cI\cap \cN_i(G)| \le z-1$. Therefore,
\begin{equation} \label{eq:oo}
B_i-A_i-(2z-2) \le B_i'-A_i'\le B_i-A_i+(2z-2) .
\end{equation}
By \eqref{eq:isingma}, we have
\begin{align*}
& \frac{P_{\sigma|G}(\sigma= X^{(\sim(\cI\cup\{i\}))} ) } { P_{\sigma|G}(\sigma= X^{(\sim\cI)} ) } \\
= & \exp\Big(2\beta(B_i'-A_i')
-\frac{2\alpha\log(n)}{n}
\Big( (\frac{n}{2}-B_i+O(n^{\theta}) )
- (\frac{n}{2}-A_i+O(n^{\theta}) )
\Big) \Big)   \\
= & \exp\Big(2\beta(B_i'-A_i')
+\frac{2\alpha\log(n)}{n}
(B_i-A_i) +o(1) \Big)
\end{align*}
Taking \eqref{eq:oo} into this equation gives us \eqref{eq:qk} and completes the proof.
\end{proof}

We further define 
\begin{align*}
\Lambda(G, z) :=
\bigcap_{i=1}^n
\Lambda_i(G, z) .
\end{align*}
The following corollary follows immediately from Lemma~\ref{lm:us} and the union bound.
\begin{corollary} \label{cr:1}
Let $a,b,\alpha,\beta> 0$ be constants satisfying that $\sqrt{a}-\sqrt{b} > \sqrt{2}$ and $\alpha>b\beta$.
Given any $r>0$, there exists an integer $z>0$ such that for large enough $n$, 
$$
P_{\SIBM} (\sigma\in \Lambda(G, z)
| \dist(\sigma, X) \le n/2)
> 1 - n^{-2r} .
$$
Equivalently, for any $r>0$, there is an integer $z>0$ and a set $\cG_{\good}$ such that

\noindent (i)
$P(G\in\cG_{\good}) \ge 1- O(n^{-r})$.

\noindent (ii) For every $G\in\cG_{\good}$,
$$
P_{\sigma|G} \big( \sigma \in  \Lambda(G, z)  \big| \dist(\sigma, X) \le n/2 \big) =1- O(n^{-r}).
$$
\end{corollary}
\begin{proof}
We only explain why the first statement implies the second one.
Define a set 
$$
\cG_{\bad}:= \big\{ G: P_{\sigma|G} \big( \sigma \notin  \Lambda(G, z)  \big| \dist(\sigma, X) \le n/2 \big)
>n^{-r} \big\} .
$$
Then by the first statement and the Markov inequality, $P(G\in\cG_{\bad})\le n^{-r}$. Therefore, the second statement follows by taking $\cG_{\good}$ to be the complement of $\cG_{\bad}$.
\end{proof}

Now we are ready to prove Proposition~\ref{prop:nr}.

\noindent
{\em Proof of Proposition~\ref{prop:nr}.}
Recall that $\sigma^{(1)},\dots,\sigma^{(m)}$ are the samples of SIBM.
At the beginning of this section, we have shown that after the alignment step in Algorithm~\ref{alg:ez}, with probability $1-O(n^{-r})$ for any constant $r>0$ only the following two scenarios will happen: Either $\dist(\sigma^{(j)}, X)=o(n)$ for all $j\in[m]$ or $\dist(\sigma^{(j)}, -X)=o(n)$ for all $j\in[m]$. Without loss of generality, we assume the former case, i.e., we assume that 
$\dist(\sigma^{(j)}, X) \le n/2$ for all $j\in[m]$. 

Corollary~\ref{cr:1} together with Proposition~\ref{prop:43} implies that there is an integer $z>0$ and a set $\cG_{\good}$ such that

\noindent (i)
$P(G\in\cG_{\good}) \ge 1- O(n^{-4})$.

\noindent (ii) For every $G\in\cG_{\good}$, conditioning on the event $\dist(\sigma^{(j)}, X) \le n/2$ for all $j\in[m]$,
\begin{equation}  \label{eq:chui}
P_{\sigma|G} \big( \sigma^{(j)}\in  \Lambda(G, z)
\text{~and~} \dist(\sigma^{(j)}, X) \le n^\theta
\text{~for all~} j\in[m]  \big) 
=1- O(n^{-4}) ,
\end{equation}
where we can choose $\theta$ to be any constant in the open interval $(g(\beta), 1)$.
By Lemma~\ref{lm:et},
$$
P_{\sigma|G}(\sigma_i^{(j)} \neq X_i \big| \sigma^{(j)}\in \Lambda(G,z) ,\dist(\sigma^{(j)}, X) \le n^{\theta} ) 
 \le \exp\Big(2\big(\beta+\frac{\alpha\log(n)}{n} \big) (B_i-A_i + 2z) \Big)
$$
for all $i\in[n]$ and all $j\in[m]$.
For $i\in[n]$, define 
$$
\Phi_i := |\{j\in[m]: \sigma_i^{(j)} \neq X_i\}|
$$
as the number of samples for which $\sigma_i^{(j)} \neq X_i$.
For an integer $u\in[m]$, we have
$$
\{\Phi_i \ge u\} =
\bigcup_{\cJ\subseteq[m], |\cJ|=u}
\{\sigma_i^{(j)} \neq X_i \text{~for all~} j\in \cJ\} .
$$
Therefore, by the union bound,
\begin{align*}
& P_{\sigma|G}( \Phi_i \ge u  ~\big|~  \sigma^{(j)}\in \Lambda(G,z) ,\dist(\sigma^{(j)}, X) \le n^{\theta} \text{~for all~} j\in[m] ) \\
\le & \binom{m}{u} \exp\Big(2u\big(\beta+\frac{\alpha\log(n)}{n} \big) (B_i-A_i + 2z) \Big)
\quad\quad\quad\quad\text{~for all~} i\in[n] .
\end{align*}
Combining this with \eqref{eq:chui}, we obtain that for every $G\in\cG_{\good}$, conditioning on the event $\dist(\sigma^{(j)}, X) \le n/2$ for all $j\in[m]$,
$$
P_{\sigma|G} \big( \Phi_i \ge u \big) \le \binom{m}{u} \exp\Big(2u \big(\beta+\frac{\alpha\log(n)}{n} \big) (B_i-A_i + 2z) \Big) + O(n^{-4}).
$$
Using the union bound, we have
$$
P_{\sigma|G} \big(\exists i \text{~s.t.~} \Phi_i \ge u \big) \le 
C \sum_{i=1}^n \exp\Big(2u \big(\beta+\frac{\alpha\log(n)}{n} \big) (B_i-A_i ) \Big) + o(1) ,
$$
where $C:=\binom{m}{u}\exp(4u(\beta+\alpha)z)$ is a constant.
In Section~\ref{sect:k=1}, we have shown that if $\beta>\beta^*$, then there is a set $\cG^{(1)}$ such that (i) $P(G\in \cG^{(1)})=1-o(1)$; (ii) for every $G\in\cG^{(1)}$,
$$
\sum_{i=1}^n \exp\Big(2 \big(\beta+\frac{\alpha\log(n)}{n} \big) (B_i-A_i ) \Big) = o(1) .
$$
More precisely, we proved that for every $G\in\cG^{(1)}$, $\frac{\sum_{i=1}^n P_{\sigma|G}(\sigma=X^{(\sim i)} )}
{P_{\sigma|G}(\sigma=X)}=o(1)$ by first using the upper bound
$\frac{\sum_{i=1}^n P_{\sigma|G}(\sigma=X^{(\sim i)} )}
{P_{\sigma|G}(\sigma=X)}
\le \sum_{i=1}^n \exp\Big(2\big(\beta+\frac{\alpha\log(n)}{n} \big) (B_i-A_i) \Big)$
in \eqref{eq:fd} and then proving the right-hand side is $o(1)$.
This immediately implies that if $u\beta>\beta^*$, then for every $G\in\cG^{(1)}\cap\cG_{\good}$, conditioning on the event $\dist(\sigma^{(j)}, X) \le n/2$ for all $j\in[m]$,
$$
P_{\sigma|G} \big(\exists i \text{~s.t.~} \Phi_i \ge u \big) \le 
C \sum_{i=1}^n \exp\Big(2u \big(\beta+\frac{\alpha\log(n)}{n} \big) (B_i-A_i ) \Big) + o(1) =o(1).
$$
Since $P(G\in\cG^{(1)}\cap\cG_{\good})=1-o(1)$, we conclude that $P_{\SIBM} \big(\Phi_i \le u-1 \text{~for all~} i\in[n] \big)=1-o(1)$ conditioning on the event $\dist(\sigma^{(j)}, X) \le n/2$ for all $j\in[m]$.

By assumption we have $\lfloor \frac{m+1}{2} \rfloor \beta>\beta^\ast$. Therefore, conditioning on the event $\dist(\sigma^{(j)}, X) \le n/2$ for all $j\in[m]$, we have $\Phi_i\le \lfloor \frac{m-1}{2} \rfloor$ for all $i\in[n]$ with probability $1-o(1)$. As a consequence, after the majority voting step in Algorithm~\ref{alg:ez}, $\hat{X}_i=X_i$ for all $i\in[n]$ with probability $1-o(1)$. This completes the proof of Proposition~\ref{prop:nr}.
\hfill\qed

\section{Samples differ from $\pm X$ in $\Theta(n^{g(\beta)})$ coordinates when $\beta\le\beta^\ast$}  \label{sect:struct}
Recall the definitions of $A_i$ and $B_i$ in previous sections; see the beginning of Section~\ref{sect:k=1}. 
By definition,
$A_i\sim \Binom(\frac{n}{2}-1,\frac{a\log(n)}{n})$ and $B_i\sim \Binom(\frac{n}{2}, \frac{b\log(n)}{n})$, and they are independent. 
Also note that $A_i$ and $B_i$ are functions of the underlying graph $G$.


\begin{proposition}  \label{prop:con}
Let $(X,G)\sim \SSBM(n,a\log(n)/n, b\log(n)/n)$, where $\sqrt{a}-\sqrt{b} > \sqrt{2}$.
Suppose that $0< \beta\le \beta^\ast$.
Then there is a set $\cG_{\con}$ such that (i) $P(G \in \cG_{\con}) = 1-o(1)$ and (ii) for every $G\in \cG_{\con}$, 
$$
\sum_{i=1}^n \exp\big(2\beta (B_i-A_i) \big)
=(1+o(1)) n^{g(\beta)} .
$$
\end{proposition}

\begin{proof}
Let $\cG_1:=\{G:B_i-A_i<0\text{~for all~}i\in[n]\}$. By \eqref{eq:tD}, we have $P(G\in\cG_1)=1-o(1)$. We will prove that
\begin{align}
& E \Big[ \sum_{i=1}^n  \exp\big(2\beta (B_i-A_i) \big) ~\Big|~ G\in\cG_1 \Big]
= (1+o(1)) n^{g(\beta)}  , \label{eq:cg} \\
& \Var \Big[ \sum_{i=1}^n  \exp\big(2\beta (B_i-A_i) \big) ~\Big|~ G\in\cG_1 \Big]
= o (n^{g(\beta)} ) .  \label{eq:sw}
\end{align}
Then the proposition follows immediately from  Chebyshev's inequality and the fact that $g(\beta)\ge 0$ when $0<\beta\le\beta^\ast$ (see Lemma~\ref{lm:ele}).

In Proposition~\ref{prop:df} (see Appendix~\ref{ap:um}), we prove \eqref{eq:cg} for $0<\beta<\frac{1}{4}\log\frac{a}{b}$. By Lemma~\ref{lm:ele}, $\beta^\ast<\frac{1}{4}\log\frac{a}{b}$, so \eqref{eq:cg} holds for $0< \beta\le \beta^\ast$.
Now we are left to prove \eqref{eq:sw}. Observe that
\begin{equation} \label{eq:mh1}
\begin{aligned}
& \Var \Big[ \sum_{i=1}^n  \exp\big(2\beta (B_i-A_i) \big) ~\Big|~ G\in\cG_1 \Big] \\
= & \sum_{i=1}^n  \Var \big[\exp\big(2\beta (B_i-A_i) \big) ~\big|~ G\in\cG_1 \big] \\
& \hspace*{1.2in} + \sum_{i,j\in[n],i\neq j}
\Cov(\exp\big(2\beta (B_i-A_i) \big), \exp\big(2\beta (B_j-A_j) \big) ~\big|~ G\in\cG_1 ) \\
\le & \sum_{i=1}^n E \big[ \exp\big(4\beta (B_i-A_i) \big) ~\big|~ G\in\cG_1 \big] \\
& \hspace*{1.2in} + \sum_{i,j\in[n],i\neq j}
\Cov(\exp\big(2\beta (B_i-A_i) \big), \exp\big(2\beta (B_j-A_j) \big) ~\big|~ G\in\cG_1) .
\end{aligned}
\end{equation}
By Corollary~\ref{cr:yy} in Appendix~\ref{ap:um}, for all $\beta>0$ we have
$$
\sum_{i=1}^n E[ \exp\big(4\beta (B_i-A_i) \big) \big| G\in\cG_1]
= O ( n^{\tilde{g}(2\beta)} ) .
$$
By Lemma~\ref{lm:ele}, $\tilde{g}(\beta)$ is a decreasing function, and it is strictly decreasing when $\beta\le\beta^\ast<\frac{1}{4}\log\frac{a}{b}$. Therefore, $g(\beta)=\tilde{g}(\beta)>\tilde{g}(2\beta)$ whenever $\beta\le\beta^\ast$. As a consequence, 
\begin{equation} \label{eq:mh2}
\sum_{i=1}^n E[ \exp\big(4\beta (B_i-A_i) \big) \big| G\in\cG_1]
< o ( n^{g(\beta)} ) .
\end{equation}

Now we are left to bound the covariance of $\exp\big(2\beta (B_i-A_i) \big)$ and $\exp\big(2\beta (B_j-A_j) \big)$ for $i\neq j$.
Define $\xi_{ij}=\xi_{ij}(G):=\mathbbm{1}[\{i,j\}\in E(G)]$ as the indicator function of the edge $\{i,j\}$ connected in graph $G$.
Now suppose that\footnote{The case of $X_i=X_j$ can be handled in the same way.} $X_i\neq X_j$. Then we can decompose $B_i$ and $B_j$ as $B_i=B_i'+\xi_{ij}$ and $B_j=B_j'+\xi_{ij}$, where both $B_i'$ and $B_j'$ have distribution $\Binom(\frac{n}{2} - 1 , \frac{b\log(n)}{n})$, and the five random variables $A_i, A_j, B_i',B_j'$ and $\xi_{ij}$ are independent.
Therefore,
\begin{align*}
& \Cov(\exp\big(2\beta (B_i-A_i) \big), \exp\big(2\beta (B_j-A_j) \big) ) \\
= & E[\exp\big(2\beta (B_i-A_i+B_j-A_j) \big)] 
-E[\exp\big(2\beta (B_i-A_i) \big)] E[\exp\big(2\beta (B_j-A_j) \big)]   \\
= & E[\exp\big(2\beta (B_i'-A_i) \big)] E[\exp\big(2\beta (B_j'-A_j) \big)]  \Big( E[\exp(4\beta \xi_{ij})] -
\big(E[\exp(2\beta \xi_{ij})]\big)^2 \Big) \\
= & E[\exp\big(2\beta (B_i'-A_i) \big)] E[\exp\big(2\beta (B_j'-A_j) \big)] \\
& \hspace*{1.2in}
\Big( 1-\frac{b\log(n)}{n} + \frac{b\log(n)}{n} e^{4\beta} -
\Big(1-\frac{b\log(n)}{n} + \frac{b\log(n)}{n} e^{2\beta} \Big)^2 \Big) \\
= & \Theta\Big( \frac{\log(n)}{n} \Big) E[\exp\big(2\beta (B_i'-A_i) \big)] E[\exp\big(2\beta (B_j'-A_j) \big)] \\
= & \Theta\Big( \frac{\log(n)}{n} \Big) E[\exp\big(2\beta (B_i-A_i) \big)] E[\exp\big(2\beta (B_j-A_j) \big)]  ,
\end{align*} 
where the last equality holds because $\exp\big(2\beta (B_i'-A_i) \big)$ differs from $\exp\big(2\beta (B_i-A_i) \big)$  by a factor of at most $e^{2\beta}$. 
By \eqref{eq:pl} in Appendix~\ref{ap:um}, $E \big[  \exp\big(2\beta (B_i-A_i) \big) ~\big|~ G\in\cG_1 \big] 
= (1+o(1)) E \big[  \exp\big(2\beta (B_i-A_i) \big) \big]$ when $0<\beta\le\beta^\ast$. Similarly, one can also show that $E[\exp\big(2\beta (B_i-A_i+B_j-A_j) \big) ~\big|~ G\in\cG_1 ] = (1+o(1)) E[\exp\big(2\beta (B_i-A_i+B_j-A_j) \big)]$ when $0<\beta\le\beta^\ast$.
Therefore,
\begin{align*}
& \Cov \big(\exp\big(2\beta (B_i-A_i) \big), \exp\big(2\beta (B_j-A_j)  \big) ~\big|~ G\in\cG_1 \big) \\
= & (1+o(1)) \Cov \big(\exp\big(2\beta (B_i-A_i) \big), \exp\big(2\beta (B_j-A_j)  \big) \big) \\
= & \Theta\Big( \frac{\log(n)}{n} \Big) E[\exp\big(2\beta (B_i-A_i) \big)] E[\exp\big(2\beta (B_j-A_j) \big)] .
\end{align*}
As a consequence,
\begin{align*}
& \sum_{i,j\in[n],i\neq j}
\Cov \big(\exp\big(2\beta (B_i-A_i) \big), \exp\big(2\beta (B_j-A_j) \big) ~\big|~ G\in\cG_1 \big) \\
= & \Theta\Big( \frac{\log(n)}{n} \Big) \sum_{i,j\in[n],i\neq j} \Big( E[\exp\big(2\beta (B_i-A_i) \big)] E[\exp\big(2\beta (B_j-A_j) \big)] \Big) \\
\le & \Theta\Big( \frac{\log(n)}{n} \Big)
\Big(\sum_{i=1}^n  E[\exp\big(2\beta (B_i-A_i) \big)] \Big)^2 \\
\overset{(a)}{=} & \Theta(n^{2g(\beta)-1} \log(n)) \\
\overset{(b)}{=} & o(n^{g(\beta)})
\end{align*}
where equality $(a)$ follows from \eqref{eq:lb}, and $(b)$ follows from $g(\beta)<1$ when $0<\beta\le\beta^\ast$; see Lemma~\ref{lm:ele} (vi).
Finally, \eqref{eq:sw} follows immediately from this bound and \eqref{eq:mh1}--\eqref{eq:mh2}.
\end{proof}

\begin{remark}
One might wonder why we use the conditional expectation and variance to prove Proposition~\ref{prop:con} instead of using the unconditional ones. By \eqref{eq:lb} in Appendix~\ref{ap:um},
$$
E \Big[ \sum_{i=1}^n  \exp\big(2\beta (B_i-A_i) \big) \Big]
= (1+o(1)) n^{g(\beta)}  
$$
for all $\beta$. This gives us the same estimate as the conditional expectation in \eqref{eq:cg}. However, the unconditional variance can be much larger than the conditional one in \eqref{eq:sw}. Recall from \eqref{eq:mh1} that we use $\sum_{i=1}^n E[ \exp\big(4\beta (B_i-A_i) \big)]$ to bound the variance. 
By Corollary~\ref{cr:yy} in Appendix~\ref{ap:um},
for the unconditional case this sum is of order $\Theta(n^{g(2\beta)})$ while for the conditional case it is of order $O(n^{\tilde{g}(2\beta)})$. One can show that if $\beta\le\beta^\ast$, then we always have $g(\beta)>\tilde{g}(2\beta)$. However, it is possible that $g(2\beta)>g(\beta)$. In particular, when $\beta=\beta^\ast$, we have $g(\beta^\ast)=0$, i.e., the  expectation (both conditional and unconditional) is of order $\Theta(1)$, but it is possible that $g(2\beta^\ast)>0$, i.e., the unconditional variance is $\omega(1)$. 
\end{remark}

\begin{theorem}  \label{thm:dist}
Let $a,b,\alpha,\beta> 0$ be constants satisfying that $\sqrt{a}-\sqrt{b} > \sqrt{2}$, $\alpha>b\beta$ and $0<\beta\le \beta^\ast$. Let 
$
(X,G,\sigma) \sim \SIBM(n,a\log(n)/n, b\log(n)/n,\alpha,\beta, 1) .
$
Then
$$
P_{\SIBM}(\dist(\sigma,\pm X)= \Theta(n^{g(\beta)}) ) = 1-o(1) .
$$
\end{theorem}
\begin{proof}
We prove the following equivalent form:
$$
P_{\SIBM}(\dist(\sigma, X) = \Theta(n^{g(\beta)}) ~\big|~ \dist(\sigma, X) \le n/2) = 1-o(1) .
$$
Corollary~\ref{cr:1} together with Proposition~\ref{prop:43} implies that there is an integer $z>0$ and a set $\cG_{\good}$ such that

\noindent (i)
$P(G\in\cG_{\good}) \ge 1- O(n^{-4})$.

\noindent (ii) For every $G\in\cG_{\good}$, 
\begin{equation}  \label{eq:hs}
P_{\sigma|G} \big( \sigma \in  \Lambda(G, z)
\text{~and~} \dist(\sigma, X) \le n^\theta ~\big|~ \dist(\sigma, X) \le n/2 \big) 
=1- O(n^{-4}) ,
\end{equation}
where we can choose $\theta$ to be any constant in the open interval $(\tilde{g}(\beta), 1)$.
By Lemma~\ref{lm:et} we know that for all $i\in[n]$,
$P_{\sigma|G}(\sigma_i \neq X_i \big| \sigma\in \Lambda(G,z) ,\dist(\sigma, X) \le n^{\theta} )$ differ from
$\exp\Big(2\big(\beta+\frac{\alpha\log(n)}{n} \big) (B_i-A_i) \Big)$
by at most a constant factor. Since $|B_i-A_i|=O(\log(n))$ with probability $1-o(1)$, the term $\frac{\alpha\log(n)}{n}(B_i-A_i)=o(1)$ and is negligible. Thus we conclude that
$$
\underline{C}
\exp\big(2\beta (B_i-A_i) \big) \le
P_{\sigma|G}(\sigma_i \neq X_i \big| \sigma\in \Lambda(G,z) ,\dist(\sigma, X) \le n^{\theta} ) \le \overline{C}
\exp\big(2\beta (B_i-A_i) \big)
$$
for all $i\in[n]$, where $\underline{C}$ and $\overline{C}$ are constants that are independent of $n$.
In fact, we obtained a much stronger inequality \eqref{eq:qk} in the proof of Lemma~\ref{lm:et}. More precisely, inequality \eqref{eq:qk} can be reformulated as follows: For every $\bar{\sigma}\in\{\pm 1\}^n$ such that $\bar{\sigma}\in \Lambda(G,z)$ and $\dist(\bar{\sigma}, X) \le n^{\theta}$,
\begin{equation} \label{eq:soon}
\underline{C}
\exp\big(2\beta (B_i-A_i) \big) \le
P_{\sigma|G}(\sigma_i \neq X_i \big| \sigma_j=\bar{\sigma}_j \text{~for all~} j\neq i ) \le \overline{C}
\exp\big(2\beta (B_i-A_i) \big).
\end{equation}
In some sense, it tells us that conditioning on the event $\{\sigma\in \Lambda(G,z) ,\dist(\sigma, X) \le n^{\theta}\}$, the random variables $\sigma_1,\dots,\sigma_n$ are ``almost" independent.
Now define 
$\phi_i := \mathbbm{1}[\sigma_i \neq X_i]$ for $i\in[n]$, and
we want to estimate $\dist(\sigma,X)=\sum_{i=1}^n \phi_i$. 
Given a fixed graph $G$ and a random sample $\sigma$, we define {\bf Bernoulli} random variables $\underline{S}_1,\dots, \underline{S}_n$ and $\overline{S}_1,\dots,\overline{S}_n$ with the following properties:
\begin{enumerate}
\item $\underline{S}_1,\dots, \underline{S}_n$ are conditionally independent given the event $\{\sigma\in \Lambda(G,z) ,\dist(\sigma, X) \le n^{\theta}\}$. $\overline{S}_1,\dots,\overline{S}_n$ are also conditionally independent given the event $\{\sigma\in \Lambda(G,z) ,\dist(\sigma, X) \le n^{\theta}\}$.
\item Conditioning on the event $\{\sigma\in \Lambda(G,z) ,\dist(\sigma, X) \le n^{\theta}\}$, $P(\underline{S}_i=1)=\underline{C}
\exp\big(2\beta (B_i-A_i) \big)$ and $P(\overline{S}_i=1)=\overline{C}
\exp\big(2\beta (B_i-A_i) \big)$ for all $i\in[n]$.
\end{enumerate}
By these two properties, conditioning on the event $\{\sigma\in \Lambda(G,z) ,\dist(\sigma, X) \le n^{\theta}\}$, we have
\begin{align*}
& E[\underline{S}_1 + \dots + \underline{S}_n] = \underline{C}
\sum_{i=1}^n \exp\big(2\beta (B_i-A_i) \big) , \quad
\Var(\underline{S}_1 + \dots + \underline{S}_n) \le \underline{C}
\sum_{i=1}^n \exp\big(2\beta (B_i-A_i) \big), \\
& E[\overline{S}_1 + \dots + \overline{S}_n] = \overline{C}
\sum_{i=1}^n \exp\big(2\beta (B_i-A_i) \big) , \quad
\Var(\overline{S}_1 + \dots + \overline{S}_n) \le \overline{C}
\sum_{i=1}^n \exp\big(2\beta (B_i-A_i) \big) ,
\end{align*}
where we use the fact that the variance of a Bernoulli random variable is always upper bounded by its expectation.
By Proposition~\ref{prop:con}, for all $G\in\cG_{\good}\cap\cG_{\con}$, we have 
\begin{align*}
& E[\underline{S}_1 + \dots + \underline{S}_n] = \Theta(n^{g(\beta)}) , \quad
\Var(\underline{S}_1 + \dots + \underline{S}_n) = O(n^{g(\beta)}) , \\
& E[\overline{S}_1 + \dots + \overline{S}_n] = \Theta(n^{g(\beta)}) , \quad
\Var(\overline{S}_1 + \dots + \overline{S}_n) = O(n^{g(\beta)}) 
\end{align*}
conditioning on the event $\{\sigma\in \Lambda(G,z) ,\dist(\sigma, X) \le n^{\theta}\}$. Since $g(\beta)\ge 0$ for all $0<\beta\le \beta^\ast$, by Chebyshev's inequality we know that both $\underline{S}_1 + \dots + \underline{S}_n=\Theta(n^{g(\beta)})$ and $\overline{S}_1 + \dots + \overline{S}_n=\Theta(n^{g(\beta)})$ with probability $1-o(1)$ conditioning on the event $\{\sigma\in \Lambda(G,z) ,\dist(\sigma, X) \le n^{\theta}\}$.
This together with Lemma~\ref{lm:corre} at the end of this section\footnote{$\phi_1,\dots,\phi_n$ play the role of $R_1,\dots,R_n$ in Lemma~\ref{lm:corre}. Recall that $\dist(\sigma,X)=\sum_{i=1}^n \phi_i$. Therefore, $P(\dist(\sigma,X)\ge \Theta(n^{g(\beta)})) \ge P(\underline{S}_1 + \dots + \underline{S}_n\ge\Theta(n^{g(\beta)}))=1-o(1)$ and $P(\dist(\sigma,X)\le \Theta(n^{g(\beta)})) \ge P(\overline{S}_1 + \dots + \overline{S}_n\le\Theta(n^{g(\beta)}))=1-o(1)$. Thus we conclude that $\dist(\sigma,X)=\Theta(n^{g(\beta)})$ with probability $1-o(1)$ conditioning on the event $\{\sigma\in \Lambda(G,z) ,\dist(\sigma, X) \le n^{\theta}\}$.} implies that $\dist(\sigma,X)=\Theta(n^{g(\beta)})$ with probability $1-o(1)$ conditioning on the event $\{\sigma\in \Lambda(G,z) ,\dist(\sigma, X) \le n^{\theta}\}$. Combining this with \eqref{eq:hs}, we obtain that for every $G\in\cG_{\good}\cap\cG_{\con}$, $\dist(\sigma,X)=\Theta(n^{g(\beta)})$ with probability $1-o(1)$ conditioning on $\{\dist(\sigma, X) \le n/2\}$. Finally, the theorem follows from $P(G\in\cG_{\good}\cap\cG_{\con})=1-o(1)$.
\end{proof}

For $0<\beta<\beta^\ast$, we have $g(\beta)>0$, so Theorem~\ref{thm:dist} immediately implies that $P_{\SIBM}(\sigma= \pm X)=o(1)$. However, when $\beta=\beta^\ast$, we have $g(\beta^\ast)=0$. In this case, Theorem~\ref{thm:dist} tells us that
$
P_{\SIBM}(\dist(\sigma,\pm X)= \Theta(1) ) = 1-o(1) ,
$
but this is not sufficient for us to draw any conclusion on $P_{\SIBM}(\sigma= \pm X)$.
Below we use Proposition~\ref{prop:con} to prove that when $\beta=\beta^\ast$, $P_{\SIBM}(\sigma= \pm X)$ is bounded away from $1$.

\begin{proposition}  \label{prop:zj}
Let $a,b,\alpha> 0$ be constants satisfying that $\sqrt{a}-\sqrt{b} > \sqrt{2}$ and $\alpha>b\beta^\ast$. Let 
$
(X,G,\sigma) \sim \SIBM(n,a\log(n)/n, b\log(n)/n,\alpha,\beta^\ast, 1) .
$
Then 
$$
P_{\SIBM}(\sigma= \pm X) \le \frac{1}{2}(1+o(1)) .
$$
\end{proposition}
\begin{proof}
We prove an equivalent form
$$
P_{\SIBM}(\sigma= X | \dist(\sigma, X) \le n/2 ) \le \frac{1}{2}(1+o(1)) .
$$
By \eqref{eq:isingma}, we have
\begin{align*}
\frac{P_{\sigma|G}(\sigma=X^{(\sim i)} )}
{P_{\sigma|G}(\sigma=X)}
& = \exp\Big(2\big(\beta^\ast+\frac{\alpha\log(n)}{n} \big) (B_i-A_i)
-\frac{2\alpha\log(n)}{n} \Big) \\
& = (1+o(1)) \exp\big(2\beta^\ast (B_i-A_i) \big)  ,
\end{align*}
where the last equality holds for almost all $G$ since $|B_i-A_i|=O(\log(n))$ with probability $1-o(1)$; see Lemma~\ref{lm:bmd} in Appendix~\ref{ap:6}.
Take $\cG_{\con}$ from Proposition~\ref{prop:con}. Then for every $G\in\cG_{\con}$,
\begin{align*}
\frac{\sum_{i=1}^n P_{\sigma|G}(\sigma=X^{(\sim i)} )}
{P_{\sigma|G}(\sigma=X)}  = (1+o(1)) \sum_{i=1}^n \exp\big(2\beta^\ast (B_i-A_i) \big)  =1+o(1) .
\end{align*}
As a consequence, for every $G\in\cG_{\con}$,
\begin{align*}
P_{\sigma|G} (\sigma= X | \dist(\sigma, X) \le n/2 ) <
\frac {P_{\sigma|G}(\sigma=X)}
{P_{\sigma|G}(\sigma=X)+\sum_{i=1}^n P_{\sigma|G}(\sigma=X^{(\sim i)} )}
=\frac{1}{2} (1+o(1)) .
\end{align*}
By Proposition~\ref{prop:con}, $P(G\in\cG_{\con})=1-o(1)$, so
$$
P_{\SIBM}(\sigma= X | \dist(\sigma, X) \le n/2 ) \le \frac{1}{2}(1+o(1)) .
$$
\end{proof}

\begin{lemma} \label{lm:corre}
Let $R_1,\dots,R_n, \underline{S}_1,\dots, \underline{S}_n, \overline{S}_1,\dots,\overline{S}_n$ be $3n$ Bernoulli random variables. Suppose that $\underline{S}_1,\dots, \underline{S}_n$ are independent, and that $\overline{S}_1,\dots,\overline{S}_n$ are independent. Further assume that 
\begin{equation} \label{eq:srcd}
P(\underline{S}_i=1) \le
P(R_i=1|R_j=r_j \text{~for all~} j\neq i)
\le P(\overline{S}_i=1)
\end{equation}
for all $i\in[n]$ and all $(r_1,\dots,r_n)\in\{0,1\}^n$.
Let $R=R_1+\dots+R_n, \underline{S}=\underline{S}_1+\dots+ \underline{S}_n, \overline{S}=\overline{S}_1+\dots+\overline{S}_n$. Then
$$
P(R\ge k)\ge P(\underline{S}\ge k)
\text{~and~}
P(R\le k)\ge P(\overline{S}\le k)
\quad \quad
\text{for all~} k\in[n] .
$$
\end{lemma}
\begin{proof}
We first prove $P(R\ge k)\ge P(\underline{S}\ge k)$ for all $k\in[n]$.
We will recursively construct Bernoulli random variables $(R_1^{(j)}, \dots, R_n^{(j)})$ for $j=1,2,\dots,n$ such that $(R_1^{(n)}, \dots, R_n^{(n)})$ has the same joint distribution as $(\underline{S}_1,\dots, \underline{S}_n)$. For each $j\in[n]$, define $R^{(j)}=R_1^{(j)} + \dots + R_n^{(j)}$. We will prove that 
\begin{equation} \label{eq:tg1}
P(R\ge k)\ge P(R^{(1)}\ge k)\ge P(R^{(2)}\ge k) \ge \dots \ge P(R^{(n)}\ge k)= P(\underline{S}\ge k)
\text{~for all~} k\in[n] .
\end{equation}

We start with the construction of $(R_1^{(1)}, \dots, R_n^{(1)})$: Let the joint distribution of $R_2^{(1)}, \dots, R_n^{(1)}$ be exactly the same as that of $R_2, \dots, R_n$. Let $R_1^{(1)}$ be independent of $R_2^{(1)}, \dots, R_n^{(1)}$ with marginal probability $P(R_1^{(1)}=1)=P(\underline{S}_1=1)$.
Then we have
\begin{align*}
& P(R\ge k) \\
= & P(R_2+\dots+R_n\ge k) + P(R_2+\dots+R_n = k-1) ~ P(R_1=1 | R_2+\dots+R_n = k-1) \\
= & P(R_2^{(1)}+ \dots+ R_n^{(1)} \ge k)
+ P(R_2^{(1)}+ \dots+ R_n^{(1)} = k-1) ~ P(R_1=1 | R_2+\dots+R_n = k-1) \\
\ge & P(R_2^{(1)}+ \dots+ R_n^{(1)} \ge k)
+ P(R_2^{(1)}+ \dots+ R_n^{(1)} = k-1) ~ P(R_1^{(1)}=1) \\
= & P(R^{(1)}\ge k) ,
\end{align*}
where the second equality follows from the assumption that $(R_2^{(1)}, \dots, R_n^{(1)})$ and $(R_2, \dots, R_n)$ have the same joint distribution, and the inequality follows from \eqref{eq:srcd}.

Now suppose that we have constructed $(R_1^{(j-1)}, \dots, R_n^{(j-1)})$. Then we construct $(R_1^{(j)}, \dots, R_n^{(j)})$ as follows: Let $R_1^{(j)}, \dots, R_{j-1}^{(j)}, R_{j+1}^{(j)}, \dots R_n^{(j)}$ have exactly the same joint distribution as $R_1^{(j-1)}, \dots, R_{j-1}^{(j-1)}, \linebreak[4] R_{j+1}^{(j-1)}, \dots, R_n^{(j-1)}$. Let $R_j^{(j)}$ be independent of $R_1^{(j)}, \dots, R_{j-1}^{(j)}, R_{j+1}^{(j)}, \dots R_n^{(j)}$ with marginal probability  $P(R_j^{(j)}=1)=P(\underline{S}_j=1)$. By induction one can easily see that the joint distribution of $R_j^{(j-1)}, R_{j+1}^{(j-1)}, \dots, R_n^{(j-1)}$ is the same as the joint distribution of $R_j, R_{j+1}, \dots R_n$. Moreover, $R_1^{(j-1)}, \dots, \linebreak[4] R_{j-1}^{(j-1)}$ are independent of $R_j^{(j-1)}, R_{j+1}^{(j-1)}, \dots, R_n^{(j-1)}$. Therefore, for all $(r_1,\dots,r_n)\in\{0,1\}^n$,
\begin{equation} \label{eq:rotj}
\begin{aligned}
& P(R_j^{(j-1)}=1 | R_i^{(j-1)}=r_i \text{~for all~} i\neq j)
= P(R_j^{(j-1)}=1 | R_i^{(j-1)}=r_i \text{~for all~} i> j) \\
= & P(R_j=1 | R_i=r_i \text{~for all~} i> j)
\ge P(\underline{S}_j=1)
= P(R_j^{(j)}=1) ,
\end{aligned}
\end{equation}
where the inequality follows from \eqref{eq:srcd}. This further implies that
\begin{align*}
& P(R^{(j-1)}\ge k) \\
= & P(\sum_{i\in[n]\setminus\{j\}} R_i^{(j-1)} \ge k) + P( \sum_{i\in[n]\setminus\{j\}} R_i^{(j-1)} = k-1) ~ P(R_j^{(j-1)}=1 | \sum_{i\in[n]\setminus\{j\}} R_i^{(j-1)} = k-1) \\
= & P(\sum_{i\in[n]\setminus\{j\}} R_i^{(j)} \ge k) + P( \sum_{i\in[n]\setminus\{j\}} R_i^{(j)} = k-1) ~ P(R_j^{(j-1)}=1 | \sum_{i\in[n]\setminus\{j\}} R_i^{(j-1)} = k-1) \\
\ge & P(\sum_{i\in[n]\setminus\{j\}} R_i^{(j)} \ge k) + P( \sum_{i\in[n]\setminus\{j\}} R_i^{(j)} = k-1) ~ P(R_j^{(j)}=1) \\
= & P(R^{(j)}\ge k) ,
\end{align*}
where the inequality follows from \eqref{eq:rotj}. By noticing that $(R_1^{(n)}, \dots, R_n^{(n)})$ has the same joint distribution as $(\underline{S}_1,\dots, \underline{S}_n)$, we conclude the proof of \eqref{eq:tg1}.

In order to prove $P(R\le k)\ge P(\overline{S}\le k)$ for all $k\in[n]$, we observe that 
$$
P(1-R_i=1|R_j=r_j \text{~for all~} j\neq i)
\ge P(1-\overline{S}_i=1)
$$
for all $i\in[n]$ and all $(r_1,\dots,r_n)\in\{0,1\}^n$. Applying the conclusion above to the Bernoulli random variables $1-R_1,\dots,1-R_n$ and $1-\overline{S}_1,\dots,1-\overline{S}_n$, we obtain that $P(n-R\ge k)\ge P(n-\overline{S}\ge k)$ for all $k\in[n]$, so $P(R\le k)\ge P(\overline{S}\le k)$ for all $k\in[n]$.
\end{proof}

\section{Exact recovery is not solvable when $\lfloor \frac{m+1}{2} \rfloor \beta < \beta^\ast$}
\label{sect:converse}

\begin{lemma} \label{lm:qq}
Let 
$$
(X,G,\{\sigma^{(1)},\dots,\sigma^{(m)}\})\sim \SIBM(n,a\log(n)/n, b\log(n)/n,\alpha,\beta, m) .
$$
If there is a pair $i,i'\in[n]$ satisfying the following two conditions: (1) $\sigma_i^{(j)}=\sigma_{i'}^{(j)}$ for all $j\in[m]$ and (2) $X_i=-X_{i'}$, then it is not possible to distinguish the case $X_i=-X_{i'}=1$ from the case $X_i=-X_{i'}=-1$. In other words, conditioning on the samples, the posterior probability of the ground truth being $X$ is the same as that of the ground truth being $X^{(\sim\{i,i'\})}$.
\end{lemma}
\begin{proof}
Intuitively, this proposition clearly holds. For a rigorous proof, one can follow the same steps as the proof of Lemma~\ref{lm:cc} in Appendix~\ref{ap:6}, and we do not repeat it here.
\end{proof}

Clearly, the original samples satisfy the above conditions if and only if the aligned samples satisfy the above conditions in Lemma~\ref{lm:qq}.

\begin{lemma}  \label{lm:cvs}
Let $v=(v^{(1)},v^{(2)},\dots,v^{(m)})\in\{\pm 1\}^m$ be a vector of length $m$ and let $u:=|\{i\in[m]:v^{(i)}=-1\}|$ be the number of $-1$'s in $v$.
Let $\sigma^{(1)},\dots,\sigma^{(m)}$ be the aligned samples of SIBM (see the alignment step in Algorithm~\ref{alg:ez}).
Without loss of generality we assume that the aligned samples satisfy  $\dist(\sigma^{(j)}, X) \le n/2$ for all $j\in[m]$.
Define 
\begin{equation}  \label{eq:rl}
\begin{aligned}
& T_+:= | \{i\in[n]:(\sigma_i^{(1)},\dots,\sigma_i^{(m)})=v, X_i=1\} | ,  \\
& T_-:= | \{i\in[n]:(\sigma_i^{(1)},\dots,\sigma_i^{(m)})=v, X_i=-1\} | .
\end{aligned}
\end{equation}
If $u\beta<\beta^\ast$, then $P_{\SIBM}\big(T_+ = \Theta(n^{g(u\beta)}) \big) = 1-o(1)$. 
Similarly, if $(m-u)\beta<\beta^\ast$,
then $P_{\SIBM}\big(T_- = \Theta(n^{g((m-u)\beta)}) \big) = 1-o(1)$. 
\end{lemma}
\begin{proof}


We only prove the claims about $T_+$ since the proof for $T_-$ is virtually identical.
The proof follows the same steps as the proof of Theorem~\ref{thm:dist}. All we need to do is to replace $\beta$ with $u\beta$ in the proof of Theorem~\ref{thm:dist}. For the sake of completeness, we provide the proof here.

Corollary~\ref{cr:1} together with Proposition~\ref{prop:43} implies that there is an integer $z>0$ and a set $\cG_{\good}$ such that

\noindent (i)
$P(G\in\cG_{\good}) \ge 1- O(n^{-4})$.

\noindent (ii) For every $G\in\cG_{\good}$, 
\begin{equation}  \label{eq:ljrc}
\begin{aligned}
& P_{\sigma|G} \big( \sigma^{(j)}\in  \Lambda(G, z)
\text{~and~} \dist(\sigma^{(j)}, X) \le n^\theta
\text{~for all~} j\in[m] ~\big| \dist(\sigma^{(j)}, X) \le n/2
\text{~for all~} j\in[m]  \big) \\
& =1- O(n^{-4}) ,
\end{aligned}
\end{equation}
where we can choose $\theta$ to be any constant in the open interval $(g(\beta), 1)$.
By Lemma~\ref{lm:et} we know that for all $i\in[n]$ and all $j\in[m]$,
$P_{\sigma|G}(\sigma_i^{(j)} \neq X_i \big| \sigma^{(j)}\in \Lambda(G,z) ,\dist(\sigma^{(j)}, X) \le n^{\theta} )$ differ from
$\exp\Big(2\big(\beta+\frac{\alpha\log(n)}{n} \big) (B_i-A_i) \Big)$
by at most a constant factor. Since $|B_i-A_i|=O(\log(n))$ with probability $1-o(1)$, the term $\frac{\alpha\log(n)}{n}(B_1-A_i)=o(1)$ and is negligible. Moreover, since the $m$ samples are independent given the graph $G$, we conclude that
\begin{align*}
& \underline{C}
\exp\big(2u\beta (B_i-A_i) \big) \\
\le
& P_{\sigma|G} \big( (\sigma_i^{(1)},\dots,\sigma_i^{(m)})=v ~\big|~ \sigma^{(j)}\in  \Lambda(G, z)
\text{~and~} \dist(\sigma^{(j)}, X) \le n^\theta
\text{~for all~} j\in[m] \big) \\
\le & \overline{C}
\exp\big(2u\beta (B_i-A_i) \big)
\end{align*}
for all $i\in[n]$, where $\underline{C}$ and $\overline{C}$ are constants that are independent of $n$.
In fact, we obtained a much stronger inequality \eqref{eq:qk} in the proof of Lemma~\ref{lm:et}. More precisely, inequality \eqref{eq:qk} can be reformulated as follows: For every $\bar{\sigma}^{(1)},\dots,\bar{\sigma}^{(m)}\in\{\pm 1\}^n$ such that $\bar{\sigma}^{(j)}\in  \Lambda(G, z)$
and $\dist(\bar{\sigma}^{(j)}, X) \le n^\theta$
for all $j\in[m]$,
\begin{equation} \label{eq:s56}
\begin{aligned}
& \underline{C}
\exp\big(2u\beta (B_i-A_i) \big) \\
\le
& P_{\sigma|G} \big( (\sigma_i^{(1)},\dots,\sigma_i^{(m)})=v ~\big|~ (\sigma_{i'}^{(1)},\dots,\sigma_{i'}^{(m)})= (\bar{\sigma}_{i'}^{(1)},\dots,\bar{\sigma}_{i'}^{(m)}) \text{~for all~} i'\neq i \big) \\
\le & \overline{C}
\exp\big(2u\beta (B_i-A_i) \big).
\end{aligned}
\end{equation}
Define 
$\phi_i := \mathbbm{1}[(\sigma_i^{(1)},\dots,\sigma_i^{(m)})=v]$ for $i\in[n]$.
In some sense, \eqref{eq:s56} tells us that conditioning on the event $\{\sigma^{(j)}\in  \Lambda(G, z)
\text{~and~} \dist(\sigma^{(j)}, X) \le n^\theta
\text{~for all~} j\in[m]\}$, the random variables $\phi_1,\dots,\phi_n$ are ``almost" independent.
Given a fixed graph $G$ and random samples $\sigma^{(1)},\dots,\sigma^{(m)}$, we define {\bf Bernoulli} random variables $\underline{S}_1,\dots, \underline{S}_n$ and $\overline{S}_1,\dots,\overline{S}_n$ with the following properties:
\begin{enumerate}
\item $\underline{S}_1,\dots, \underline{S}_n$ are conditionally independent given the event $\{\sigma^{(j)}\in  \Lambda(G, z)
\text{~and~} \dist(\sigma^{(j)}, X) \le n^\theta
\text{~for all~} j\in[m]\}$. $\overline{S}_1,\dots,\overline{S}_n$ are also conditionally independent given the event $\{\sigma^{(j)}\in  \Lambda(G, z)
\text{~and~} \dist(\sigma^{(j)}, X) \le n^\theta
\text{~for all~} j\in[m]\}$.
\item Conditioning on the event $\{\sigma^{(j)}\in  \Lambda(G, z)
\text{~and~} \dist(\sigma^{(j)}, X) \le n^\theta
\text{~for all~} j\in[m]\}$, $P(\underline{S}_i=1)=\underline{C}
\exp\big(2u\beta (B_i-A_i) \big)$ and $P(\overline{S}_i=1)=\overline{C}
\exp\big(2u\beta (B_i-A_i) \big)$ for all $i\in[n]$.
\end{enumerate}
Define a set $[n]_+:=\{i\in[n]:X_i=1\}$.
Then $T_+=\sum_{i\in[n]_+} \phi_i$.
By the two properties above, conditioning on the event $\{\sigma^{(j)}\in  \Lambda(G, z)
\text{~and~} \dist(\sigma^{(j)}, X) \le n^\theta
\text{~for all~} j\in[m]\}$, we have
\begin{align*}
& E\Big[ \sum_{i\in[n]_+}\underline{S}_i \Big] = \underline{C}
\sum_{i\in[n]_+} \exp\big(2u\beta (B_i-A_i) \big) , \quad
\Var \Big(\sum_{i\in[n]_+}\underline{S}_i \Big) \le \underline{C}
\sum_{i\in[n]_+} \exp\big(2u\beta (B_i-A_i) \big) , \\
& E \Big[\sum_{i\in[n]_+} \overline{S}_i \Big] = \overline{C}
\sum_{i\in[n]_+} \exp\big(2u\beta (B_i-A_i) \big) , \quad
\Var \Big(\sum_{i\in[n]_+} \overline{S}_i \Big) \le \overline{C}
\sum_{i\in[n]_+} \exp\big(2u\beta (B_i-A_i) \big) ,
\end{align*}
where we use the fact that the variance of a Bernoulli random variable is always upper bounded by its expectation.
In Proposition~\ref{prop:con}, we have shown that for $u\beta\le \beta^\ast$, there is a set $\cG_{\con}$ such that (i) $P(G \in \cG_{\con}) = 1-o(1)$ and (ii) for every $G\in \cG_{\con}$, 
$
\sum_{i=1}^n \exp\big(2u\beta (B_i-A_i) \big)
=(1+o(1)) n^{g(u\beta)} .
$
Using exactly the same method, one can show that there is another set $\cG_{\con}'$ such that (i) $P(G \in \cG_{\con}') = 1-o(1)$ and (ii) for every $G\in \cG_{\con}'$, 
$
\sum_{i\in[n]_+} \exp\big(2u\beta (B_i-A_i) \big)
=\Theta (n^{g(u\beta)}) .
$
Therefore, for every $G\in\cG_{\good}\cap\cG_{\con}'$, 
\begin{align*}
& E \Big[ \sum_{i\in[n]_+}\underline{S}_i \Big] = \Theta(n^{g(u\beta)}) , \quad
\Var \Big( \sum_{i\in[n]_+}\underline{S}_i \Big) = O(n^{g(u\beta)}) , \\
& E \Big[ \sum_{i\in[n]_+}\overline{S}_i \Big] = \Theta(n^{g(u\beta)}) , \quad
\Var \Big( \sum_{i\in[n]_+}\overline{S}_i \Big) = O(n^{g(u\beta)}) 
\end{align*}
conditioning on the event $\{\sigma^{(j)}\in  \Lambda(G, z)
\text{~and~} \dist(\sigma^{(j)}, X) \le n^\theta
\text{~for all~} j\in[m]\}$. Since $g(u\beta)> 0$ for all $u\beta< \beta^\ast$, by Chebyshev's inequality we know that both $\sum_{i\in[n]_+}\underline{S}_i=\Theta(n^{g(u\beta)})$ and $\sum_{i\in[n]_+}\overline{S}_i=\Theta(n^{g(u\beta)})$ with probability $1-o(1)$ conditioning on the event $\{\sigma^{(j)}\in  \Lambda(G, z)
\text{~and~} \dist(\sigma^{(j)}, X) \le n^\theta
\text{~for all~} j\in[m]\}$.
This together with Lemma~\ref{lm:corre} in the previous section implies that $T_+=\Theta(n^{g(u\beta)})$ with probability $1-o(1)$ conditioning on the event $\{\sigma^{(j)}\in  \Lambda(G, z)
\text{~and~} \dist(\sigma^{(j)}, X) \le n^\theta
\text{~for all~} j\in[m]\}$. Combining this with \eqref{eq:ljrc}, we obtain that for every $G\in\cG_{\good}\cap\cG_{\con}'$, $T_+=\Theta(n^{g(u\beta)})$ with probability $1-o(1)$ conditioning on $\dist(\sigma^{(j)}, X) \le n/2$ for all $j\in[m]$. Finally, the lemma follows from $P(G\in\cG_{\good}\cap\cG_{\con}')=1-o(1)$.
\end{proof}

Using the above two lemmas, we obtain the following proposition
\begin{proposition}
Let $a,b,\alpha,\beta> 0$ be constants satisfying that $\sqrt{a}-\sqrt{b} > \sqrt{2}$ and $\alpha>b\beta$. 
Let 
$$
(X,G,\{\sigma^{(1)},\dots,\sigma^{(m)}\})\sim \SIBM(n,a\log(n)/n, b\log(n)/n,\alpha,\beta, m) .
$$
If $\lfloor \frac{m+1}{2} \rfloor \beta<\beta^\ast$, then no algorithm can recover $X$ from the samples with constant success probability, i.e., the success probability of any algorithm is $o(1)$.
\end{proposition}
\begin{proof}
First observe that $m-\lfloor \frac{m+1}{2} \rfloor \le \lfloor \frac{m+1}{2} \rfloor$. If $\lfloor \frac{m+1}{2} \rfloor \beta<\beta^\ast$, then $(m-\lfloor \frac{m+1}{2} \rfloor) \beta<\beta^\ast$.
Now pick a vector $v\in\{\pm 1\}^m$ such that it has $\lfloor \frac{m+1}{2} \rfloor$ coordinates being $-1$ and $(m-\lfloor \frac{m+1}{2} \rfloor)$ coordinates being $1$. Then by Lemma~\ref{lm:cvs}, with probability $1-o(1)$, both $T_+$ and $T_-$ are $\omega(1)$; see definition of $T_+$ and $T_-$ in \eqref{eq:rl}. Therefore, we can find $\omega(1)$ pairs of $i,i'\in[n]$ satisfying the following two conditions: (1) $\sigma_i^{(j)}=\sigma_{i'}^{(j)}$ for all $j\in[m]$ and (2) $X_i=-X_{i'}$. Then by Lemma~\ref{lm:qq}, it is not possible to distinguish $X$ from $X^{(\sim\{i,i'\})}$ for $\omega(1)$ pairs of $i,i'$. Therefore, the success probability of any recovery algorithm is $o(1)$.
\end{proof}

\section{Future directions}
\label{sect:future}

We conclude this paper with two future directions. First, in this paper we mainly focus on the regime $\alpha>b\beta$, where we establish a sharp threshold on the sample complexity. When $\alpha<b\beta$, we only give a lower bound $\Omega(\log^{1/4}(n))$ on the sample complexity, which is almost surely not tight. An interesting future direction is to find the optimal sample complexity when $\alpha<b\beta$.
Second, in this paper we assume that there are only two communities/clusters in the SSBM. A natural future direction is to extend the result in this paper to $k$-community SSBM. In this case, we also need to extend the Ising model from binary alphabet to general alphabet. Such an extension of Ising model has been considered, for example, in \cite{Klivans17}.

\appendices

\section{Auxiliary lemmas used in Section~\ref{sect:direct}}
\label{ap:6}

\begin{lemma} \label{lm:bq}
For $0<\theta<1$,
\begin{align}
& P_{\SIBM}(\sigma_i=-X_i
\big| \dist(\sigma,X) \le n^\theta) \le 2 n^{\theta -1}
\quad \text{for all~} i\in[n] , \label{eq:l1}\\
& P_{\SIBM}(\sigma_i=-X_i \text{~for all~}  i\in\tilde{\cI}
~\big| \dist(\sigma,X) \le n^\theta) \le \Big(\frac{n^\theta}{n/2-|\tilde{\cI}|}
\Big)^{|\tilde{\cI}|}
\quad \text{for all~} \tilde{\cI}\subseteq [n] .   \label{eq:l2}
\end{align}
\end{lemma}
\begin{proof}
Define $\dist_+(\sigma,X):=|\{i\in[n]:X_i=1, \sigma_i=-1\}|$ and $\dist_-(\sigma,X):=|\{i\in[n]:X_i=-1, \sigma_i=1\}|$. Clearly, $\dist(\sigma,X)=\dist_+(\sigma,X)+\dist_-(\sigma,X)$. Inequality \eqref{eq:l1} follows immediately from the following equality:
$$
P_{\SIBM}(\sigma_i=-X_i|
\dist_+(\sigma,X)=u_+,
\dist_-(\sigma,X)=u_-) 
=\left\{
\begin{array}{cc}
  2u_+/n   & \mbox{if~} X_i=1 \\
  2u_-/n   & \mbox{if~} X_i=-1
\end{array}
\right.
$$
Without loss of generality,
we only prove the case of $X_i=1$, and
we need the following definition for the proof of this equality:
For $\cI\subseteq[n]$, define $\cI_+:=\{i\in\cI:X_i=+1\}$ and $\cI_-:=\{i\in\cI:X_i=-1\}$.
Then
\begin{align*}
& P_{\SIBM}(\sigma_i=-X_i|
\dist_+(\sigma,X)=u_+,
\dist_-(\sigma,X)=u_-)  \\
= & \frac{P_{\SIBM}(\sigma_i=-X_i ,
\dist_+(\sigma,X)=u_+,
\dist_-(\sigma,X)=u_-)}{P_{\SIBM}(
\dist_+(\sigma,X)=u_+,
\dist_-(\sigma,X)=u_-)} \\
= & \frac{\sum_{\cI:i\in\cI_+,|\cI_+|=u_+,|\cI_-|=u_-}P_{\SIBM}(\sigma=X^{(\sim\cI)}) }
{\sum_{\cI:|\cI_+|=u_+,|\cI_-|=u_-}P_{\SIBM}(\sigma=X^{(\sim\cI)}) } \\
\overset{(a)}{=} & \frac{\binom{n/2-1}{u_+ -1} \binom{n/2}{u_-}}
{\binom{n/2}{u_+} \binom{n/2}{u_-}}
= 2u_+/n ,
\end{align*}
where equality (a) follows from Lemma~\ref{lm:cc} below.
Similarly, inequality \eqref{eq:l2} follows from the following inequality:
For $u_+\ge |\tilde{\cI}_+|$ and $u_-\ge |\tilde{\cI}_-|$,
\begin{align*}
& P_{\SIBM}(\sigma_i=-X_i \text{~for all~}  i\in\tilde{\cI} ~ \big|
\dist_+(\sigma,X)=u_+,
\dist_-(\sigma,X)=u_-)  \\
= & \frac{P_{\SIBM}(\sigma_i=-X_i \text{~for all~}  i\in\tilde{\cI} ,
\dist_+(\sigma,X)=u_+,
\dist_-(\sigma,X)=u_-)}{P_{\SIBM}(
\dist_+(\sigma,X)=u_+,
\dist_-(\sigma,X)=u_-)} \\
= & \frac{\sum_{\cI:\tilde{\cI}_+\subseteq\cI_+,
\tilde{\cI}_-\subseteq\cI_-,
|\cI_+|=u_+,|\cI_-|=u_-}P_{\SIBM}(\sigma=X^{(\sim\cI)}) }
{\sum_{\cI:|\cI_+|=u_+,|\cI_-|=u_-}P_{\SIBM}(\sigma=X^{(\sim\cI)}) }  \\
\overset{(a)}{=} & \frac{\binom{n/2-|\tilde{\cI}_+|}{u_+ -|\tilde{\cI}_+|} \binom{n/2 - |\tilde{\cI}_-|}{u_- - |\tilde{\cI}_-|}}
{\binom{n/2}{u_+} \binom{n/2}{u_-}}
< \Big(\frac{u_+}{n/2- |\tilde{\cI}_+|} \Big)^{|\tilde{\cI}_+|}
\Big(\frac{u_-}{n/2- |\tilde{\cI}_-|} \Big)^{|\tilde{\cI}_-|} \\
< & \Big(\frac{u_+ + u_-}{n/2- |\tilde{\cI}|} \Big)^{|\tilde{\cI}|}  ,
\end{align*}
where equality (a) again follows from Lemma~\ref{lm:cc} below.
\end{proof}

\begin{lemma} \label{lm:cc}
Let $\cI,\cI'\subseteq[n]\setminus\{i\}$ be two subsets such that $|\cI_+|=|\cI_+'|$ and $|\cI_-|=|\cI_-'|$. 
Then $P_{\SIBM}(\sigma=X^{(\sim\cI)}) = P_{\SIBM}(\sigma=X^{(\sim\cI')})$.
\end{lemma}

\begin{proof}
Let $\cG_{[n]}$ be the set consisting of all the graphs with vertex set $[n]$.
A permutation $\pi\in S_n$ on the vertex set $[n]$ also induces a permutation on $\cG_{[n]}$: For $G\in\cG_{[n]}$, define the graph $\pi(G)\in\cG_{[n]}$ as the graph with the edge set $E(\pi(G))$ satisfying that $\{\pi(i),\pi(j)\}\in E(\pi(G))$ if and only if $\{i,j\}\in E(G)$.
It is easy to see that for any $\pi\in S_n$ and any $G\in\cG_{[n]}$,
$$
Z_G(\alpha,\beta)
=Z_{\pi(G)}(\alpha,\beta),
$$
where $Z_G(\alpha,\beta)$ is defined in \eqref{eq:zg}.
Furthermore, 
if $X_i=X_{\pi(i)}$ for all $i\in[n]$, then for any graph $G\in\cG_{[n]}$, we have
$$
P_{\SSBM}(G)=P_{\SSBM}(\pi(G))  ,
$$
where $P_{\SSBM}$ is the distribution given in Definition~\ref{def:SSBM}.

Under the assumptions $|\cI_+|=|\cI_+'|$ and $|\cI_-|=|\cI_-'|$,
it is easy to see that there exists a permutation $\pi$ on the vertex set $[n]$ satisfying the following two conditions: 
(i) $X_i=X_{\pi(i)}$ for all $i\in[n]$; 
(ii) $\pi(\cI)=\cI'$, i.e., $\pi(i)\in\cI'$ for all $i\in\cI$.
For such a permutation $\pi$, one can verify that
$$
X_i^{(\sim\cI)} = X_{\pi(i)}^{(\sim\pi(\cI))}
= X_{\pi(i)}^{(\sim \cI')}
$$
for all $i\in[n]$.
Therefore,
\begin{align*}
& P_{\SIBM}(\sigma=X^{(\sim\cI)})
=\sum_{G\in\cG_{[n]}}
P_{\SSBM}(G) P_{\sigma|G}(\sigma=X^{(\sim\cI)}) \\
= & \sum_{G\in\cG_{[n]}} P_{\SSBM}(G) \frac{1}{Z_G(\alpha,\beta)}
\exp\Big(\beta\sum_{\{i,j\}\in E(G)} X_i^{(\sim\cI)} X_j^{(\sim\cI)}
-\frac{\alpha\log(n)}{n} \sum_{\{i,j\}\notin E(G)} X_i^{(\sim\cI)} X_j^{(\sim\cI)} \Big)  \\
= & \sum_{G\in\cG_{[n]}} P_{\SSBM}(G) \frac{1}{Z_G(\alpha,\beta)}  \\
& \hspace*{1in}
\exp\Big(\sum_{\{i,j\}} X_i^{(\sim\cI)} X_j^{(\sim\cI)}
\Big( \big( \beta+\frac{\alpha\log(n)}{n} \big) \mathbbm{1}[\{i,j\}\in E(G)]
-\frac{\alpha\log(n)}{n} \Big)
 \Big)  \\
= & \sum_{G\in\cG_{[n]}} P_{\SSBM}(G) \frac{1}{Z_G(\alpha,\beta)}  \\
& \hspace*{0.6in}
\exp\Big(\sum_{\{i,j\}} X_{\pi(i)}^{(\sim \cI')} X_{\pi(j)}^{(\sim \cI')}
\Big( \big( \beta+\frac{\alpha\log(n)}{n} \big) \mathbbm{1}[\{\pi(i),\pi(j)\}\in E(\pi(G))]
-\frac{\alpha\log(n)}{n} \Big)
 \Big)   \\
 = & \sum_{G\in\cG_{[n]}} P_{\SSBM}(G) \frac{1}{Z_G(\alpha,\beta)}  \\
& \hspace*{0.8in}
\exp\Big(\sum_{\{i,j\}} X_i^{(\sim \cI')} X_j^{(\sim \cI')}
\Big( \big( \beta+\frac{\alpha\log(n)}{n} \big) \mathbbm{1}[\{i, j\}\in E(\pi(G))]
-\frac{\alpha\log(n)}{n} \Big)
 \Big)   \\
  = & \sum_{G\in\cG_{[n]}} P_{\SSBM}(\pi(G)) \frac{1}{Z_{\pi(G)}(\alpha,\beta)}  \\
& \hspace*{0.8in}
\exp\Big(\sum_{\{i,j\}} X_i^{(\sim \cI')} X_j^{(\sim \cI')}
\Big( \big( \beta+\frac{\alpha\log(n)}{n} \big) \mathbbm{1}[\{i, j\}\in E(\pi(G))]
-\frac{\alpha\log(n)}{n} \Big)
 \Big)  \\
= & \sum_{G\in\cG_{[n]}}
P_{\SSBM}(\pi(G)) P_{\pi(G)}(\sigma=X^{(\sim\cI')}) \\
= & P_{\SIBM}(\sigma=X^{(\sim\cI')}) .
\end{align*}
This completes the proof of the lemma.
\end{proof}

\begin{lemma} \label{lm:bmd}
Let $Y\sim \Binom(n + o(n), a\log(n)/n)$. Then for $r>8$ and large enough $n$, we have
$$
P(Y\ge r a \log(n)) < n^{-r} .
$$
\end{lemma}
\begin{proof}
The moment generating function of $Y$ is
$$
E[e^{sY}]
= (1-a\log(n)/n+e^s a\log(n)/n)^{n+o(n)}
= \exp\big(\log(n) \big(ae^s-a+ o(1) \big) \big) ,
$$
where we use Taylor expansion $\log(1+x)=x+o(x)$ for $x\to 0$ in the last equality.
Then by Chernoff bound, for $s>0$,
$$
P(Y\ge r a \log(n))
\le \exp((e^s-1-rs + o(1))a\log(n) ) .
$$
Taking $s=\log(r)>0$ into this bound, we obtain that
$$
P(Y\ge r a \log(n))
\le \exp((r-1-r\log(r) + o(1))a\log(n) )
< n^{-r} ,
$$
where the last inequality holds for large enough $n$ and follows from the assumption that $r>8>e^2$.
\end{proof}

\section{Auxiliary propositions used in Section~\ref{sect:struct}} \label{ap:um}

We first prove a tight estimate of $P(B_i-A_i = t\log(n))$ for $t=\Theta(1)$. Note that \eqref{eq:upba} gives an upper bound on $P(B_i-A_i \ge t\log(n))$ for $t\in [\frac{1}{2}(b-a), 0]$, so it is also an upper bound of $P(B_i-A_i = t\log(n))$. Below we prove that the upper bound in \eqref{eq:upba} is in fact tight up to a $\Theta(1/ \sqrt{\log(n)})$ factor for all $t=\Theta(1)$ such that $t\log(n)$ is an integer.

We use $f(n)=\Theta(g(n))$ and $f(n)\asymp g(n)$ interchangeably if there is a constant $C>0$ such that $C^{-1}g(n)\le f(n)\le C g(n)$ for large enough $n$.

\begin{proposition}  \label{prop:99}
For any $t$ such that $t\log(n)$ is an integer and $|t|<100a$,
\begin{equation} \label{eq:ly}
\begin{aligned}
& P(B_i-A_i = t\log(n))  \\
\asymp & \frac{1} {\sqrt{\log(n)}} \exp\Big(\log(n)
\Big(\sqrt{t^2+ab} -t\big(\log(\sqrt{t^2+ab}+t)-\log(b) \big) -\frac{a+b}{2}  \Big)\Big) .
\end{aligned}
\end{equation}
\end{proposition}

\begin{proof}
Since
$$
P(B_i-A_i = t\log(n))
= \sum_{s\log(n)=0}^{n/2}
P(B_i = s\log(n)) P(A_i=(s-t)\log(n)) ,
$$
we first calculate tight estimates of $P(B_i = s\log(n))$ and $P(A_i=(s-t)\log(n))$. (The summation $\sum_{s\log(n)=0}^{n/2}$ in the above equation means that the quantity $s\log(n)$ ranges over all integer values from $0$ to $n/2$.)
By Lemma~\ref{lm:bmd} in Appendix~\ref{ap:6}, we only need to focus on the regime where both $|s|$ and $|t|$ are bounded from above by some (large) constants, e.g., $100a$.
Therefore,
\begin{align*}
& P(B_i = s\log(n)) \\
= & \binom{n/2}{s\log(n)}
\Big( \frac{b\log(n)}{n} \Big)^{s\log(n)}
\Big( 1- \frac{b\log(n)}{n} \Big)^{n/2-s\log(n)}  \\
= & (1+o(1))
\frac{(n/2)^{s\log(n)}}{(s\log(n))!} \Big( \frac{b\log(n)}{n} \Big)^{s\log(n)}
\exp \Big(-\frac{b\log(n)}{2} \Big)  \\
= & (1+o(1))
\frac{1} {(s\log(n))!} \Big( \frac{b\log(n)}{2} \Big)^{s\log(n)}
\exp \Big(-\frac{b\log(n)}{2} \Big)  \\
\overset{(a)}{=} & (1+o(1))
\frac{1} {\sqrt{2\pi s\log(n)}}
\Big(\frac{e}{s\log(n)} \Big)^{s\log(n)}
\Big( \frac{b\log(n)}{2} \Big)^{s\log(n)}
\exp \Big(-\frac{b\log(n)}{2} \Big)  \\
= & (1+o(1))
\frac{1} {\sqrt{2\pi s\log(n)}}
\Big( \frac{ e b }{2s} \Big)^{s\log(n)}
\exp \Big(-\frac{b\log(n)}{2} \Big) \\
= & (1+o(1))
\frac{1} {\sqrt{2\pi s\log(n)}}
\exp\Big( \log(n) \Big( 
s+ s\log(b)-s\log(2)-s\log(s)-\frac{b}{2}
\Big)\Big) ,
\end{align*}
where $(a)$ follows from Stirling's formula.
Similarly, when $s > t$,
\begin{align*}
& P(A_i=(s-t)\log(n)) \\
= & (1+o(1))
\frac{1} {\sqrt{2\pi (s-t)\log(n)}}
\exp\Big( \log(n) \Big( 
 (s-t)(1+\log(a)-\log(2)-\log(s-t))-\frac{a}{2}
\Big)\Big)
\end{align*}
Define a function
$$
h_t(s):=(2s-t)(1-\log(2))
+s\log\frac{ab}{s(s-t)}
+t\log\frac{s-t}{a} -\frac{a+b}{2} .
$$
Then for $s>\max(0,t)$,
$$
P(B_i = s\log(n)) P(A_i=(s-t)\log(n))
= (1+o(1)) \frac{1} {2\pi \log(n)\sqrt{s (s-t)}}
\exp(h_t(s) \log(n)) .
$$
Therefore, for $t$ such that $t\log(n)$ is an integer, we have
\begin{equation} \label{eq:gj}
\begin{aligned}
& P(B_i-A_i = t\log(n))  \\
= & (1+o(1)) \sum_{s\log(n)=\max(0,t\log(n))}^{n/2}
\frac{1} {2\pi \log(n)\sqrt{s (s-t)}}
\exp(h_t(s) \log(n)) .
\end{aligned}
\end{equation}
In order to estimate this sum, we need to analyze the function $h_t(s)$. Its first and second derivatives are
$h_t'(s)=\log\frac{ab}{4s(s-t)}$
and $h_t''(s)=-\frac{1}{s}-\frac{1}{s-t}<0$, so $h_t(s)$ is a concave function and takes maximum at $s^\ast$ such that $h_t'(s^\ast)=0$.
Simple calculations show that
$s^\ast=(t+\sqrt{t^2+ab})/2>\max(0,t)$ and
$$
h_t(s^\ast)
= \sqrt{t^2+ab} -t\big(\log(\sqrt{t^2+ab}+t)-\log(b) \big) -\frac{a+b}{2} .
$$
By Lemma~\ref{lm:bmd} in Appendix~\ref{ap:6}, both $|s|$ and $|t|$ are upper bounded by some (large) constants
with probability $1-o(n^{-10})$. Therefore, the sum on the right-hand side of \eqref{eq:gj} is concentrated around a small neighborhood of $s^\ast$. In this neighborhood, we have
$$
\frac{1} {2\pi \log(n)\sqrt{s (s-t)}} = \Theta(\frac{1} {\log(n)}) .
$$
Therefore, in order to prove this proposition, we only need to show that
\begin{equation} \label{eq:jh}
\sum \exp(h_t(s) \log(n))
= \Theta\Big(\sqrt{\log(n)} \exp(h_t(s^\ast) \log(n))\Big) ,
\end{equation}
where the summation is taken over this small neighborhood.
We will show that $\exp(h_t(s) \log(n))$ varies by a constant factor within a window of length $\Theta(1/\sqrt{\log(n)})$ around $s^\ast$, and
then drops off geometrically fast beyond that window. First observe that $h_t(s)\approx h_t(s^\ast) - h_t''(s^\ast) (s-s^\ast)^2$ in the neighborhood of $s^\ast$, so when $|s-s^\ast|=\Theta(1/\sqrt{\log(n)})$, we have $h_t(s^\ast) \log(n) - h_t(s) \log(n) = \Theta(1)$.
Also note that when $s$ is in the range $(s^\ast-\Theta(1/\sqrt{\log(n)}) , s^\ast+\Theta(1/\sqrt{\log(n)}))$, the quantity $s\log(n)$ takes $\Theta(\sqrt{\log(n)})$ integer values.
Now pick some constant $c>0$. By the above analysis we have
$$
\sum_{s^\ast\log(n) - c\sqrt{\log(n)} \le s\log(n) \le s^\ast\log(n) + c\sqrt{\log(n)}}
\exp(h_t(s) \log(n))
= \Theta\Big(\sqrt{\log(n)} \exp(h_t(s^\ast) \log(n))\Big) .
$$
For $s>s^\ast+c/\sqrt{\log(n)}$, we use the fact that concave functions are always bounded from above by its tangent lines. Therefore,
\begin{align*}
h_t(s) & \le h_t(s^\ast+c/\sqrt{\log(n)})
+ h_t'(s^\ast+c/\sqrt{\log(n)}) 
(s- s^\ast - c/\sqrt{\log(n)}) \\
& \le h_t(s^\ast) - \frac{c'}{\sqrt{\log(n)}}
(s- s^\ast - c/\sqrt{\log(n)}) ,
\end{align*}
where we use the fact that $h_t'(s^\ast+c/\sqrt{\log(n)}) = \Theta(1/\sqrt{\log(n)})$, and $c'$ is another constant that depends on $c$.
Therefore,
\begin{align*}
& \sum_{s\log(n) > s^\ast\log(n) + c\sqrt{\log(n)}}
\exp(h_t(s) \log(n))  \\
\le & \sum_{s\log(n) > s^\ast\log(n) + c\sqrt{\log(n)}}
\exp \Big( h_t(s^\ast) \log(n) 
-\frac{c'}{\sqrt{\log(n)}}
(s\log(n)- s^\ast \log(n)- c \sqrt{\log(n)}) \Big)   \\
= & \exp(h_t(s^\ast) \log(n))
\sum_{j>0} \exp \Big( 
-\frac{c'}{\sqrt{\log(n)}} j \Big) \\
\le & \exp(h_t(s^\ast) \log(n))
\frac{1}{1-\exp(-c'/ \sqrt{\log(n)})} \\
= & O\Big(\sqrt{\log(n)} \exp(h_t(s^\ast) \log(n))\Big) .
\end{align*}
The sum over $s\log(n) < s^\ast\log(n) - c\sqrt{\log(n)}$ can be bounded in the same way. Thus we have shown \eqref{eq:jh}, and this completes the proof of the proposition.
\end{proof}

Let $\cG_1:=\{G:B_i-A_i<0\text{~for all~}i\in[n]\}$. By \eqref{eq:tD}, we have $P(G\in\cG_1)=1-o(1)$. In the proposition below, we will prove that if $0<\beta<\frac{1}{4}\log\frac{a}{b}$, then the conditional expectation $E \Big[ \sum_{i=1}^n  \exp\big(2\beta (B_i-A_i) \big) ~\Big|~ G\in\cG_1 \Big]$ is very close to the unconditional expectation $E \Big[ \sum_{i=1}^n  \exp\big(2\beta (B_i-A_i) \big) \Big]$. On the other hand, if $\beta\ge\frac{1}{4}\log\frac{a}{b}$, then the conditional expectation is $O(n^{\tilde{g}(\beta)})$ while the unconditional one is $\Theta(n^{g(\beta)})$. Since $\tilde{g}(\beta)<g(\beta)$ when $\beta>\frac{1}{4}\log\frac{a}{b}$, the conditional expectation is much smaller than the unconditional one in this case.
\begin{proposition}  \label{prop:df}
Assume that $\sqrt{a}-\sqrt{b}>\sqrt{2}$.
If $0<\beta<\frac{1}{4}\log\frac{a}{b}$, then
$$
E \Big[ \sum_{i=1}^n  \exp\big(2\beta (B_i-A_i) \big) ~\Big|~ G\in\cG_1 \Big] 
= (1+o(1)) E \Big[ \sum_{i=1}^n  \exp\big(2\beta (B_i-A_i) \big) \Big]
= (1+o(1)) n^{g(\beta)}  .
$$
If $\beta\ge\frac{1}{4}\log\frac{a}{b}$, then
\begin{align*}
E \Big[ \sum_{i=1}^n  \exp\big(2\beta (B_i-A_i) \big) ~\Big|~ G\in\cG_1 \Big] 
= O(n^{\tilde{g}(\beta)})  .
\end{align*}
\end{proposition}
\begin{proof}
We first calculate the unconditional expectation:
By writing $A_i$ and $B_i$ as sums of independent Bernoulli random variables, we obtain that
\begin{align*}
E[e^{2\beta(B_i-A_i)}]
& =\Big(1-\frac{b\log(n)}{n}+\frac{b\log(n)}{n} e^{2\beta} \Big)^{n/2}
\Big(1-\frac{a\log(n)}{n}+\frac{a\log(n)}{n} e^{-2\beta} \Big)^{n/2-1}  \\
& = 
\exp\Big(\frac{\log(n)}{2} ( a e^{-2\beta}+b e^{2\beta} -a-b )
+o(1) \Big) \\
& = (1+o(1)) n^{g(\beta)-1} .
\end{align*}
Therefore, for all $\beta$ we have
\begin{equation} \label{eq:lb}
E \Big[ \sum_{i=1}^n  \exp\big(2\beta (B_i-A_i) \big) \Big]
= (1+o(1)) n^{g(\beta)}  .
\end{equation}
Now let us switch to conditional expectation. In light of \eqref{eq:lb}, for the case of $0<\beta<\frac{1}{4}\log\frac{a}{b}$ we only need to prove that the conditional expectation is very close to the unconditional expectation.
To that end, we first reprove a weaker version of \eqref{eq:lb} using Proposition~\ref{prop:99}. This will help us estimate the difference between the conditional and unconditional expectations.

Define $D(G,t):=|\{i\in[n]:B_i-A_i= t\log(n)\}|$.
Similarly to \eqref{eq:fd}, we have
\begin{equation} \label{eq:s1}
\sum_{i=1}^n \exp\big(2\beta (B_i-A_i) \big) 
=  \sum_{t\log(n)=-n/2}^{n/2}
D(G,t) \exp\big(2\beta t \log(n) \big) .
\end{equation}
By definition, $D(G,t)=\sum_{i=1}^n \mathbbm{1}[B_i-A_i= t\log(n)]$, so
$E[D(G,t)]=n P(B_i-A_i= t\log(n))$. Therefore, by Proposition~\ref{prop:99} and the definition of function $f_{\beta}(t)$ in \eqref{eq:gt},
$$
E[D(G,t)
\exp\big(2\beta t \log(n) \big)]
\asymp \frac{1} {\sqrt{\log(n)}} \exp( f_{\beta}(t) \log(n) ) .
$$
As a consequence,
\begin{equation}  \label{eq:wf1}
\begin{aligned}
E \Big[ \sum_{i=1}^n  \exp\big(2\beta (B_i-A_i) \big) \Big]     
& =  \sum_{t\log(n)=-n/2}^{n/2}
E \big[ D(G,t) \exp\big(2\beta t \log(n) \big) \big]  \\
& \asymp  \frac{1} {\sqrt{\log(n)}} \sum_{t\log(n)=-n/2}^{n/2}
 \exp( f_{\beta}(t) \log(n) ) .
\end{aligned}
\end{equation}
By the proof of Lemma~\ref{lm:tus}, $f_{\beta}(t)$ is a concave function and takes maximum at
$t^\ast=\frac{b e^{2\beta}-a e^{-2\beta}}{2}$. 
Similarly to the analysis of \eqref{eq:jh}, $\exp( f_{\beta}(t) \log(n))$ varies by a constant factor within a window of length $\Theta(1/\sqrt{\log(n)})$ around $t^\ast$, and
then drops off geometrically fast beyond that window. Since $t\log(n)$ takes $\Theta(\sqrt{\log(n)})$ integer values when $t$ takes values in such a window, we have
\begin{equation} \label{eq:wf2}
\sum_{t\log(n)=-n/2}^{n/2}
 \exp( f_{\beta}(t) \log(n) )
 \asymp \sqrt{\log(n)}
 \exp( f_{\beta}(t^\ast) \log(n) ) .
\end{equation}
By the proof of Lemma~\ref{lm:tus}, we have $f_{\beta}(t^\ast)=g(\beta)$. Taking this into \eqref{eq:wf1} and \eqref{eq:wf2}, we obtain that
$$
E \Big[ \sum_{i=1}^n  \exp\big(2\beta (B_i-A_i) \big) \Big]
= \Theta (n^{g(\beta)}) .
$$
Now let us consider the conditional expectation. When conditioning on the event $G\in\cG_1$, we have $D(G,t)=0$ for all $t\ge 0$. In this case, the range of sum in both \eqref{eq:s1} and \eqref{eq:wf1} reduces from $[-n/2,n/2]$ to $[-n/2,0)$. By Lemma~\ref{lm:5t} below, we have $P(B_i-A_i= t\log(n)~|~G\in\cG_1)= (1+o(1))P(B_i-A_i= t\log(n))$ for $t<0$, and so
$E[D(G,t)|G\in\cG_1]=(1+o(1))E[D(G,t)]$ for $t<0$. Therefore,
$$
E\Big[ \sum_{i=1}^n \exp\big(2\beta (B_i-A_i) \big) \Big| G\in\cG_1 \Big]
= (1+o(1)) \sum_{t\log(n)=-n/2}^{-1}
E\big[ D(G,t) \exp\big(2\beta t \log(n) \big) \big] .
$$
From the analysis of \eqref{eq:wf1}, we know that 
\begin{align*}
& \sum_{t\log(n)=-n/2}^{n/2}
E\big[ D(G,t) \exp\big(2\beta t \log(n) \big) \big]  \\
= & (1+o(1)) \sum_{t\log(n)=t^\ast\log(n)-\Theta(\sqrt{\log(n)})}^{t^\ast\log(n)+\Theta(\sqrt{\log(n)})}
E\big[ D(G,t) \exp\big(2\beta t \log(n) \big) \big] .
\end{align*}
Therefore, if $t^\ast=\frac{b e^{2\beta}-a e^{-2\beta}}{2}<0$, or equivalently $0<\beta<\frac{1}{4}\log\frac{a}{b}$, then
$$
\sum_{t\log(n)=-n/2}^{n/2}
E\big[ D(G,t) \exp\big(2\beta t \log(n) \big) \big] = (1+o(1))
\sum_{t\log(n)=-n/2}^{-1}
E\big[ D(G,t) \exp\big(2\beta t \log(n) \big) \big] ,
$$
i.e.,
\begin{equation} \label{eq:bz}
E \Big[ \sum_{i=1}^n  \exp\big(2\beta (B_i-A_i) \big) ~\Big|~ G\in\cG_1 \Big] 
= (1+o(1)) E \Big[ \sum_{i=1}^n  \exp\big(2\beta (B_i-A_i) \big) \Big] .
\end{equation}
On the other hand,
if $t^\ast=\frac{b e^{2\beta}-a e^{-2\beta}}{2}\ge 0$, or equivalently $\beta\ge\frac{1}{4}\log\frac{a}{b}$, then
\begin{equation}  \label{eq:wq}
E \Big[ \sum_{i=1}^n  \exp\big(2\beta (B_i-A_i) \big) ~\Big|~ G\in\cG_1 \Big]   
\asymp  \frac{1} {\sqrt{\log(n)}} \sum_{t\log(n)=-n/2}^{-1}
 \exp( f_{\beta}(t) \log(n) ) .
\end{equation}
Since $f_{\beta}(t)$ is concave, it is an increasing function when $t<t^\ast$. Therefore, $f_{\beta}(t)<f_{\beta}(0)=g(\frac{1}{4}\log\frac{a}{b})=\tilde{g}(\beta)$ for $\beta\ge\frac{1}{4}\log\frac{a}{b}$.
Similarly to the analysis of \eqref{eq:jh} and \eqref{eq:wf1},
$\exp( f_{\beta}(t) \log(n))$ varies by a constant factor within a window of length $O(1/\sqrt{\log(n)})$ around $t=0$, and
then drops off geometrically fast beyond that window. As a consequence,
\begin{align*}
\sum_{t\log(n)=-n/2}^{-1}
 \exp( f_{\beta}(t) \log(n) )
=   O(\sqrt{\log(n)}) \exp(\tilde{g}(\beta) \log(n)) .
\end{align*}
Taking this into \eqref{eq:wq} completes the proof of the proposition.
\end{proof}

The following corollary follows immediately from Proposition~\ref{prop:df} and \eqref{eq:lb}:
\begin{corollary} \label{cr:yy}
For all $\beta>0$,
\begin{align*}
& E \Big[ \sum_{i=1}^n  \exp\big(2\beta (B_i-A_i) \big)  \Big] 
= \Theta(n^{g(\beta)}),  \\
& E \Big[ \sum_{i=1}^n  \exp\big(2\beta (B_i-A_i) \big) ~\Big|~ G\in\cG_1 \Big] 
= O(n^{\tilde{g}(\beta)})  .
\end{align*}
\end{corollary}

\begin{remark}
From the proof of \eqref{eq:bz}, we can see that if $0<\beta<\frac{1}{4}\log\frac{a}{b}$, then
\begin{equation} \label{eq:pl}
E \big[  \exp\big(2\beta (B_i-A_i) \big) ~\big|~ G\in\cG_1 \big] 
= (1+o(1)) E \big[  \exp\big(2\beta (B_i-A_i) \big) \big] ,
\end{equation}
i.e., we can remove the summation in \eqref{eq:bz}. We will use this in the proof of Proposition~\ref{prop:con}.
\end{remark}

\begin{lemma}  \label{lm:5t}
Let $\cG_1:=\{G:B_i-A_i<0\text{~for all~}i\in[n]\}$. Then $P(B_i-A_i= t\log(n)~|~G\in\cG_1)= (1+o(1))P(B_i-A_i= t\log(n))$ for all $t<0$ such that $t\log(n)$ is an integer.
\end{lemma}

\begin{proof}
Note that 
$$
P(B_i-A_i= t\log(n)~|~G\in\cG_1)
=\frac{P \big(B_j-A_j<0\text{~for all~} j\in[n]\setminus\{i\}, B_i-A_i= t\log(n) \big)}{P(G\in\cG_1)} .
$$
By \eqref{eq:tD}, we have $P(G\in\cG_1)=1-o(1)$, so
\begin{align*}
& P(B_i-A_i= t\log(n)~|~G\in\cG_1) \\
= & (1+o(1)) P \big(B_j-A_j<0\text{~for all~} j\in[n]\setminus\{i\}, B_i-A_i= t\log(n) \big) \\
= & (1+o(1)) P \big(B_j-A_j<0\text{~for all~} j\in[n]\setminus\{i\} ~\big|~ B_i-A_i= t\log(n) \big) P(B_i-A_i= t\log(n)) .
\end{align*}
Therefore, to prove the lemma we only need to show that $P \big(B_j-A_j<0\text{~for all~} j\in[n]\setminus\{i\} ~\big|~ B_i-A_i= t\log(n) \big)=1-o(1)$. 

For $j\in[n]\setminus\{i\}$, define $\xi_{ij}=\xi_{ij}(G):=\mathbbm{1}[\{i,j\}\in E(G)]$ as the indicator function of the edge $\{i,j\}$ connected in graph $G$. We also define 
$$
B_j'=\left\{
\begin{array}{cl}
  B_j-\xi_{ij}   & \mbox{if~} X_i\neq X_j \\
  B_j   &   \mbox{if~} X_i= X_j
\end{array}
\right.
\quad \text{and} \quad
A_j'=\left\{
\begin{array}{cl}
  A_j   & \mbox{if~} X_i\neq X_j \\
  A_j-\xi_{ij}   &   \mbox{if~} X_i= X_j
\end{array}
\right. .
$$
Then $B_j'-A_j'$ differs from $B_j-A_j$ by at most $1$. Therefore, $B_j'-A_j'<-1$ implies that $B_j-A_j<0$, and so $P \big(B_j-A_j<0\text{~for all~} j\in[n]\setminus\{i\} ~\big|~ B_i-A_i= t\log(n) \big) \ge P \big(B_j'-A_j'<-1 \text{~for all~} j\in[n]\setminus\{i\} ~\big|~ B_i-A_i= t\log(n) \big)$. Now we only need to prove that the right-hand side is $1-o(1)$. Also note that the two sets of random variables $\{B_j',A_j':j\in[n]\setminus\{i\}\}$ and $\{B_i,A_i\}$ are independent, so $P \big(B_j'-A_j'<-1 \text{~for all~} j\in[n]\setminus\{i\} ~\big|~ B_i-A_i= t\log(n) \big)=P \big(B_j'-A_j'<-1 \text{~for all~} j\in[n]\setminus\{i\}  \big)$. By definition, we have $B_j'\sim\Binom(n/2-\Theta(1),b\log(n)/n)$ and $A_j'\sim\Binom(n/2-\Theta(1),a\log(n)/n)$ for all $j\in[n]\setminus\{i\}$. Then following exactly the same proof\footnote{First use Chernoff bound as we did in Proposition~\ref{prop:cher} and then use the union bound.} as that of \eqref{eq:tD}, we have
$$
P \big(B_j'-A_j'<-1 \text{~for all~} j\in[n]\setminus\{i\}  \big)
\ge 1- n^{1-\frac{(\sqrt{a}-\sqrt{b})^2}{2} +o(1)} = 1-o(1).
$$
This completes the proof of the lemma.
\end{proof}

\newpage 

\setlength{\bibsep}{6pt}

{\small
\bibliographystyle{alpha}

\bibliography{SIBM}

\begin{thebibliography}{VMLC16}

\bibitem[Abb17]{Abbe17}
E.~Abbe.
\newblock Community detection and stochastic block models: {Recent}
  developments.
\newblock {\em The Journal of Machine Learning Research}, 18(1):6446--6531,
  2017.

\bibitem[ABH16]{ABH16}
E.~{Abbe}, A.~S. {Bandeira}, and G.~{Hall}.
\newblock Exact recovery in the stochastic block model.
\newblock {\em IEEE Transactions on Information Theory}, 62(1):471--487, Jan
  2016.

\bibitem[AS15]{Abbe15}
E.~Abbe and C.~Sandon.
\newblock Community detection in general stochastic block models: {Fundamental}
  limits and efficient algorithms for recovery.
\newblock In {\em 2015 IEEE 56th Annual Symposium on Foundations of Computer
  Science}, pages 670--688. IEEE, 2015.

\bibitem[BC09]{Bickel09}
P.~J. Bickel and A.~Chen.
\newblock A nonparametric view of network models and {Newman--Girvan} and other
  modularities.
\newblock {\em Proceedings of the National Academy of Sciences},
  106(50):21068--21073, 2009.

\bibitem[BCLS87]{Bui87}
T.~N. Bui, S.~Chaudhuri, F.~T. Leighton, and M.~Sipser.
\newblock Graph bisection algorithms with good average case behavior.
\newblock {\em Combinatorica}, 7(2):171--191, 1987.

\bibitem[BGd08]{Banerjee08}
O.~Banerjee, L.~El Ghaoui, and A.~d’Aspremont.
\newblock Model selection through sparse maximum likelihood estimation for
  multivariate gaussian or binary data.
\newblock {\em Journal of Machine learning research}, 9(Mar):485--516, 2008.

\bibitem[BMS08]{Bresler08}
G.~Bresler, E.~Mossel, and A.~Sly.
\newblock Reconstruction of {Markov} random fields from samples: {Some}
  observations and algorithms.
\newblock In {\em Approximation, Randomization and Combinatorial Optimization.
  Algorithms and Techniques}, pages 343--356. Springer, 2008.

\bibitem[Bop87]{Boppana87}
R.~B. Boppana.
\newblock Eigenvalues and graph bisection: {An} average-case analysis.
\newblock In {\em 28th Annual Symposium on Foundations of Computer Science
  (sfcs 1987)}, pages 280--285. IEEE, 1987.

\bibitem[Bre15]{Bresler15}
G.~Bresler.
\newblock Efficiently learning {Ising} models on arbitrary graphs.
\newblock In {\em Proceedings of the forty-seventh annual ACM symposium on
  Theory of computing}, pages 771--782, 2015.

\bibitem[BRS19]{Berthet19}
Q.~Berthet, P.~Rigollet, and P.~Srivastava.
\newblock Exact recovery in the {Ising} blockmodel.
\newblock {\em The Annals of Statistics}, 47(4):1805--1834, 2019.

\bibitem[CK01]{Condon01}
A.~Condon and R.~M. Karp.
\newblock Algorithms for graph partitioning on the planted partition model.
\newblock {\em Random Structures \& Algorithms}, 18(2):116--140, 2001.

\bibitem[CWA12]{Choi12}
D.~S. Choi, P.~J. Wolfe, and E.~M. Airoldi.
\newblock Stochastic blockmodels with a growing number of classes.
\newblock {\em Biometrika}, 99(2):273--284, 2012.

\bibitem[CX16]{Chen16}
Y.~Chen and J.~Xu.
\newblock Statistical-computational tradeoffs in planted problems and submatrix
  localization with a growing number of clusters and submatrices.
\newblock {\em The Journal of Machine Learning Research}, 17(1):882--938, 2016.

\bibitem[DF89]{Dyer89}
M.~E. Dyer and A.~M. Frieze.
\newblock The solution of some random {NP}-hard problems in polynomial expected
  time.
\newblock {\em Journal of Algorithms}, 10(4):451--489, 1989.

\bibitem[FO05]{Feige05}
U.~Feige and E.~Ofek.
\newblock Spectral techniques applied to sparse random graphs.
\newblock {\em Random Structures \& Algorithms}, 27(2):251--275, 2005.

\bibitem[HKM17]{Hamilton17}
L.~Hamilton, F.~Koehler, and A.~Moitra.
\newblock Information theoretic properties of {Markov} random fields, and their
  algorithmic applications.
\newblock In {\em Advances in Neural Information Processing Systems}, pages
  2463--2472, 2017.

\bibitem[HWX16]{Hajek16}
B.~Hajek, Y.~Wu, and J.~Xu.
\newblock Achieving exact cluster recovery threshold via semidefinite
  programming.
\newblock {\em IEEE Transactions on Information Theory}, 62(5):2788--2797,
  2016.

\bibitem[Isi25]{Ising25}
E.~Ising.
\newblock Beitrag zur theorie des ferromagnetismus.
\newblock {\em Zeitschrift f{\"u}r Physik}, 31(1):253--258, 1925.

\bibitem[KM17]{Klivans17}
A.~Klivans and R.~Meka.
\newblock Learning graphical models using multiplicative weights.
\newblock In {\em 2017 IEEE 58th Annual Symposium on Foundations of Computer
  Science (FOCS)}, pages 343--354. IEEE, 2017.

\bibitem[LR15]{Lei15}
J.~Lei and A.~Rinaldo.
\newblock Consistency of spectral clustering in stochastic block models.
\newblock {\em The Annals of Statistics}, 43(1):215--237, 2015.

\bibitem[McS01]{Mcsherry01}
F.~McSherry.
\newblock Spectral partitioning of random graphs.
\newblock In {\em Proceedings 42nd IEEE Symposium on Foundations of Computer
  Science}, pages 529--537. IEEE, 2001.

\bibitem[MNS16]{MNS16}
E.~Mossel, J.~Neeman, and A.~Sly.
\newblock Consistency thresholds for the planted bisection model.
\newblock {\em Electronic Journal of Probability}, 21, 2016.

\bibitem[Pei19]{Peixoto19}
T.~P. Peixoto.
\newblock Network reconstruction and community detection from dynamics.
\newblock {\em Physical review letters}, 123(12):128301, 2019.

\bibitem[RWL10]{Ravikumar10}
P.~Ravikumar, M.~J. Wainwright, and J.~D. Lafferty.
\newblock High-dimensional {Ising} model selection using $\ell_1$-regularized
  logistic regression.
\newblock {\em The Annals of Statistics}, 38(3):1287--1319, 2010.

\bibitem[SN97]{Snijders97}
T.~A.~B. Snijders and K.~Nowicki.
\newblock Estimation and prediction for stochastic blockmodels for graphs with
  latent block structure.
\newblock {\em Journal of classification}, 14(1):75--100, 1997.

\bibitem[SW12]{Santhanam12}
N.~P. Santhanam and M.~J. Wainwright.
\newblock Information-theoretic limits of selecting binary graphical models in
  high dimensions.
\newblock {\em IEEE Transactions on Information Theory}, 58(7):4117--4134,
  2012.

\bibitem[VMLC16]{Vuffray16}
M.~Vuffray, S.~Misra, A.~Lokhov, and M.~Chertkov.
\newblock Interaction screening: {Efficient} and sample-optimal learning of
  {Ising} models.
\newblock In {\em Advances in Neural Information Processing Systems}, pages
  2595--2603, 2016.

\bibitem[Vu18]{Vu18}
V.~Vu.
\newblock A simple {SVD} algorithm for finding hidden partitions.
\newblock {\em Combinatorics, Probability and Computing}, 27(1):124--140, 2018.

\bibitem[WSD19]{Wu19}
S.~Wu, S.~Sanghavi, and A.~G. Dimakis.
\newblock Sparse logistic regression learns all discrete pairwise graphical
  models.
\newblock In {\em Advances in Neural Information Processing Systems}, pages
  8069--8079, 2019.

\end{thebibliography}
}
\end{document}